\documentclass[11pt,a4paper]{article}

\usepackage{latexsym,amsfonts,amssymb,setspace,galois,amsmath,color,mathrsfs, amstext,bm,bookmark,lineno,enumerate,lipsum,multicol}
\usepackage{amsbsy, amsopn, amscd, amsxtra, amsthm,authblk,indentfirst}
\usepackage{graphicx,epsfig,subfigure,caption}
\usepackage{hyperref,cleveref,upref}
\usepackage{geometry}
\usepackage[authoryear,numbers,round,sort&compress,colon,super]{natbib}
\bibpunct{[}{]}{;}{n}{}{,}
\usepackage{float}
\usepackage{yhmath}
\usepackage{microtype}

\numberwithin{equation}{section}              % numerazionedelleequazioni
\newtheorem{theorem}{Theorem}[section]
\newtheorem{lemma}{Lemma}[section]
\newtheorem{proposition}{Proposition}[section]
\newtheorem*{proposition*}{Proposition}

\newtheorem*{corollary*}{Corollary}

\newtheorem*{definitions*}{Definitions}
\newtheorem*{conjecture*}{\bf Conjecture}
\newtheorem{example}{\bf Example}[section]
\newtheorem*{example*}{\bf Example}
\theoremstyle{remark}
\newtheorem{remark}{\bf Remark}[section]

\makeatletter

\newcommand\figurecaption{\def\@captype{figure}\caption}
\newcommand\tablecaption{\def\@captype{table}\caption}
\makeatother
\date{}                                     %TITOLO\date\today
\title{\bf{Turing-Hopf bifurcation and spatio-temporal patterns of a ratio-dependent Holling-Tanner system with diffusion}\thanks {Supported by the National Natural Science Foundation of China (No.11371112).}}

\author[a]{Qi An}
\author[a]{{ Weihua Jiang\thanks {Corresponding author.
			E-mail address: jiangwh@hit.edu.cn}}}
\affil[a]{Department of Mathematics, Harbin Institute of Technology, Harbin, 150001, P.R. China.}
\begin{document}	
	\maketitle
%	\hrule
	%******TEXTE***************************************************
	\begin{abstract}
		A diffusive ratio-dependent Holling-Tanner system subject to Neumann boundary conditions is considered. The existence of multiple bifurcations, including Turing-Hopf bifurcation, Turing-Truing bifurcation, Hopf-double-Turing bifurcation and triple-Turing bifurcation, are given. Among them, the Turing-Hopf bifurcation are carried out in details by the normal form method. We theoretically prove that the system exists various spatio-temporal patterns, such as, non-constant steady state, the spatially inhomogeneous periodic or quasi-periodic solution, etc. Numerical simulations are presented to illustrate our theoretical results.
	\end{abstract}
%	\hrule
\section{Introduction}
For a long time, the predator-prey models have received extensive concerns from both mathematicians and biologists. The Lotka-Volterra model is one of the most classical models and was first put forward in the 1920s. With the deepening of research, this simplest ecological model is questioned because of its irrational assumptions and inaccurate predictions.
In the 1960s, May \cite{May1973Stability} first make two adjustments to it:  addition the self-regulation of prey and the incorporation of a Holling type \uppercase\expandafter{\romannumeral2 } functional response function \cite{holling1959components}. This model is also known as the Holling-Tanner prey-predator model \cite{tanner1975stability} and has the form 
 \begin{equation}\label{eqA1}
 \left\{
 \begin{aligned}
 &\frac{\mathrm{d}u}{\mathrm{d}t}=r_{1}u(1-\frac{u}{k})-\frac{quv}{u+m},&\\
 &\frac{\mathrm{d}v}{\mathrm{d}t}=r_{2}v(1-\frac{v}{\gamma u}),&
 \end{aligned}
 \right.
 \end{equation}
Here $u(t)$, $v(t)$ represent the densities of prey and predator, respectively. In addition, the parameters have the following meanings:
\begin{itemize}
	\item $r_1$, $r_2$ are the intrinsic growth rates of the prey and predator, respectively.
	\item $k$ is the carrying capacity of the prey, and $\gamma u$ play the role as the prey-dependent carrying capacity of the predator. 
	$\gamma$ is the conversion rate of prey to predator birth, and can be seen as a measure of the quality of the prey as food. 
	\item $q$ is the maximum value of prey consumed by per predator per unit time.
	\item $m$ is a saturation value. It is the value of prey required to reach half of the maximum rate of $q$.
\end{itemize}
The system \eqref{eqA1} is regarded as one of the prototypical predator–prey models has been extensively studied. A lot of interesting questions, such as the equivalence between local and global stability, collapse of two limit cycles, the uniqueness of the limit cycle and bifurcations,  have been solved \cite{saez1999dynamics,hsu1995global,hsu1998uniqueness,wollkind1988metastability,braza2003bifurcation}. 

Further consider the influence of diffusion to \eqref{eqA1}, the two species may exhibit inhomogeneous distribution in a spatial domain $\Omega \in \mathbb{R}^n$. Therefore, we should consider the following reaction-diffusion system,
 \begin{equation}\label{eqA2}
\left\{
\begin{aligned}
&\frac{\mathrm{d}u}{\mathrm{d}t}-D_{1}\Delta u=r_{1}u(1-\frac{u}{k})-\frac{quv}{u+m},&&x\in\Omega, t>0,&\\
&\frac{\mathrm{d}v}{\mathrm{d}t}-D_{2}\Delta v=r_{2}v(1-\frac{v}{\gamma u}),& &x\in\Omega, t>0,&\\
&\partial_{\eta}u=\partial_{\eta}v=0,&&x\in\partial\Omega,{t}>0,&\\
&u(x,0)=u_0(x),v(x,0)=v_0(x),&&x\in\Omega,&\\
\end{aligned}
\right.
\end{equation}
where $\eta$ is the outward unit normal vector on  $\partial\Omega$, and $D_1,D_2$ are the diffusion coefficients of prey and predator, respectively.
 The no-flux boundary condition means that the system is self-contained and closed to the exterior environment. For this diffusion model, 
  Peng and Wang in \cite{peng2005positive} investigated the existence and non-existence of the non-constant steady state solutions. Furthermore, the parameter conditions for global stability of positive constant steady state are given in \cite{peng2007global,chen2012global,qi2016study,shi2010positive}, respectively.
 Li et al. \cite{li2011hopf} considered the Turing and Hopf bifurcations. Related work on the modified system \eqref{eqA2} can also be found in \cite{peng2008stationary,huang2014bifurcations,du2004diffusive}.

With a  non-dimensionalized change of variables:
\begin{equation*}
\begin{aligned}
&u\rightarrow \frac{u}{k},& &v\rightarrow \frac{v}{\gamma k},& &{t}\rightarrow r_{1}t.&
\end{aligned}
\end{equation*}
and let
\begin{equation*}
\begin{aligned}
&d_{1}=\frac{D_{1}}{ r_{1}},& &{d}_{2}=\frac{D_{2}}{ r_{1}},& &a=\frac{q\gamma}{r_{1}},& &r=\frac{r_{2}}{r_{1}},& &b=\frac{m}{k}<1.&
\end{aligned}
\end{equation*}
We obtain the simplified dimensionless ratio-dependent Holling-Tanner system with diffusion
\begin{equation}\label{eqA3}
\left\{
\begin{aligned}
&\frac{\mathrm{d}}{\mathrm{d}t}u-d_{1}\Delta u=u(1-u)-\frac{auv}{u+b},&&x\in\Omega,~t>0,&\\
&\frac{\mathrm{d}}{\mathrm{d}t}v-d_{2}\Delta v=rv(1-\frac{v}{u}),& &x\in\Omega,~t>0,&\\
&\partial_{\eta}u=\partial_{\eta}v=0,& &x\in\partial\Omega,~t>0,&\\
&u(x,0)=u_0(x), v(x,0)=v_0(x),& &x\in\Omega.&\\
\end{aligned}\right.
\end{equation}
For system \eqref{eqA3}, Ma and Li \cite{ma2013bifurcation} studied the Hopf bifurcation and the steady state bifurcation of simple and double eigenvalues. Banerjee, M. and Banerjee, S. \cite{banerjee2012turing} investigated  the Turing and non-Turing patterns with $\Omega$ is a two-dimensional bounded connected square domain. The formation of various spatio-temporal patterns have been extensively studied in recent years \cite{baurmann2007instabilities,rovinsky1992interaction,song2016turing,yang2016spatial,Yi2010Spatiotemporal,yi2009bifurcation,Su2009Hopf}.  
 And Turing-Hopf bifurcation can be regarded as one of the important mechanisms to generate the spatio-temporal patterns. Study the Turing-Hopf bifurcation of the predator-prey system can helps to understand more ecological phenomenas. Therefore, we will study this problem in this Holling-Tanner system.

For convenience, we consider the spatial domain $\Omega = (0,l\pi)$ with $l\in \mathbb{R}^+$, 
\begin{equation}\label{eqA}
\left\{
\begin{aligned}
&\frac{\mathrm{d}}{\mathrm{d}t}u-d_{1}\Delta u=u(1-u)-\frac{auv}{u+b},&&x\in(0,l\pi),~t>0,&\\
&\frac{\mathrm{d}}{\mathrm{d}t}v-d_{2}\Delta v=rv(1-\frac{v}{u}),& &x\in(0,l\pi),~t>0,&\\
&u_x(0,t)=v_x(0,t)=0, \;\;u_x(i\pi,t)=v_x(l\pi,t)=0,& &t>0,&\\
&u(x,0)=u_0(x), v(x,0)=v_0(x),& &x\in(0,l\pi).&\\
\end{aligned}\right.
\end{equation}
Choosing the birth ratio $r$ and the domain size $l$ as the main  bifurcation parameters to consider the Turing-Hopf bifurcation, we show that the system \eqref{eqA} exhibits a variety of spatio-temporal patterns. Among them, the existence of the spatially inhomogeneous quasi-periodic orbits is proved first time in both theoretically and numerically, to the best of our knowledge. We point out that our results are follow the algorithm in \cite{An2017}, which is mainly based on the central manifold theorem \cite{Lin1992Centre} and the normal form theory \cite{Faria2000Normal}. 

The paper is organized as follows. In Section 2, we devote to the bifurcation analysis of the ratio-dependent Holling-Tanner system \eqref{eqA}. The conditions of the  existence of Hopf bifurcation, steady state bifurcation, Turing-Hopf bifurcating, Bogdanov- Tankens bifurcation, Hopf-double-zero bifurcation and Triple-zero bifurcation are obtained. In Section 3, we give the detailed dynamics of the \eqref{eqA} with the parameter near the Turing-Hopf singularity. Finally a conclusion section complete the paper.

\section{Bifurcation analysis of the ratio-dependent Holling-Tanner system}

The system \eqref{eqA} has two non-negative constant steady states: $(1,0)$ and $(u_{0},v_{0})$, where
$$u_{0}=v_{0}=\dfrac{1}{2}[(1-a-b)+\sqrt{(a+b-1)^{2}+4b}]<1,$$
satisfies $(u_{0}-1)(u_{0}+b)+au_{0}=0.$
Among them, $(1,0)$ is always an unstable equilibrium point. In this section, we will mainly study the effect of the birth ratio $r$ and the domain size $l$ to the dynamics of system \eqref{eqA} near the coexistence equilibrium point $(u_{0},v_{0})$.

Define the real-valued phase space $$X:=\{(u,v)\in{H^{2}(0,\pi)\times H^{2}(0,\pi)} : (u_{x},v_{x})|_{x=0,\pi}=0\},$$ 
and the corresponding complex phase space 
$X_{\mathbb{C}}:=\{ x_{1}+ix_{2} : x_{1},x_{2}\in X \},$ 

By the translation $\hat{u}=u-u_{0}$, $\hat{v}=v-v_{0}$ and the space scale $x \rightarrow {x}/{l}$, the system \eqref{eqA} can be written as an abstract equation in phase space $X_{\mathbb{C}}$, 
\begin{equation}\label{eqB}
\frac{\mathrm{d}}{\mathrm{d}t}U =D(r,l)\Delta U +L(r,l)U+F(r,l,U).
\end{equation}
Here $U=(\hat{u},\hat{v})^{\mathrm{T}}\in X_{\mathbb{C}}$,
 $D(r,l)=\dfrac{1}{l^2}\,\mathrm{diag}(d_1,d_2)$,
%$D(r,l)=\dfrac{1}{l^2}\left(\begin{array}{cc}
%d_1& 0\\0 & d_2 \end{array}\right),$ 
$L(r,l):X_{\mathbb{C}}\rightarrow X_{\mathbb{C}} $ is the linearized operator given by 
\begin{equation}\label{eqL}
L(r,l)={\left(\begin{array}{cc}
	A_0 &B_0\\
	r & -r
	\end{array}\right),}
\end{equation}
with
% $A_0=1-2u_{0}-\dfrac{abu_{0}}{(b+u_{0})^{2}}$, $B_0=-\dfrac{au_{0}}{b+u_{0}}$. 
\begin{align*}
&A_0=1-2u_{0}-\dfrac{abu_{0}}{(b+u_{0})^{2}}=\dfrac{u_{0}}{b+u_{0}}(1-b-2u_{0}),&\\
&B_0=-\dfrac{au_{0}}{b+u_{0}}=u_{0}-1<0.&
\end{align*}
$F(r,l,\cdot):X_{\mathbb{C}}\rightarrow X_{\mathbb{C}} $ is a $C^{k}$ $(k\geq3)$ function and  given by
\begin{equation}\label{eqF}
F(r,l,\phi)=\begin{pmatrix}
f_{1}(r,l,\phi)-A_0\phi_{1}-B_0\phi_{2}\\
f_{2}(r,l,\phi) -\;r\phi_{1}\;+\;r\phi_{2}
\end{pmatrix}
\end{equation}
with
\begin{align*}
&f_{1}(r,l,\phi)=(\phi_{1}+\!u_{0})(1-\phi_{1}-u_{0})-
\frac{a(\phi_{1}+u_{0})(\phi_{2}+\!v_{0})}{\phi_{1}+b+\!u_{0}},&\\
&f_{2}(r,l,\phi)=r(\phi_{2}+v_{0})(1-\frac{\phi_{2}+v_{0}}{\phi_{1}+u_{0}}),&
\end{align*}
for $\phi=(\phi_{1},\phi_{2})\in X_{\mathbb{C}}$ and satisfies $F(r,l,0)=0,~D_{\phi}F(r,l,0)=0$.

The linearized system of \eqref{eqA} at $(u_0,v_0)$ is
\begin{equation}
\frac{\mathrm{d}}{\mathrm{d}t}U =D(r,l)\Delta U +L(r,l)U.
\end{equation}
And the corresponding characteristic equation is
\begin{equation}\label{eqcha}
\mathbf{\Delta}(\lambda)y=\lambda y -D(r,l)\Delta y -L(r,l)y=0,
\end{equation}
for some $y\in\mathrm{dom}(\Delta)\backslash\{0\},$ which is equivalent to the sequence of characteristic equations
\begin{equation}\label{characteristic}
\begin{aligned}
&\lambda^{2}-{T_{n}}(r,l)\lambda+{D_{n}}(r,l)=0,&&n=0,1,2,\cdots&
\end{aligned}
\end{equation}
with
\begin{align*}
&{T_{n}}(r,l)=A_{0}-(d_{1}+d_{2})\dfrac{n^{2}}{l^{2}}-r,&\\
&{D_{n}}(r,l)=d_{2}\dfrac{n^{2}}{l^{2}}(d_{1}\dfrac{n^{2}}{l^{2}}-A_{0})+r(d_{1}\dfrac{n^{2}}{l^{2}}-A_{0}-B_{0}).&
\end{align*}

It is obvious that
\begin{equation}\label{eqTnDn}
\begin{aligned}
&{T_{n}}(r,l)<0 \hspace{0.3cm}\Longleftrightarrow\hspace{0.3cm} r>r^{H}_{n}(l):=A_{0}-(d_{1}+d_{2})\frac{n^{2}}{l^{2}},\\
&{D_{n}}(r,l)>0\hspace{0.2cm}\Longleftrightarrow\hspace{0.3cm} r>r_{n}^{T}(l):=-{d_{2}\frac{n^{2}}{l^{2}}(d_{1}\frac{n^{2}}{l^{2}}-A_{0})}/{(d_{1}\frac{n^{2}}{l^{2}}-A_{0}-B_{0})},
\end{aligned}
\end{equation}
and we can get the following conclusion directly.
\begin{lemma}
	 For system  \eqref{eqA}, assume that $a,r,l,d_1,d_2>0$, $1>b>0$.  
	 If $a\leq\dfrac{(b+1)^{2}}{2(1-b)},$ then the constant steady state $(u_0,v_0)$ of \eqref{eqA} is local asymptotic stability for arbitrary $r,l>0$.

\end{lemma}
\begin{proof} 
Since $a\leq\dfrac{(b+1)^{2}}{2(1-b)}$, we have $A_{0}\leq 0$. Which means $r_{n}^{H}(l),r_{n}^{T}(l)\leq 0$ for all $l>0$ and $n\in\mathbb{N}$. Consequently, we obtain ${T_{n}}(r,l)<0$ and ${D_{n}}(r,l)>0$ for arbitrary $r,l>0$ and $n\in\mathbb{N}$. Or more precisely, all the eigenvalues of character equation \eqref{eqcha} have negative real part.   This completes the proof.
\end{proof}

In the following, we will mainly study the dynamics of the system \eqref{eqA} when $a>\dfrac{(b+1)^{2}}{2(1-b)}$. 
For the further study, we define two  auxiliary  functions 
$$g_{1}(x)=\dfrac{d_{2}x(A_{0}-d_{1}x)}{d_{1}x-A_{0}-B_{0}},\hspace{0.5cm}x\geq 0,$$
and
$$g_{2}(x)=A_{0}-(d_{1}+d_{2})x,\hspace{0.5cm}x\geq 0.$$
These two functions have the following properties.
\begin{proposition}\label{pro1} Assume that $d_1,d_2>0$, $1>b>0$, $a>\dfrac{(b+1)^{2}}{2(1-b)}$. Then we have
	\begin{enumerate}
		\item $\left\{\begin{aligned}
		&g_{1}'(x)>0,\;\, \textit{when}\;x\in [0,\hat{x}),&\\
		&g_{1}'(x)= 0,\;\,\textit{when}\;x=\hat{x}, \hspace{1cm}\textit{with}\; \hat{x}=\dfrac{1}{d_{1}}[(A_{0}+B_{0})+\sqrt{B_{0}(A_{0}+B_{0})}]<\frac{A_{0}}{d_{1}},& \\
		&g_{1}'(x)<0,\;\,\textit{when}\;x\in(\hat{x},+\infty),&
		\end{aligned}\right.$
		\item $g_{1}(0)=g_{1}(\dfrac{A_{0}}{d_{1}})=0,$
		%		\item $\max\limits_{x\in[0,\infty)}g(x) = g_{1}(\hat{x})=-\dfrac{d_{2}}{d_{1}}\{A_{0}+2[B_{0}+\sqrt{B_{0}(A_{0}+B_{0})}]\},$
		\item $g_2(x)$ is a linear decreasing function, and $g_{2}(0)=A_{0}>0,~ g_{2}(\dfrac{A_{0}}{d_{1}+d_2})=0,~ g_{2}(\dfrac{A_{0}}{d_{1}})<0$.
		\item There is only one intersection point $\bar{x}$ of $g_{1}(x)$ and $g_{2}(x)$ in the interval $(0,\frac{A_{0}}{d_{1}})$.
		Here $\bar{x}=\dfrac{1}{2d_{1}^{2}}[2d_{1}A_{0}+(d_{1}+d_{2})B_{0}+\sqrt{(d_{1}+d_{2})^{2}B_{0}^{2}+4d_{1}d_{2}A_{0}B_{0}}]$.
		%	    \item In the interval $(0,\frac{A_{0}}{d_{1}})$, these two functions have only one intersection point 
		%	    \item $g_{2}(\hat{x})=-\dfrac{d_{2}}{d_{1}}\{A_{0}+(1+\dfrac{d_{1}}{d_{2}})[B_{0}+\sqrt{B_{0}(A_{0}+B_{0})}]\}.$
		%\item $g(\hat{x})-f(\hat{x})=(\dfrac{d_{2}}{d_{1}}-1)[B_{0}+\sqrt{B_{0}(A_{0}+B_{0})}]$
		\item $ r_n^T (l)=g_1(\frac{n^2}{l^2})\leq g_1(\hat{x})$,\quad $ r_n^H(l)=g_2(\frac{n^2}{l^2})\leq A_0$. 
	\end{enumerate} 
\end{proposition}
\begin{proof}
	The proof is trivial and will be omitted.
\end{proof}

It is easy to say that the characteristic equation \eqref{eqcha} has pure imaginary eigenvalues for some $r,l>0$ only when there exists a $n\in\mathbb{N}$, such that $T_n(r,l)=0, \; D_n(r,l)>0$. Which can occurs if
\begin{equation}\label{eqrH}
r = r_n^H(l)>r_{n}^{T}(l)>0, \hspace{1cm} n\in\mathbb{N}.
\end{equation}
And \eqref{eqcha} has zero eigenvalues for some $r,l>0$ if and only if there is a $n\in\mathbb{N}$, such that $D_n(r,l)=0$. Which can occurs when
\begin{equation}\label{eqrS}
r = r_{n}^{T}(l)>0, \hspace{1cm} n\in\mathbb{N}.
\end{equation}
 
With the combination of above analysis, we have the following conclusions about the eigenvalues of the characteristic equation \eqref{eqcha} with zero real part.
\begin{lemma}
	 Assume that $d_1,d_2,r,l>0$, $1>b>0$, $a>\dfrac{(b+1)^{2}}{2(1-b)}$. $r_{n}^{H}(l)$, $r_{n}^{T}(l)$ are defined by \eqref{eqTnDn}. Let \begin{equation}
	 l_n^H:= n \sqrt{\frac{1}{\bar{x}}}, \hspace{1cm}
	 l_n^{T}:= n \sqrt{\frac{d_1}{A_0}},  \hspace{1cm}
	 \forall n \in \mathbb{N}.
	 \end{equation} 
	 And $N_1(l),N_2(l)\in\mathbb{N}$ are two non-negative integers, such that 
	 $l_{N_{1}}^{H}<l\leq l_{N_{1}+1}^{H},\;\; l_{N_{2}}^{T}<l\leq l_{N_{2}+1}^{T}.$ Then we have:
	 \begin{enumerate}
	 	\item The characteristic equation \eqref{eqcha} has one pair of simple pure imaginary eigenvalues $\pm\mathrm{i}\sqrt{D_n(r_n^H(l),l)}:=\pm \mathrm{i}w_{n}$, when $r=r_{n}^{H}(l)$ $(0\leq n \leq N_{1})$.
	 	\item If $0<l\leq \sqrt{\frac{d_1}{A_0}}$ $(i.e., \; N_2 = 0)$, then the characteristic equation \eqref{eqcha} has no zero eigenvalues.
	 	\item If $l>\sqrt{\frac{d_1}{A_0}}$ $(i.e., \; N_2 \geq 1)$, then the characteristic equation \eqref{eqcha} has at least one zero eigenvalue when $r=r_{n}^{T}(l)$ $(1\leq n \leq N_{2})$.
	 \end{enumerate}	 
\end{lemma}
\begin{proof} 1. According to 4 in {\bf Proposition} \ref{pro1}, we have $r_n^H(l)>r_{n}^{T}(l)$ $(n\in\mathbb{N})$ only when $0<\frac{n^2}{l^2}<\bar{x}$ $(i.e., \; l>l_n^H)$. From \eqref{eqrH} and $l_{N_{1}}^{H}<l\leq l_{N_{1}+1}^{H}$, the existence of the pure imaginary eigenvalues when  $r = r_n^H$  $(0\leq n \leq N_{1})$ has been proved. Meanwhile, $T_{n}(r,l)$ about $r$ and $n$ is a strictly monotonic function, which means that $T_n(r_n^H,l)=0$ $(0\leq n \leq N_{1})$ and $T_j(r_n^H,l)\neq 0$ $(j\neq n)$. This completes the proof.
	
	2. Since $A_0+B_0<0$, we can get $r_{n}^{T}(l)>0$ $(n\geq 1)$ is equivalent to $l>l_n^{T}$ and $r_0^{T} = 0$.  If $0<l\leq \sqrt{\frac{d_1}{A_0}}$, that means $D_n(r,l)>0$ for all $r>0$, which completes the proof.
	
	3. If $l>\sqrt{\frac{d_1}{A_0}}$, then $r_{n}^{T}(l)>0$ if and only if $1\leq n \leq N_{2}$. From \eqref{eqrS}, the statement is proved.
\end{proof}
Further consider the impact of the domain size $l$ on the system \eqref{eqA}. Define the following sets
\begin{equation}\label{eql}
\begin{aligned}
%&S_1:=(l_{N_{2}}^{T},l_{N_{2}+1}^{T}],\\
%&S_2:=({l}_{N_{1}}^{H},{l}_{N_{1}+1}^{H}]\cap(l_{N_{2}}^{T},l_{N_{2}+1}^{T}],\\
&L_{TT}:=\{l\in S_1: r_{i}^{T}(l)=r_{j}^{T}(l),~1\leq i<j\leq N_{2}\},\\
&L_{TH}:=\{l\in S_2: r_{i}^{H}(l)=r_{j}^{T}(l),~0\leq i\leq N_{1}<j\leq N_{2}\},\\
&L_{TTH}:=\{l\in S_2: r_{i}^{H}(l)=r_{j}^{T}(l)=r_{k}^{T}(l),~0\leq i\leq N_{1}<j<k\leq N_{2}\},
\end{aligned}
\end{equation}
with 
$S_1:=(l_{N_{2}}^{T},l_{N_{2}+1}^{T}],$ $S_2:=({l}_{N_{1}}^{H},{l}_{N_{1}+1}^{H}]\cap(l_{N_{2}}^{T},l_{N_{2}+1}^{T}]$ and $N_2\geq 1.$ Each set has only a limited number of elements, since they composed of the roots of a finite number of polynomials which satisfy the condition $l\in S_i (i=1,2)$.
\begin{example}

	Let $d_1=0.417243$, $d_2=4.697383$, $a=1.472554$, $b=0.045949$. When $N_1=0$, $N_2=3$,  we can calculate that $L_{TT}=\{3.022593,  3.617713\}$, $L_{TH} =\{3.022593\}$ and $L_{TTH}=\{3.022593\}$. 
%	\begin{equation*}
%	\begin{aligned}
%	&L_{TT}=\{3.022593,  3.617713\},\\
%	&L_{TH} =\{3.022593\},\\
%	&L_{TTH}=\{3.022593\}.
%	\end{aligned}
%	\end{equation*}
%	\begin{equation*}
%	\begin{aligned}
%	& r_{0}^{H}(l)= 0.501219,\quad r_n^H(l)\leq g_1(\bar{x})\;(n\geq 1),\\
%	&r_{1}^{T}(l)= 0.501219, \quad r_{2}^{T}(l)= 1.084062, \quad r_{3}^{T}(l)= 0.501219,\quad r_{n}^{T}(l)\leq 0 \;(n\geq 4).&
%	\end{aligned}
%	\end{equation*}
\begin{figure}[htbp]
\begin{minipage}{0.48\linewidth}
	If choose $l = 3.022593$, then we have   		
\begin{equation*}
\begin{aligned}
& r_{0}^{H}(l)= 0.501219,\quad r_n^H(l)\leq g_1(\bar{x})\;(n\geq 1),\\
&r_{1}^{T}(l)= 0.501219, \quad r_{2}^{T}(l)= 1.084062,\\ &r_{3}^{T}(l)= 0.501219,\quad r_{n}^{T}(l)\leq 0 \;(n\geq 4).
\end{aligned}
\end{equation*}
For more intuitive understanding, please refer to Figure \ref{fig-g}.\\\\	
\end{minipage}
\begin{minipage}{0.48\linewidth} 
\centering                                      
\includegraphics[scale=0.36]{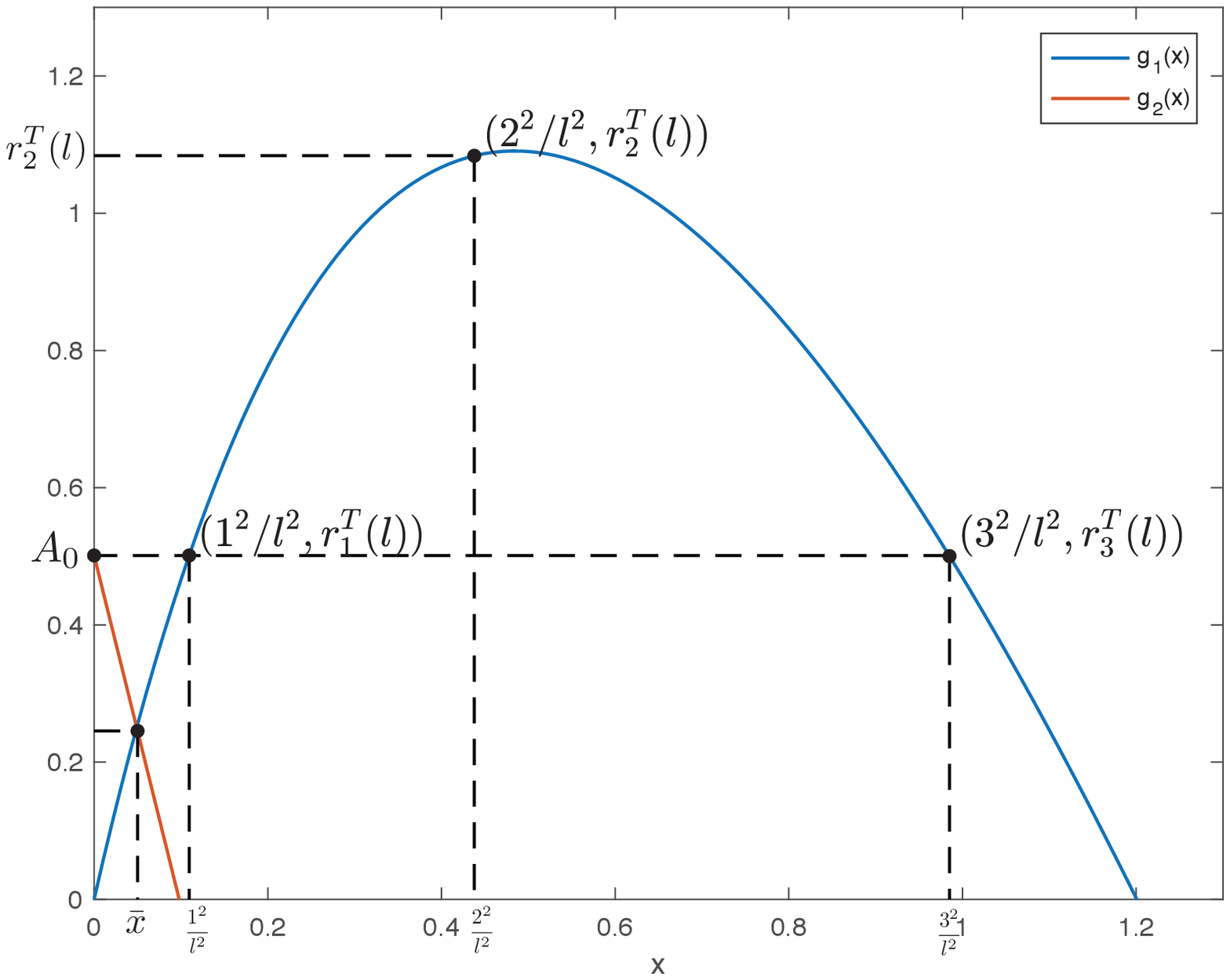}
\caption{}          
\label{fig-g}
\end{minipage}
\end{figure}
\end{example}

Benefit from \eqref{eql}, we have a more accurate conclusion about the eigenvalues with zero real part in the following.
\begin{theorem}\label{theorem1}
	Assume that $d_1,d_2,r>0$, $1>b>0$, $a>\dfrac{(b+1)^{2}}{2(1-b)}$, $l>\sqrt{\frac{d_1}{A_0}}$. $r_{n}^{H}(l)$, $r_{n}^{T}(l)$, $L_{TT}$, $L_{TH}$, $L_{TTH}$ are defined by \eqref{eqTnDn} and \eqref{eql}, respectively. 
	\begin{enumerate}
		\item If $l\notin L_{TT}\cup L_{TH}$, then the characteristic equation \eqref{eqcha} has just one pair of pure imaginary eigenvalues $\pm \mathrm{i}\, \omega_i$ when $r = r_i^H(l)\,(0\leq i\leq N_1)$, and one zero eigenvalue when $r = r_{j}^{T}(l)\,(1\leq j\leq N_2)$. 
		
        \item If $l\in L_{TH}\setminus L_{TT}$, then the characteristic equation \eqref{eqcha} has one pair of pure imaginary eigenvalues $\pm \mathrm{i} \omega_i$ and one zero eigenvalue  when $r = r_i^H(l)=r_{j}^{T}(l)$. Here $0\leq i\leq N_1 <j \leq N_2$ are two non-negative integers which satisfy $r_i^H(l)=r_{j}^{T}(l)$.
        
%	    \item If $l\in L_{TT}\setminus L_{TH}$,  then there exist two non-negative integers $1\leq i< j\leq N_2$, such that $r_{i}^{T}(l) = r_{j}^{T}(l)$. And the characteristic equation \eqref{eqcha} has two zero eigenvalues when $r = r_{i}^{T}(l) = r_{j}^{T}(l)$.
	    
	    \item If $l\in L_{TT}\setminus (L_{TH}\cup \{l_{N_1+1}^H\})$ (or $l=l_{N_1+1}^H\notin L_{TT}$),  then the characteristic equation \eqref{eqcha} has two zero eigenvalues when $r = r_{i}^{T}(l) = r_{j}^{T}(l)$ (or $r = r_{N_1+1}^H(l) = r_{N_1+1}^S(l)$). Here $1\leq i< j\leq N_2$ are two non-negative integers which  satisfy $r_{i}^{T}(l)=r_{j}^{T}(l)$.
	    
%	    \item If $l=(N_1+1)\sqrt{\frac{1}{\bar{x}}}$,  then the characteristic equation \eqref{eqcha} has two zero eigenvalues when $r = r_{N_1+1}^{T}(l) = r_{N_1+1}^H(l)$.
	    \item If $l=l_{N_1+1}^H\in L_{TT}$, then the characteristic equation \eqref{eqcha} has three zero eigenvalues when $r = r_{N_1+1}^H(l) = r_{N_1+1}^{T}(l)$.
	    \item If $l\in L_{TTH}$,  then the characteristic equation \eqref{eqcha} has one pair of pure imaginary eigenvalues $\pm \mathrm{i}\, \omega_i$ and two zero eigenvalues when $r = r_i^H(l) = r_{j}^{T}(l)=r_k^{T}(l)$. Here $0\leq i\leq N_1 <j<k \leq N_2$ are three non-negative integers which satisfy $r_{i}^{T}(l)=r_{j}^{T}(l) = r_k^{T}(l)$.
	\end{enumerate}
\end{theorem}
\begin{proof} Here we only give the proof of the first result, and the remainder of the arguments is analogous to it. 
	 Due to $l\notin L_{TT}\cup L_{TH}$, we have 
	 \begin{equation*}
     r_i^H(l) \neq r_{n}^{T}(l),\qquad \forall\; 0\leq i\leq N_1,\; n\in \mathbb{N},
	 \end{equation*}
	 and 
	 \begin{equation*}
	 r_{j}^{T}(l)\neq r_{n}^{T}(l),\qquad \forall\; 1\leq j \leq N_2,\; j\neq n\in\mathbb{N}.
	 \end{equation*}
That is 
     \begin{equation*}
     T_{i}(r_{i}^{H}(l),l) = 0,\; D_{i}(r_{i}^{H}(l),l) > 0,\; T_n(r_{i}^{H}(l),l) \neq 0,  \; D_{n}(r_{i}^{H}(l),l), \neq 0, \;\forall n\neq i,
     \end{equation*}
     and  
     \begin{equation*}
     T_{j}(r_{j}^{T}(l),l) \neq 0,\; D_{j}(r_{j}^{T}(l),l) = 0,\; T_n(r_{j}^{T}(l),l) \neq 0,  \; D_{n}(r_{j}^{T}(l),l), \neq 0, \;\forall n\neq j.
     \end{equation*}
Thus when $r = r_i^H(l)\,(0\leq i\leq N_1)$ (or $r = r_{j}^{T}(l)\,(1\leq j\leq N_2)$), all eigenvalues except $\pm\mathrm{i}\, \omega_i$ (or $0$) have non-zero real part.    
Which completes the proof.
\end{proof}

According to the properties of $g_1(x)$, $g_2(x)$ in {\bf{Proposition \ref{pro1}}}, we know that $\max\limits_{n\geq 0}r_n^H(l) = A_0>0$, and
\begin{equation}\label{eqR}
\max\limits_{n\geq 0}r_{n}^{T}(l) = \left\{\begin{aligned}
&0,&\quad  &\text{if}\;\,0<l\leq \sqrt{\frac{d_1}{A_0}},&\\ 
&\max\limits_{1\leq n\leq N_2}r_{n}^{T}(l):=r_*,&\quad  &\text{if}\;\,l> \sqrt{\frac{d_1}{A_0}}.&\end{aligned}\right.
\end{equation}
Combine with \eqref{eqTnDn} and the linear stability theory, we have the following result.
\begin{theorem}\label{throremC}
	For system \eqref{eqA}, assume that $d_1,d_2,r,l>0$, $1>b>0$, $a>\dfrac{(b+1)^{2}}{2(1-b)}$. Then the constant steady state $(u_{0},v_{0})$ of \eqref{eqA} is locally asymptotically stable when $r>\max\{A_{0},\,r_{*}\}$ and unstable when $r<\max\{A_{0},\,r_{*}\}$.
\end{theorem}
\begin{proof} If $r>\max\{A_{0},\,r_{*}\}$, we have $r>r_{n}^{H}$ and $r>r_{n}^{T}$, $\forall n\in\mathbb{N}$. Which means $T_{n}(r,0)<0$ and $D_{n}(r,0)>0$, $\forall n\in\mathbb{N}$, thus all eigenvalues of \eqref{eqcha} have strictly negative real part and $(u_0,v_0)$ is  locally asymptotically stable. 
	When $r<\max\{A_{0},\,r_{*}\}$, we have either $T_{0}(r,0)>0$ or $D_{n_{*}}(r,0)<0$ for some $n_*\in\mathbb{N}$. That means there exists at least one dimensional unstable manifold near $(u_{0},v_{0})$. Thus $(u_{0},v_{0})$ is unstable. 
\end{proof}

%Form \eqref{eqTnDn}, we observe that the equilibrium $(u_0,v_0)$ in \eqref{eqA} is stable if the value of the birth ratio $r$ is large enough, or more precisely, it is when $r>\max\{A_0,r_*\}$. With the decrease of the birth ratio $r$, the system \eqref{eqA} will exhibit different dynamic behaviors when $r$ reaches the value of $A_0$ first or $r_*$. 
%Thus, it is meaningful to study the size of $A_0$ and $r_*$.

Form the bifurcation theory, it is easy to know that with the decrease of the birth ratio $r$, the system \eqref{eqA} will exhibit different dynamic behaviors when $r$ reaches the value of $A_0$ first or $r_*$. 
Thus, it is meaningful to study the size of $A_0$ and $r_*$.

Through some simple calculations, we can calculate that 
\begin{equation}\label{eqAg}
\begin{split}
A_0 - g_1(\hat{x}) 
&= \frac{1}{d_1}[(d_1+d_2)A_0+2d_2B_0+2d_2\sqrt{B_0(A_0+B_0)}]\\
%&=- \frac{A_0[(d_{1}+d_{2})^2A_{0}+4d_1d_2B_{0}]}{d_1[2d_{2}\sqrt{B_{0}(A_{0}+B_{0})}-(d_{1}+d_{2})A_{0}-2d_{2}B_{0}]}\\
:&= h_1(d_1,d_2,A_0,B_0)\cdot h_2(u_0,a,b,d_1,d_2)\cdot h_3(a,b,d_1,d_2),
\end{split}
\end{equation}
with 
\begin{equation*}
\begin{split}
&h_1(d_1,d_2,A_0,B_0) = \frac{A_0}{d_1[2d_{2}\sqrt{B_{0}(A_{0}+B_{0})}-(d_{1}+d_{2})A_{0}-2d_{2}B_{0}]},\\
&h_2(u_0,a,b,d_1,d_2) = \frac{u_0}{(b+u_0)[(d_1-d_2)^2a+(d_1+d_2)^2\sqrt{(a+b-1)^2+4b}]},\\
%\end{split}
%\end{equation*}
%and 
%\begin{equation*}
&h_3(a,b,d_1,d_2) = [(d_{1}\!+\!d_{2})^{4}\!-\!(d_{1}\!-\!d_{2})^{4}]a^{2}\!+\!2(d_{1}\!+\!d_{2})^{4}(b\!-\!1)a\!+\!(d_{1}\!+\!d_{2})^{4}(b\!+\!1)^{2}.
\end{split}
\end{equation*}

With regard to the size of $A_0$ and $g_1(\hat{x})$, we give the following conclusion.
\begin{lemma}\label{lemma3}
Assume that $d_1,d_2,r,l>0$, $1>b>0$, $a>\dfrac{(b+1)^{2}}{2(1-b)}$.
\begin{enumerate}
	\item $A_0> g_1(\hat{x})$ if the parameters also meet one of the following conditions
	\begin{enumerate}[(1)]
		\item $d_1\geq d_2$,
		\item $d_1<d_2$, $0<b \leq  b_*$, $0<a<a_{-}$ or $a>a_{+}$,
		\item $d_1<d_2$, $b_*<b<1$,
	\end{enumerate}
    \item $A_0 = g_1(\hat{x})$ if  
     $d_1<d_2$, $0<b \leq  b_*$, $a=a_-$ or $a=a_+$,
    \item $A_0< g_1(\hat{x})$ if 
     $d_1<d_2$, $0<b \leq  b_*$, $a_-<a<a_+$. 
\end{enumerate}
Here
\begin{equation}\label{eqab}
\begin{aligned}
&b_*:=\dfrac{[(d_{1}+d_{2})^{2}-\sqrt{(d_{1}+d_{2})^{4}-(d_{1}-d_{2})^{4}}]^{2}}{(d_{1}-d_{4})^{4}}<1,&\\%=\dfrac{[(d_{1}+d_{2})^{2}-2\sqrt{2d_{1}d_{2}(d_{1}^{2}+d_{2}^{2})]^{2}}{(d_{1}-d_{4})^{4}},\\
&a_{\pm}:=\dfrac{(1\!-\!b)(d_{1}\!+\!d_{2})^{4}\!\pm\!(d_{1}\!+\!d_{2})^{2}\sqrt{(b\!+\!1)^{2}(d_{1}\!-\!d_{2})^{4}\!-\!4b(d_{1}\!+\!d_{2})^{4}}}{(d_{1}\!+\!d_{2})^{4}\!\!-\!\!(d_{1}\!-\!d_{2})^{4}}.&
%&a_{2}:=\dfrac{(1-b)(d_{1}+d_{2})^{4}+(d_{1}+d_{2})^{2}\sqrt{(b+1)^{2}(d_{1}-d_{2})^{4}-4b(d_{1}+d_{2})^{4}}}{(d_{1}+d_{2})^{4}-(d_{1}-d_{2})^{4}},&
\end{aligned}
\end{equation}
\end{lemma}

\begin{proof}
	If $d_1\geq d_2$, then $A_0-g_1(\hat{x})\geq 2\frac{d_2}{d_1}[A_0+B_0+\sqrt{B_0(A_0+B_0)}]> 0.$ This complete the proof of (1) in the first part.
	
	If $d_1<d_2$, then it is easy to show that $h_1(d_1,d_2,A_0,B_0)>0$ and $h_2(u_0,a,b,d_1,d_2)>0$. Thus, the sign of $A_0-g_1(\hat{x})$ is the same as $h_3(a,b,d_1,d_2)$. Thanks to $h_3(a,b,d_1,d_2)$ is a parabolic equation respect to $a$, we can apply the discriminant $\Delta_a = 4(d_1+d_2)^4[(d_1-d_2)^4(b+1)^2-4(d_1+d_2)^4b]$ to distinguish its sign. A short calculation revealed that $\Delta_a\leq 0$ when $b_*\leq b<1$ and $\Delta_a> 0$ if $0<b<b_*$. Moreover, $a=a_{\pm}$ are the roots of  $h_3(a,b,d_1,d_2)=0$ when $\Delta_a> 0$. 
	Then the remaining parts of the lemma follow immediately from what we have proved.
\end{proof}

When $A_{0}<g_{1}(\hat{x})$,
solving $x$ from the equation $A_{0}=g_{1}({x})$ in the interval $x\in(0,{A_{0}}/{d_{1}})$, we get two points $x_{-}<x_{+}$ with the form
\begin{equation}\label{eqx}
x_{\pm}=\frac{1}{2d_{1}d_{2}}[(d_{2}-d_{1})A_{0}\pm\sqrt{(d_{1}+d_{2})^{2}A_{0}^{2}+4d_{1}d_{2}A_{0}B_{0}}].
%&x_{2}=\frac{1}{2d_{1}d_{2}}[(d_{2}-d_{1})A_{0}+\sqrt{(d_{1}+d_{2})^{2}A_{0}^{2}+4d_{1}d_{2}A_{0}B_{0}}].&
\end{equation}

Applying { Lemma \ref{lemma3}}, we can obtain the size of the value between $A_0$ and $r_*$.

\begin{theorem} \label{theorem2}
	For system \eqref{eqA}, 
	assume that $d_1,d_2,r,l>0$, $1>b>0$, $a>\dfrac{(b+1)^{2}}{2(1-b)}$. $r_*$, $b_*$, $a_{\pm}$, $x_{\pm}$ are defined by \eqref{eqR}, \eqref{eqab} and \eqref{eqx}, respectively. Let 
	\begin{equation}
	l_n^- := n\sqrt{\frac{1}{x_-}},  \hspace{1cm}  
	l_n^+ := n\sqrt{\frac{1}{x_+}},  \hspace{1cm} 
	\forall n\in\mathbb{N}.
	\end{equation}
	And $M_1(l),M_2(l)\in\mathbb{N}$ are two non-negative integers, such that 
	$l_{M_{1}-1}^{-}\leq l< l_{M_{1}}^{-},\;\; l_{M_{2}}^{+}<l\leq l_{M_{2}+1}^{+}.$ Then we have:
	\begin{enumerate}
		\item $A_{0}> r_{*}$ if and only if one of the following is satisfied
		\begin{enumerate}[$({\mathbf{A}}1)$]
			\item $0<l\leq \sqrt{\frac{d_1}{A_0}}$,
			\item $d_{2}\leq d_{1}$,
			\item $d_{2}>d_{1},$ $b_*< b<1$,
			\item $d_{2}>d_{1},$ $0<b\leq b_*$, $0<a< a_{-}$ or $a> a_{+}$,
			\item $d_{2}>d_{1}$, $0<b\leq b_*$, $a= a_{-}$ or $a=a_{+}$, but $l\sqrt{\hat{x}}\notin \mathbb{N}$,
			\item $d_{2}>d_{1}$, $0<b<b_*$,  $a_{-}<a<a_{+}$, $M_1(l)>M_2(l)$ and  $l\sqrt{x_{-}},l\sqrt{x_{+}}\notin \mathbb{N}.$
		\end{enumerate}
	
		\item $A_{0}=r_{*}$ if and only if one of the following is satisfied
		\begin{enumerate}
			\item [$({{\mathbf{A}}}5^{'})$] $d_{2}>d_{1}$, $0<b\leq b_*$, $a= a_{-}$ or $a=a_{+}$, and $l=n\sqrt{\frac{1}{\hat{x}}}$, $n\in\mathbb{N}$,
%			 $l\sqrt{\hat{x}}\in \mathbb{N}$,
			\item [$({{\mathbf{A}}}6^{'})$] $d_{2}>d_{1}$, $0<b<b_*$,  $a_{-}<a<a_{+}$, $M_1(l)>M_2(l)$ and  
			$l=n\sqrt{\frac{1}{x_{-}}}$ or $l=n\sqrt{\frac{1}{x_{+}}}$, $n\in\mathbb{N}$.
%		$l\sqrt{x_{-}}\in \mathbb{N}$ or $l\sqrt{x_{+}}\in \mathbb{N}$.
		\end{enumerate}

		\item $A_{0}<r_{*}$ if and only if
     	\begin{enumerate}[$({{\mathbf{A}}}6^{''})$]
     	\item $d_{2}>d_{1}$, $0<b<b_*$,  $a_{-}<a<a_{+}$, $M_1(l)\leq M_2(l)$.  Moreover, $A_{0}< r_{n}^{T}$ only when $M_1\leq n\leq M_2.$
		\end{enumerate}
	\end{enumerate}
\end{theorem}
\begin{proof} 
First of all, it is clear that $r_* = 0$ if $({{\mathbf{A}}}1)$  hold, so naturally we get $A_0>r_*$. 
Next, when parameters meets one of $({{\mathbf{A}}}2)-({{\mathbf{A}}}4)$, it follows from Lemma \ref{lemma3} that $A_0>g_1(\hat{x})\geq r_*$.	
 The condition $({{\mathbf{A}}}5)$ implies $A_{0}=g_{1}(\hat{x})$ and $r_{n}^{T}\neq g_{1}(\hat{x})$ for all $n\in\mathbb{N}$, which means $A_{0}=g_{1}(\hat{x})>r_{*}.$ But if $({{\mathbf{A}}}5^{'})$ hold, it is going to be $A_{0}=g_{1}(\hat{x})=r_{*}$, since it implies that there exists a $n_*\in\mathbb{N}$ such that $\frac{n_{*}^{2}}{l^{2}}=\hat{x}$ and  $r_{n_{*}}^{T}=g_{1}(\hat{x})=r_{*}$.

Finally, under the condition of  $d_{2}>d_{1}$, $0<b<d$,  $a_{-}<a<a_{+}$ , benefit from Lemma \ref{lemma3} we get $r_{n}^{T}>A_{0}$ for some $n\in\mathbb{N}$ only when $l_n^+<l<l_n^-.$
When $M_1(l)>M_2(l)$, we obtain that 
\begin{equation*}
\left\{
\begin{aligned}
&l\geq l_{M_1-1}^-\geq l_{n}^-, &&\forall n\leq M_2\leq M_1-1<M_1\\
&l\leq l_{M_2+1}^+\leq l_{n}^+, &&\forall n>M_2
\end{aligned}
\right.
\end{equation*}
thus $r_{n}^{T}\leq A_{0}$ for any $n\in \mathbb{N}$. Moreover, if $l\sqrt{x_{-}}$ and $l\sqrt{x_{+}}\notin \mathbb{N}$ ($i.e.$, the condition $({{\mathbf{A}}}6)$ is satisfied), it means $A_{0}\neq r_{n}^{T}$ for all $n\in\mathbb{N}$, thus  $A_{0}>r_{*}$ is proved. But if $l\sqrt{x_{-}}\in \mathbb{N}$ or $l\sqrt{x_{+}}\in \mathbb{N}$ ($i.e.$, $({{\mathbf{A}}}6^{'})$ is satisfied), it is easy to get $r_{n_{*}}^{T}=A_{0}$, thus $A_{0}=r_{*}$ is proved. When $M_1(l)\leq M_2(l)$, ($i.e.$, $({{\mathbf{A}}}6^{''})$ is satisfied), we have
\begin{equation*}
\left\{
\begin{aligned}
&l\geq l_{M_1-1}^-\geq l_n^-, &\quad&\forall n<M_1,\\
&l_{n}^+\leq l_{M_2}^+<l<l_{M_1}^-\leq l_{n}^-,&&\forall M_1\leq n\leq, M_2\\
&l\leq l_{M_2+1}^+\leq l_{n}^+, &&\forall n>M_2.
\end{aligned}
\right.
\end{equation*}
Thus $A_{0}<r_{*}$ and $A_{0}< r_{n}^{T}$ if and only if $M_1\leq n\leq M_2$. The proof is completed.
\end{proof}

So far, we have analyzed the distribution of eigenvalues with zero real part in Theorem \ref{theorem1} and the size of $A_0$, $r_*$ in Theorem \ref{theorem2}. Based on these conclusions, we obtain the following bifurcation theorems.
%Through the above analysis of eigenvalues in Theorem \ref{theorem1} and the size of $A_0$, $r_*$ in Theorem \ref{theorem2}, 
%\begin{theorem}\label{throremC}
%For system \eqref{eqA}, assume that $d_1,d_2,r,l>0$, $1>b>0$, $a>\dfrac{(b+1)^{2}}{2(1-b)}$. Then the constant steady state $(u_{0},v_{0})$ of \eqref{eqA} is locally asymptotically stable when $r>\max\{A_{0},\,r_{*}\}$ and unstable when $r<\max\{A_{0},\,r_{*}\}$.
%\end{theorem}
%\begin{proof} When $r>\max\{A_{0},\,r_{*}\}$, we have $r>r_{n}^{H}$ and $r>r_{n}^{T}$, $\forall n\in\mathbb{N}$. Which means $T_{n}(r,0)<0$ and $D_{n}(r,0)>0$, $\forall n\in\mathbb{N}$, thus all eigenvalues of \eqref{eqcha} have strictly negative real part and $(u_0,v_0)$ is  locally asymptotically stable. 
%When $r<\max\{A_{0},\,r_{*}\}$, we have either $T_{0}(r,0)>0$ or $D_{n_{*}}(r,0)<0$ for some $n_*\in\mathbb{N}$. That means there exists at least one dimensional unstable manifold near $(u_{0},v_{0})$. Thus $(u_{0},v_{0})$ is unstable. 
%\end{proof}

\begin{theorem}[{\bf Hopf bifurcation}]\label{theoremH}
For system \eqref{eqA}, 
assume that $d_1,d_2,r,l>0$, $1>b>0$, $a>\dfrac{(b+1)^{2}}{2(1-b)}$.
If $l\notin L_{TT}\cup L_{TH}$, then the system \eqref{eqA} undergoes a Hopf bifurcation when  $r=r_{n}^{H}(l)$ $(0\leq n\leq N_{1})$. The bifurcating periodic solution is spatially homogeneous if it bifurcate from $r=r_{0}^{H}(l)=A_{0}$ and spatially inhomogeneous if it bifurcate from $r=r_{n}^{H}$ and $1\leq n\leq N_{1}$. Furthermore,  the bifurcation solutions can be stable only when $a,b,d_1,d_2,l$ also meet one of $({{\mathbf{A}}}1)$-$({{\mathbf{A}}}6)$ and $r=r_{0}^{H}(l)=A_0$. (i.e., if $a,b,d_1,d_2,l$ meet one of $({{\mathbf{A}}}1)$-$({{\mathbf{A}}}6)$  and $r=r_{n}^{H}(l)$ $(1\leq n\leq N_{1})$, or $a,b,d_1,d_2,l$ meet $({{\mathbf{A}}}6^{''})$ and $r=r_{n}^{H}(l)$ $(0\leq n\leq N_{1)}$ the bifurcation solutions are  unstable. )
\end{theorem}
\begin{proof}
	Since $l\notin L_{TT}\cup L_{TH}$, then the parameters can not meet the condition $({{\mathbf{A}}}5^{'})$ or $({{\mathbf{A}}}6^{'})$. Due to the fact that $\dfrac{\partial T_n(r,l)}{\partial r} = -1<0,$
%	\begin{equation*}
%	\frac{\partial T_n(r,l)}{\partial r} = -1<0, 
%	\end{equation*} 
	the existence of the Hopf bifurcation at $r=r_{n}^{H}(l)$ $(0\leq n\leq N_{1})$ is a direct consequent of Theorem \ref{theorem1}.
	Further assume that the parameters meet one of $({{\mathbf{A}}}1)$-$({{\mathbf{A}}}6)$, then we have  $T_0(r_n^H(l),l) >0$ when $1\leq n\leq N_{1}$, since $r_0^H(l)>r_n^H(l)$. That means there exist at least one eigenvalue of \eqref{eqcha} have positive real part when $r=r_{n}^{H}$ and $1\leq n\leq N_{1}$, thus the periodic solutions which bifurcate from $r=r_{n}^{H}$ $(1\leq n\leq N_{1})$  are unstable.
	If the parameters meet $({{\mathbf{A}}}6^{''})$, the statement can be proved in the same way as above.
\end{proof}
\begin{theorem}[{\bf Turing bifurcation}]\label{theoremS}
	For system \eqref{eqA}, 
	assume that $d_1,d_2,r>0$,  $1>b>0$, $a>\dfrac{(b+1)^{2}}{2(1-b)}$.
	If $l>\sqrt{\frac{d_1}{A_0}}$ and $l\notin L_{TT}\cup L_{TH}$, then the system \eqref{eqA} undergoes a steady state bifurcation when $r=r_{n}^{T}(l)$ $(1\leq n\leq N_{2})$. Moreover, the bifurcation solutions can be stable only when $a,b,d_1,d_2,l$ also meet $({{\mathbf{A}}}6^{''})$ and $r=r_*$. 
%	\begin{enumerate}
%		\item If $a,b,d_1,d_2,l$ also meet one of $({{\mathbf{A}}}1)$ -$({{\mathbf{A}}}6)$, then  bifurcation solutions which are bifurcate from $r=r_{n}^{T}(l)$ are unstable when $1\leq n\leq N_{2}$. 
%		\item If $a,b,d_1,d_2,l$ also meet $({{\mathbf{A}}}6^{''})$, then the bifurcation solutions which are bifurcate from $r=r_{n}^{T}(l)$ are unstable when $1\leq n\leq N_{2}$ and $r_{n}^{T}(l)\neq r_*$.
%	\end{enumerate}
\end{theorem}

\begin{theorem}[{\bf Turing-Hopf bifurcation}]\label{theoremTH}
For system \eqref{eqA}, 
assume that $d_1,$ $d_2,$ $r,$ $l>0$,  $1>b>0$, $a>\dfrac{(b+1)^{2}}{2(1-b)}$. If $l\in L_{TH}$, then system \eqref{eqA}  undergoes a Turing-Hopf bifurcation at  $r=r_{i}^{H}(l)=r_{j}^{T}(l)$ $(0\leq i\leq N_{1}< j\leq N_{2})$. Moreover, the bifurcation solutions can be stable only when $a,b,d_1,d_2,l$ also meet one of $({{\mathbf{A}}}5^{'})$-$({{\mathbf{A}}}6^{'})$ and $r=A_0=r_*$.
%      \begin{enumerate}
%      	\item If $a,b,d_1,d_2,l$ also meet one of $({{\mathbf{A}}}1)$-$({{\mathbf{A}}}6)$ or $({{\mathbf{A}}}6^{''})$, then the bifurcation solutions are unstable.
%      	\item If $a,b,d_1,d_2,l$ also meet one of $({{\mathbf{A}}}5^{'})$-$({{\mathbf{A}}}6^{'})$, then the bifurcation solutions are unstable when $r = r_{i}^{H}(l) = r_{j}^{T}(l)< A_0$.
%      \end{enumerate}
\end{theorem}

\begin{theorem}[{\bf Turing-Turing bifurcation}]\label{theoremBT}
	For system \eqref{eqA}, assume that $d_1,$ $d_2,$ $r,$ $l>0$,  $1>b>0$, $a>\dfrac{(b+1)^{2}}{2(1-b)}$. If $l\in L_{TT}\setminus (L_{TH}\cup \{l_{N_1+1}^H\})$ (or $l=l_{N_1+1}^H\notin L_{TT}$), then system \eqref{eqA} undergoes a Turing-Turing bifurcation when  $r=r_{i}^{T}(l)=r_{j}^{T}(l)$ with  $1\leq i<j\leq N_2$ (or $r = r_{N_1+1}^{H}(l)=r_{N_1+1}^{T}(l)$).   Moreover, the bifurcation solutions can be stable only when $a,b,d_1,d_2,l$ also meet $({{\mathbf{A}}}6^{''})$, and $r=r_{i}^{T}(l)=r_{j}^{T}(l)=r_*$ $(1\leq i<j\leq N_2)$.
%	\begin{enumerate}
%		\item If $a,b,d_1,d_2,l$ also meet one of $({{\mathbf{A}}}1)$-$({{\mathbf{A}}}6)$ or $({{\mathbf{A}}}5^{'})$-$({{\mathbf{A}}}6^{'})$, then the bifurcation solutions are unstable.
%		\item If $a,b,d_1,d_2,l$ also meet $({{\mathbf{A}}}6^{''})$, then the bifurcation solutions are unstable when $r = _{i}^{T}(l)=r_{j}^{T}(l)< r_*$.
%	\end{enumerate}
\end{theorem}

\begin{theorem}[{\bf Hopf-double-Turing bifurcation}]\label{theoremHDZ}
	For system \eqref{eqA}, assume that $d_1,$ $d_2,$ $r,$ $l>0$,  $1>b>0$, $a>\dfrac{(b+1)^{2}}{2(1-b)}$. If $l\in L_{TTH}$, then the system \eqref{eqA} undergoes a Hopf-double-zero bifurcation at  $r=r_{i}^{H}(l)=r_{j}^{T}(l)=r_{k}^{T}(l)$ $(0\leq i\leq N_{1} < j<k\leq N_{2})$. Moreover, the bifurcation solutions can be stable only when $a,b,d_1,d_2,l$ also meet one of $({{\mathbf{A}}}5^{'})$-$({{\mathbf{A}}}6^{'})$, and $r=r_{0}^{H}(l)=r_{j}^{T}(l)=r_{k}^{T}(l)=A_0$ $(N_1+1\leq j<k\leq N_2)$.
%	\begin{enumerate}
%		\item If $a,b,d_1,d_2,l$ also meet one of $({{\mathbf{A}}}1)$-$({{\mathbf{A}}}6)$ or $({{\mathbf{A}}}6^{''})$, then the bifurcation solutions are unstable.
%		\item If $a,b,d_1,d_2,l$ also meet one of $({{\mathbf{A}}}5^{'})$-$({{\mathbf{A}}}6^{'})$, then the bifurcation solutions are unstable when $r = r_{i}^{H}(l)=r_{j}^{T}(l)=r_{k}^{T}(l)< A_0$.
%	\end{enumerate}
\end{theorem}

\begin{theorem}[{\bf Triple-Turing bifurcation}]\label{theoremTZ}
	For system \eqref{eqA}, assume that $d_1,$ $d_2,$ $r,$ $l>0$,  $1>b>0$, $a>\dfrac{(b+1)^{2}}{2(1-b)}$. If $l=l_{N_1+1}^H\in L_{TT}$, then system \eqref{eqA} undergoes a triple-Turing bifurcation at  $r=r_{N_{1}+1}^{H}(l)=r_{N_{1}+1}^{T}(l)$. Moreover, the bifurcation solutions are always unstable. 
\end{theorem}

%\begin{remark}
	 Theorem \ref{theoremH} - Theorem \ref{theoremTZ} are intended solely as a brief summary and not as a rigorous development. The strict proof of Theorem \ref{theoremS} - \ref{theoremTZ} follows in a similar manner of the proof in  Theorem \ref{theoremH}. In the above bifurcation theorems,  the stability of some bifurcation solutions can not be determined by the current analysis. 
	 We list them at here.
%\end{remark}
%\begin{remark}
%	For system \eqref{eqA} with $d_1,$ $d_2,$ $r,$ $l>0$,  $1>b>0$, $a>\dfrac{(b+1)^{2}}{2(1-b)}$. The stability of the bifurcation solutions in the following situations requires further study.
%	\begin{enumerate}
%		\item When $a,b,d_1,d_2,l$ also meet one of $({{\mathbf{A}}}1)$-$({{\mathbf{A}}}6)$, the bifurcation solutions bifurcating from $r=r_{0}^{H}(l)=A_0$. {(\bf Hopf bifurcation)}
%		\item When $a,b,d_1,d_2,l$ also meet $({{\mathbf{A}}}6^{''})$, the bifurcation solutions bifurcating from $r=r_*$. {(\bf Steady state bifurcation)}
%		\item When $a,b,d_1,d_2,l$ also meet one of $({{\mathbf{A}}}5^{'})$-$({{\mathbf{A}}}6^{'})$, the bifurcation solutions bifurcating from $r=A_0=r_*$. {(\bf Turing-Hopf bifurcation)}
%		\item When $l\in L_{TT}$ and $a,b,d_1,d_2,l$ also meet $({{\mathbf{A}}}6^{''})$, the bifurcation solutions bifurcating from $r=r_{i}^{T}(l)=r_{j}^{T}(l)=r_*$, $1\leq i<j\leq N_2$. {(\bf Bogdanov-Tankens bifurcation)}
%		\item When  $l\in L_{TTH}$ and $a,b,d_1,d_2,l$ also meet one of $({{\mathbf{A}}}5^{'})$-$({{\mathbf{A}}}6^{'})$, the bifurcation solutions bifurcating from $r=r_{0}^{H}(l)=r_{j}^{T}(l)=r_{k}^{T}(l)=A_0$, $N_1+1\leq k<k\leq N_2$. {(\bf Hopf-double-zero bifurcation)}
%	\end{enumerate}
%\end{remark}
To give back all the current analysis results to the system \eqref{eqA}, we have the following conclusion.
\begin{remark}
	The two species of system \eqref{eqA} will gradually tend to be uniform in the spatial domain with the increase of the birth ratio $r$.
	The size of the spatial domain is sufficiently large ($l>\sqrt{\frac{d_1}{A_0}}$) and the diffusion coefficient satisfies $d_2>d_1$ are two necessary conditions for these two species to exhibit the spatially inhomogeneous patterns. 
\end{remark}

\section{Spatio-temporal patterns in Holling-Tanner system with a Turing-Hopf singularity }

In this section, we will give a more detailed study of the Holling-Tanner system \eqref{eqA} with the parameters $(r,l)$ near the Turing-Hopf bifurcation point. 
%which is based on the central manifold theorem \cite{Lin1992Centre} and the normal form theory \cite{Faria2000Normal}.  
Assume that the parameters $a,b,d_1,d_2$ are satisfy one of the conditions  $({{\mathbf{A}}}5^{'})$-$({{\mathbf{A}}}6^{'})$. Let $l_*\in L_{TH}$ such that $r_*=r_{n_*}^{T}(l_*)=A_0$ for some $n_*\in\mathbb{N}$. It is obvious that  $(r_*,l_*)$ is a Turing-Hopf bifurcation point, which satisfies the hypothesis $(\textbf{H}1)$, $(\textbf{H}3)$ and $(\textbf{H}4)$ in \cite{An2017}.  

We  adopt the frame and notations of \cite{An2017}. 
 Choosing $$\mathcal{BC}:=\{\psi:[-1,0]\rightarrow X_{\mathbb{C}} : \psi~\mathrm{is~ continuous~ on}~[-1,0),~\exists \lim_{\theta\rightarrow0^{-}}\psi(\theta)\in X_{\mathbb{C}}\}$$ as the phase space.
 Taking the transformation $(\alpha_1,\alpha_2) = (r-r_*,l-l_*)$
 and rewriting \eqref{eqB} into a abstract ordinary differential equation in ${\mathcal{BC}}$,
\begin{equation}\label{eqC}
\frac{\mathrm{d}}{\mathrm{d}t}U^t =A U^t
%D(r_*,l_*)\Delta U +L(r_*,l_*)U
+X_0[\frac{1}{2!}F_0^{(2)}(\alpha_1,\alpha_2,U)+\frac{1}{3!}F_0^{(3)}(\alpha_1,\alpha_2,U)+\cdots],
\end{equation}
with $ A\varphi=\dot{{\varphi}}+X_{0}[D(r_*,l_*)\Delta\varphi(0)+L(r_*,l_*)(\varphi)-\dot{{\varphi}}(0)]$ and
\begin{equation*}
\begin{split}
&F_0^{(2)}(\alpha_1,\alpha_2,U)= 2\{\frac{\partial}{\partial r}[D(r_*,l_*)\Delta+L(r_*,l_*)]\alpha_1 U+\frac{\partial}{\partial l}[D(r_*,l_*)\Delta+L(r_*,l_*)]\alpha_2U\}
\\&\hspace{2.8cm}+\frac{\partial^2}{\partial \hat{u}^2}F(0,0)\hat{u}^2+2\frac{\partial^2}{\partial \hat{u}\hat{v}}F(0,0)\hat{u}\hat{v}+\frac{\partial^2}{\partial \hat{v}^2}F(0,0)\hat{v}^2,\\
&F_0^{(3)}(0,0,U) =  \frac{\partial^3}{\partial \hat{u}^3}F(0,0)\hat{u}^3+3\frac{\partial^3}{\partial \hat{u}^2\hat{v}}F(0,0)\hat{u}^2\hat{v}+3\frac{\partial^3}{\partial \hat{u}\hat{v}^2}\hat{u}\hat{v}^2+\frac{\partial^3}{\partial \hat{v}^3}\hat{v}^3.
\end{split}
\end{equation*}
According to the direct sum decomposition of ${\mathcal{BC}}$ about the characteristic subspaces of $\{\pm\mathrm{i}\omega,0 \}$, we decompose $U^t\in {\mathcal{BC}}$ into 
$$U^t(\theta)=\phi_{1}(\theta)z_1\beta_{n_{1}}+\bar{\phi}_{1}(\theta)\bar{z}_1\beta_{n_{1}}+\phi_{2}(\theta)z_2\beta_{n_{2}}+y(\theta).$$
There are a series of coordinate transformations $(z,y)\rightarrow (z+\frac{1}{j!}U_2^1(z),y+\frac{1}{j!}U_2^2(z))$ as shown in \cite{An2017}, that make the system \eqref{eqC}  homeomorphic to a new system with $y(\theta)=0$ is a local central manifold of it. Moreover, the solutions of \eqref{eqC} are  homeomorphic to the solutions of the new system restrict on central manifold that have the form as
\begin{equation}
W(t)=\phi_{1}(0)z_1\beta_{n_{1}}+\bar{\phi}_{1}(0)\bar{z}_1\beta_{n_{1}}+\phi_{2}(0)z_2\beta_{n_{2}}.
\end{equation}
Here 
\begin{equation}\label{eqNF}
\begin{aligned}
\dot{z_{1}}=&~~~~i\omega_{0}z_{1}\!+\!\frac{1}{2}f_{\alpha_{1}z_1}^{11}\alpha_{1}z_{1}\!+\!\frac{1}{2}f_{\alpha_{2}z_1}^{11}\alpha_{2}z_{1}\!+\!\frac{1}{6}g_{210}^{11} z_1^2{\bar{z}_1} \!+\!\frac{1}{6} g_{102}^{11}z_1z_{2}^2+O(4),\\
\dot{\bar{z}}_1=&-i\omega_{0}{\bar{z}_1}\!+\!\frac{1}{2}\overline{f_{\alpha_{1}z_1}^{11}}\alpha_{1}{\bar{z}_1}\!+\!\frac{1}{2}\overline{f_{\alpha_{2}z_1}^{11}}\alpha_{2}{\bar{z}_1}\!+\!\frac{1}{6}\overline{g_{210}^{11}}z_1{\bar{z}_1^2}\!+\!\frac{1}{6}\overline{g_{102}^{11}}{\bar{z}_1}z_{2}^2+O(4) ,\\
\dot{z_{2}}=&~~~~~~~~~~~~~\frac{1}{2}f_{\alpha_{1}z_{2}}^{13}\alpha_{1}z_{2}\!+\!\frac{1}{2}f_{\alpha_{2}z_{2}}^{13}\alpha_{2}z_{2}\!+\!\frac{1}{6}g_{111}^{13}z_1{\bar{z}_1}z_{2}\!+\!\frac{1}{6}g_{003}^{13}z_{2}^3+O(4),
\end{aligned}
\end{equation}
with the coefficients can be obtained by the computer program, which is fully depends on the formulas that proposed in \cite[Section 3]{An2017}. Moreover, equation \eqref{eqNF} is called a normal form for \eqref{eqC} (or \eqref{eqA}) relative to $\{\pm\mathrm{i}\omega,0 \}$.

For an example, take $a=0.6018$, $b = 0.0077$, $d_1 = 0.4000$, $d_2 = 19.3700$ in \eqref{eqA}. The bifurcation diagram of the nontrivial equilibrium point $(u_0,v_0) = (0.4093, 0.4093)$  in $r-l$ plane is shown in Figure \ref{figBP}. The dotted lines and the solid line represent the steady state bifurcation curves (i.e., $r=r_{n}^{T}(l)$) and the Hopf bifurcation curve (i.e., $r=r_n^H(l)>r_{n}^{T}(l)$), respectively. TH1-TH3 are the Truing-Hopf bifurcation points, which are the intersections of the solid line and the dotted lines. TT1-TT3 are the Turing-Turing bifurcation points, which are the intersections of the dotted lines. 
\begin{figure}[htbp]
	\centering                                                             
	\includegraphics[scale=0.55]{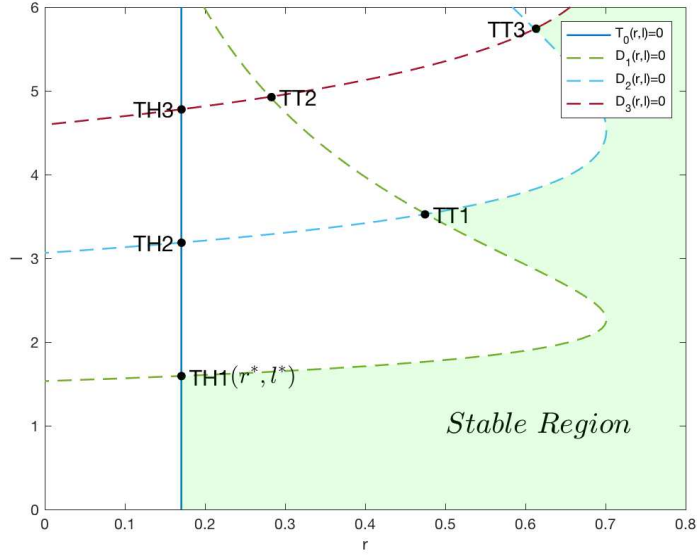}            
	\caption{Bifurcation sets with parameters in $r-l$ plane} 
	\label{figBP}                                                    
\end{figure}

In the following, we are going to work on the detailed dynamics of \eqref{eqA} with the parameters $(\alpha,l)$ near the Turing-Hopf bifurcation point TH1. Here $l_* = 1.593334\in L_{TH}$, and
\begin{equation*}
\begin{aligned}
A_0 &= r_0^H(l_*) = 0.170468,\qquad r_n^H < 0 \,(n\geq 1),\\
r_* &= r_1^{T}(l_*) = 0.170468,\qquad\; r_{n}^{T} < 0 \,(n\geq 2).
\end{aligned}
\end{equation*}
Furthermore, the characteristic equation \eqref{eqcha} has one pair of pure imaginary roots $\pm  0.267646\mathrm{i}$ and a zero root, when $(r,l)=(r_*,l_*)$. 
%In the remainder of this section, we will study the detailed dynamics near $\textmd{TH1}=(r_*,l_*)$.
Based on the algorithm in \cite[Section 3]{An2017}, the  normal forms \eqref{eqNF} of \eqref{eqA} with the Turing-Hopf singularity $(\alpha_*,l_*)$ can be obtained directly, and the coefficients in \eqref{eqNF} are  
\begin{equation*}
\begin{aligned}
&f_{\alpha_1z_1}^{11} = -1.0000 + 1.5701\mathrm{i}, \hspace{0.8cm} f_{\alpha_2z_1}^{11} = 0, \\
&f_{\alpha_1z_2}^{13}= -0.1484,\hspace{2.48cm}  f_{\alpha_2z_2}^{13} = 0.3645,\\
&g_{210}^{11} = - 0.3026 - 4.8696\mathrm{i},\hspace{1cm} g_{102}^{11} = 1.3640 - 10.1736\mathrm{i},\\
& g_{111}^{3} = -1.3543, \hspace{2.65cm}g_{003}^{13} = 0.1241.
\end{aligned}
\end{equation*}

Taking the cylindrical coordinate transformation
\begin{equation*}
\begin{aligned}
z_{1}=R\cos\theta+\mathrm{i}R\sin\theta,\quad
{\bar{z}_1}=R\cos\Theta-\mathrm{i}R\sin\Theta,\quad
z_{2}=V,
\end{aligned}
\end{equation*}
and the re-scaling
\begin{equation*}\label{change}
{\rho}=\sqrt{\frac{|\mathrm{Re}(g_{210}^{11})|}{6}}R,\qquad {v}=\sqrt{\frac{|g_{003}^{13}|}{6}}V.
%\quad\tilde{t}= t/\varepsilon,\quad  \varepsilon=\mathrm{Sign}(\mathrm{Re}(g_{210}^{11})).
\end{equation*}
We get the equivalent planar system of \eqref{eqNF}
\begin{equation}\label{eqplanar}
\begin{aligned}
&\frac{\mathrm{d}\rho}{\mathrm{d}{t}}=-\rho[\epsilon_{1}(\alpha_1,\alpha_2) + \rho^{2}+b_0 v^{2}],\\
&\frac{\mathrm{d}v}{\mathrm{d}{t}}=-v[\epsilon_{2}(\alpha_1,\alpha_2) + c_0 \rho^{2} + d_0 v^{2}].
\end{aligned}
\end{equation}
Here $b_0 =-10.9918$, $c_0 =4.4751$, $d_0 =-1$ and
\begin{equation*}
\begin{aligned}
&\epsilon_{1}(\alpha_1,\alpha_2)=0.5000\,\alpha_1,\\
&\epsilon_{2}(\alpha_1,\alpha_2)=0.0742\,\alpha_1-0.1822\,\alpha_2.\\ 
%& b_0 =-10.9918, \quad
%c_0 =4.4751,\quad
%d_0 =-1.\quad
\end{aligned}
\end{equation*}

There are four positive equilibrium points in the planar system \eqref{eqplanar}
\begin{equation*}
\begin{aligned}
&E_1 = (0,0),  \hspace{4.6cm} \mathrm{for~all}\;\; \epsilon_1, \epsilon_2,\\
&E_2 = (\sqrt{-\epsilon_1},0),\hspace{3.85cm} \mathrm{for}\; \epsilon_1<0,\\
&E_3 = (0,\sqrt{{\epsilon_2}}),\hspace{4.1cm} \mathrm{for}\;\; {\epsilon_2}>0,\\
&E_4 = (\sqrt{\frac{b_0\epsilon_2-d_0\epsilon_1}{d_0-b_0c_0}},\sqrt{\frac{c_0\epsilon_1-\epsilon_2}{d_0-b_0c_0}}),\hspace{0.75cm}\mathrm{for}\;{b_0\epsilon_2-d_0\epsilon_1},{c_0\epsilon_1-\epsilon_2}>0.
\end{aligned}
\end{equation*}
The linearized equation at each equilibrium point is
\begin{equation*}
\frac{\mathrm{d}}{\mathrm{d}{t}}\left(\begin{aligned}\rho\\v\end{aligned}\right)=-\left(\begin{array}{cc}
\epsilon_1+3\rho_i+b_0v_i & 2b_0\rho_iv_i\\
2c_0\rho_iv_i & \epsilon_2+c_0\rho_i^2+3d_0v_i^2
\end{array}\right)\left(\begin{aligned}\rho\\v\end{aligned}\right)
\end{equation*}
with $(\rho_i,v_i)=E_i$ $(i=1,2,3,4)$. By analyzing the corresponding characteristic equations, the bifurcation set in $(\alpha_1,\alpha_2)$ plane is obtained and shown in Figure \ref{figBC}.  
There are seven bifurcation lines $L_1-L_7$ that divide the $(\alpha_1,\alpha_2)$ plane into seven regions $D_1-D_7$, the detailed dynamics of \eqref{eqplanar} in each region have been shown in figure \ref{figphase}. The deeper details please refer to the Case \uppercase\expandafter{\romannumeral 7}a in \cite[Chap.7]{Guckenheimer1983Nonlinear}. 
%The corresponding bifurcation set in $(\epsilon_1,\epsilon_2)$ plane can refer to the Case \uppercase\expandafter{\romannumeral 7}a in \cite[Chap.7]{Guckenheimer1983Nonlinear}. 
According to \cite[Section 4]{An2017}, we list the corresponding relationship between the solution of the plane system \eqref{eqplanar} and the original system \eqref{eqA} in Table 1. 
\begin{table}[htbp]
\captionsetup{singlelinecheck=off,skip=1pt,font=bf}
		\centering 
		{\parbox{.85\textwidth}\caption{}}	
	\begin{tabular}{c|c}
		\hline
		\bf{Planar system} & \bf{Holling-Tanner system \eqref{eqA}}\\
		\hline
		$E_1$         & Constant steady state $(u_0,v_0)$\\
		$E_2$         & Spatially homogeneous periodic solution\\
		$E_3$         & Non-constant steady state\\
		$E_4$         & Spatially non-homogeneous periodic solution\\
		Periodic solution & Spatially non-homogeneous quasi-periodic solution\\
		\hline
	\end{tabular}
\end{table}
\begin{figure}[htbp]
	\centering
	\includegraphics[scale=0.55]{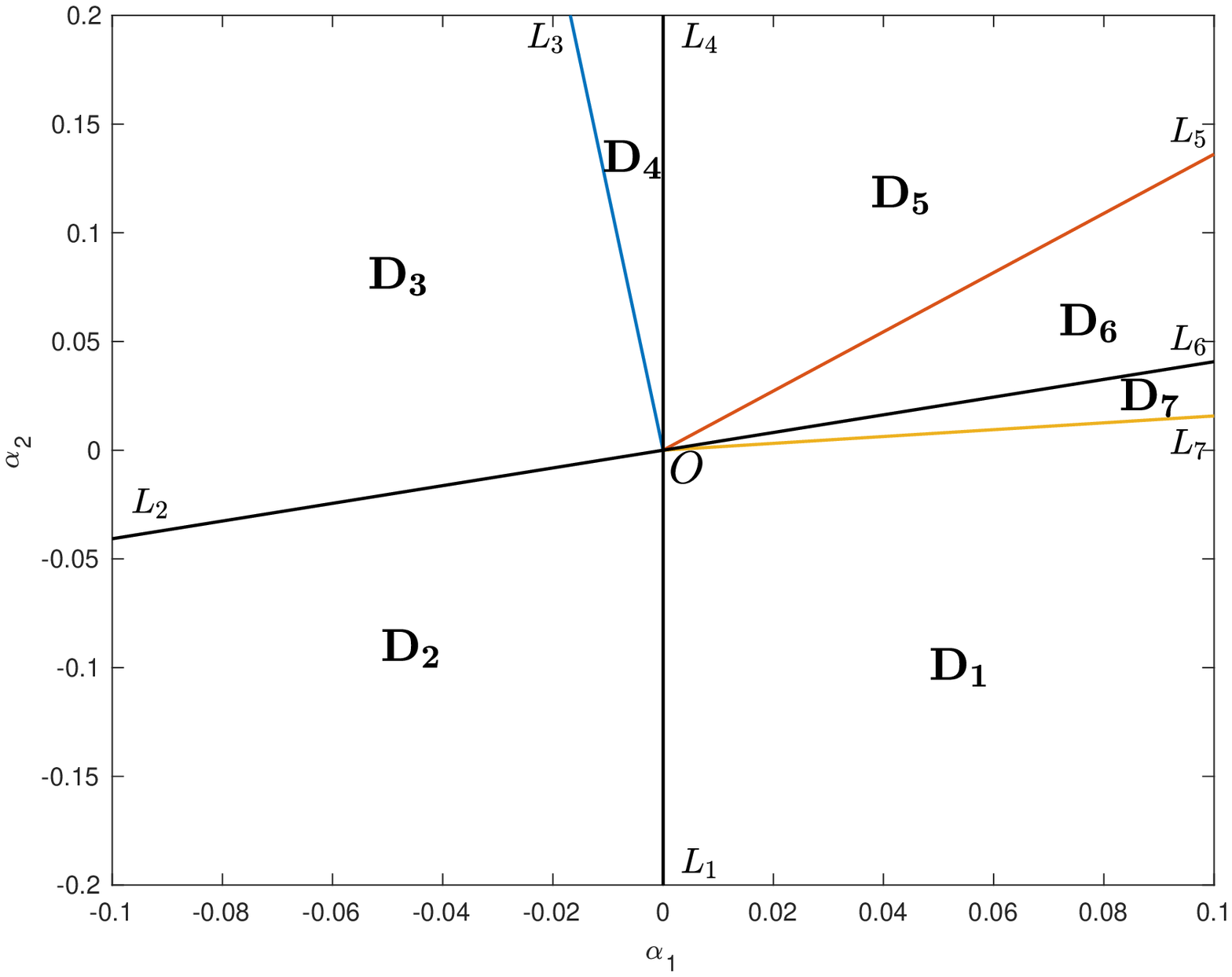} 
	\begin{multicols}{2}
		\begin{small}
			\begin{itemize}
				\item $L_1: \alpha_1=0, \; \alpha_2<0;$
				\item $L_2: \alpha_1<0, \; \alpha_2 = 0.4072 \alpha_1;$
				\item $L_3: \alpha_1<0, \; \alpha_2 = -11.8738 \alpha_1;$\\
				${(\mathit{i.e.,}\; c_0\epsilon_{1}(\alpha_1,\alpha_2)-\epsilon_{2}(\alpha_1,\alpha_2)=0)}$
				\item $L_4: \alpha_1=0, \; \alpha_2>0;$
				\item $L_5: \alpha_1>0, \; \alpha_2 =1.3614\alpha_1;$
				%	\item $L_6: \alpha_1>0, \; \alpha_2 = 1.3614\alpha_1+o(\alpha_1^2);$
				\item $L_6: \alpha_1>0, \; \alpha_2 = 0.4072\alpha_1;$
				\item $L_7: \alpha_1>0, \; \alpha_2 = 0.1575\alpha_1;$\\
				${(\mathit{i.e.,}\; b_0\epsilon_{2}(\alpha_1,\alpha_2)-d_0\epsilon_{1}(\alpha_1,\alpha_2)=0)}$
			\end{itemize}
		\end{small}
	\end{multicols}
	\caption{Bifurcation set in $(\alpha_1,\alpha_2)$ plane}
	\label{figBC}
\end{figure}
\begin{figure}[htbp]
	\centering                                      
	\includegraphics[height = 0.5 \linewidth, width = 1\linewidth]{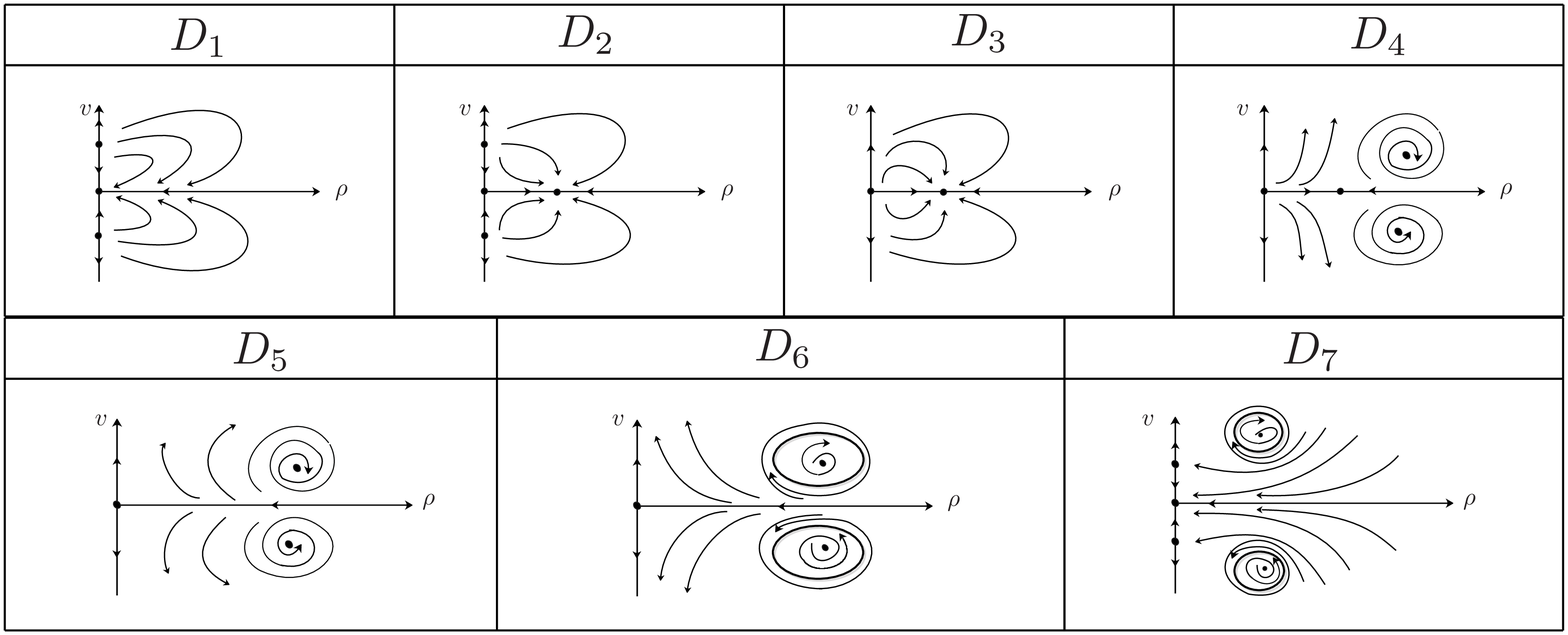}  
	\caption{Phase portraits in $D_1$-$D_7$}   
\end{figure}\label{figphase}
  
For the original Holling-Tanner system \eqref{eqA}, the detailed kinetic properties can be described as follows. When the parameters $(\alpha_1,\alpha_2)$ belongs to $D_1$, there are one stable constant steady state $(u_0,v_0)$ and two unstable non-constant steady state coexist in the system \eqref{eqA}. In Figure \ref{fig1}, we give a simulation with parameters in $D_1$, and the solution eventually stabilize to $(u_0,v_0)$.
\begin{figure}[htbp] 
	\centering                           
	\subfigure[$u(t,x)$]{       
		\begin{minipage}{0.48\linewidth} 
			\centering                                      
			\includegraphics[scale=0.33]{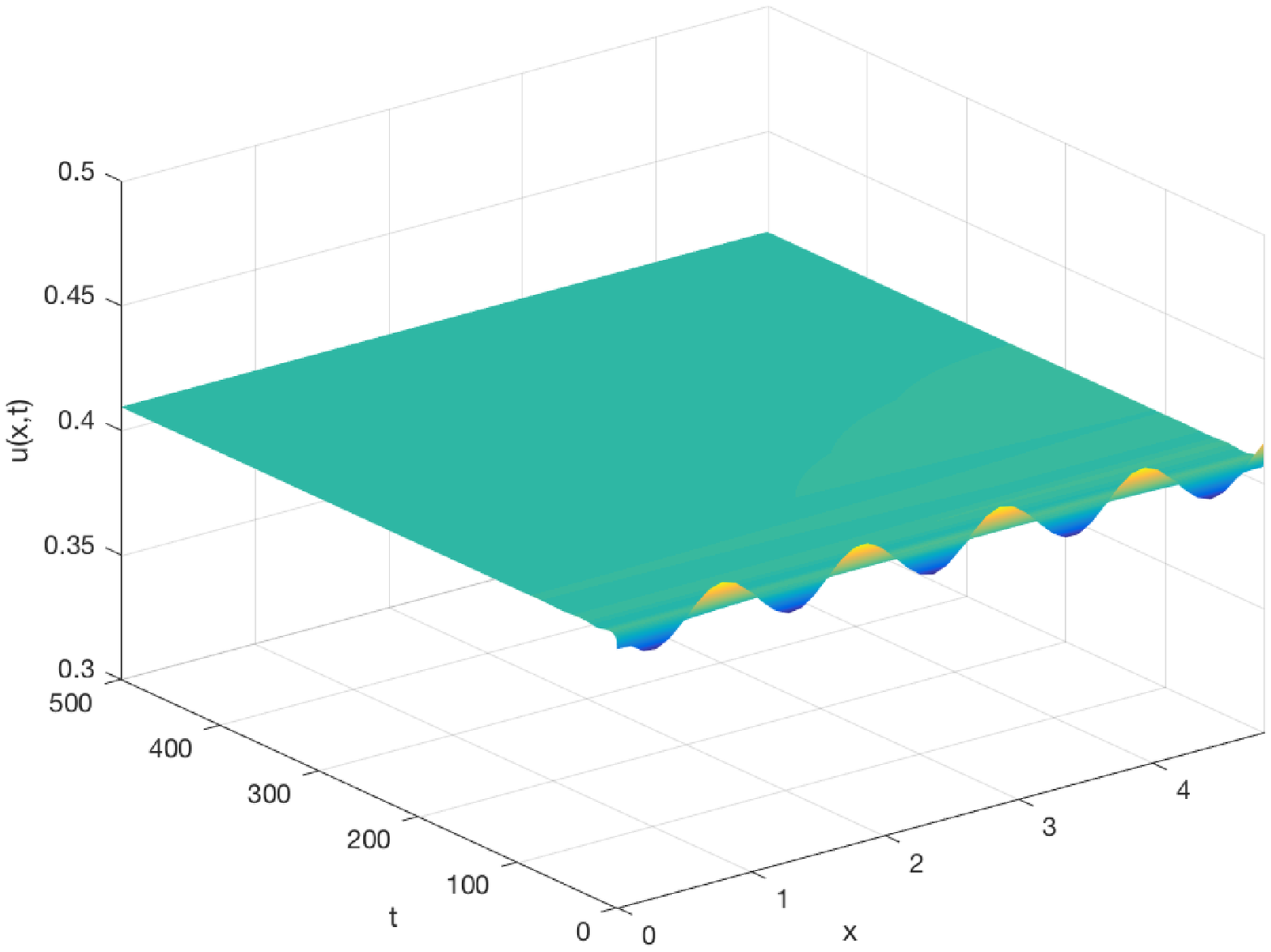}            
	\end{minipage}}
	\subfigure[$v(t,x)$]{                  
		\begin{minipage}{0.48\linewidth} 
			\centering                                     
			\includegraphics[scale=0.33]{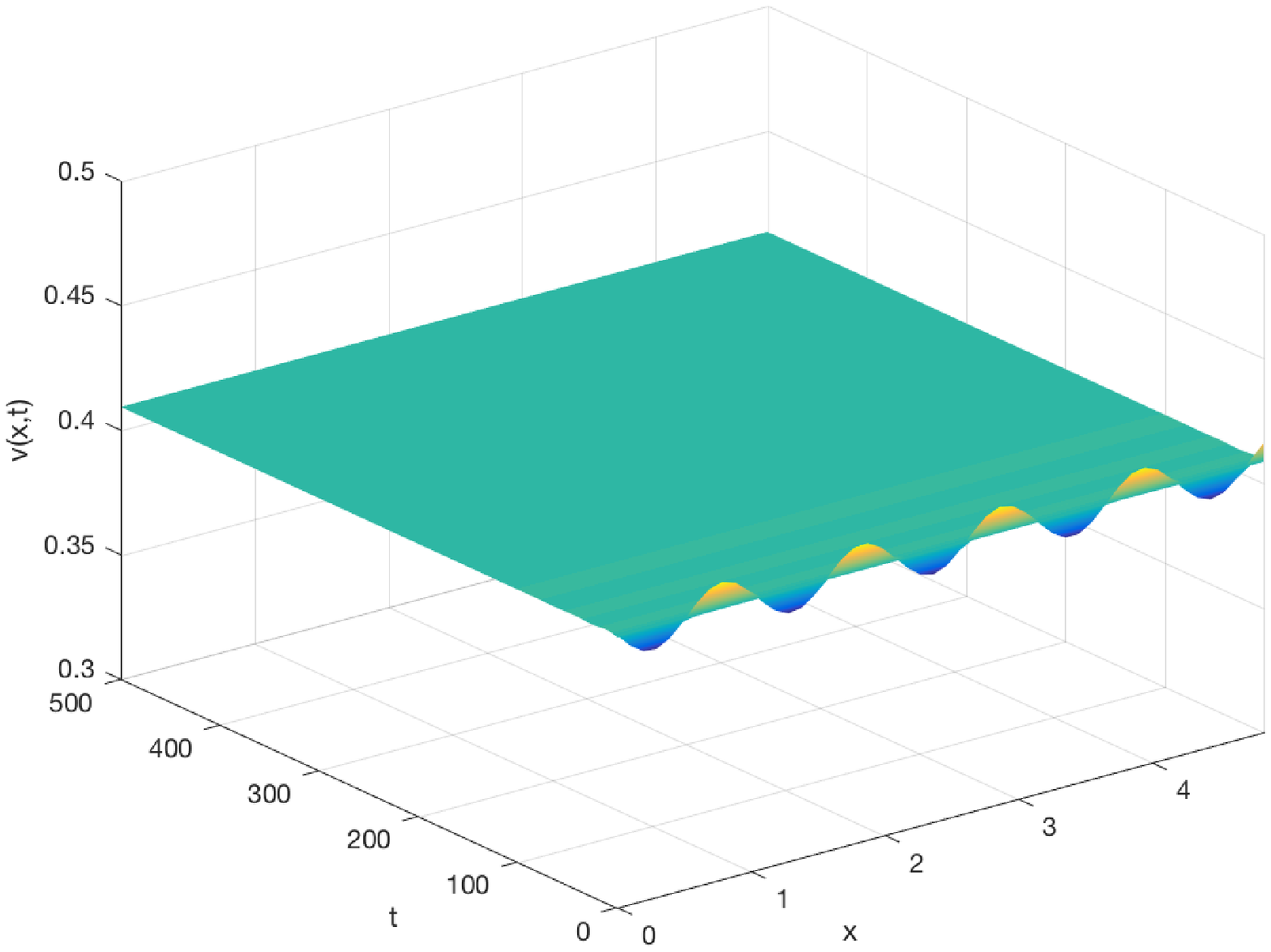}                
		\end{minipage}
	}\caption{A stable constant steady state $(u_0,v_0)$ in {$\mathbf D_1$}, with $(\alpha_1,\alpha_2)=(0.0373,-0.0543)$ and the initial functions are $u_0(x)=v_0(x)=u_0+0.01\sin 6x$.} 
	\label{fig1}                                                    
\end{figure}

As the parameters pass through the pitchfork line $L_1$ of $(u_0,v_0)$ from $D_1$ to $D_2$, a stable spatially homogeneous periodic solution is generated  and the equilibrium loses its stability at the same time. In Figure \ref{fig2}, $(\alpha_1,\alpha_2)$ are chosen in $D_2$, and the stable spatially homogeneous periodic solution are shown.
\begin{figure}[htbp]
	\centering                           
	\subfigure[$u(t,x)$]{       
		\begin{minipage}{0.47\linewidth} 
			\centering                                      
			\includegraphics[scale=0.3]{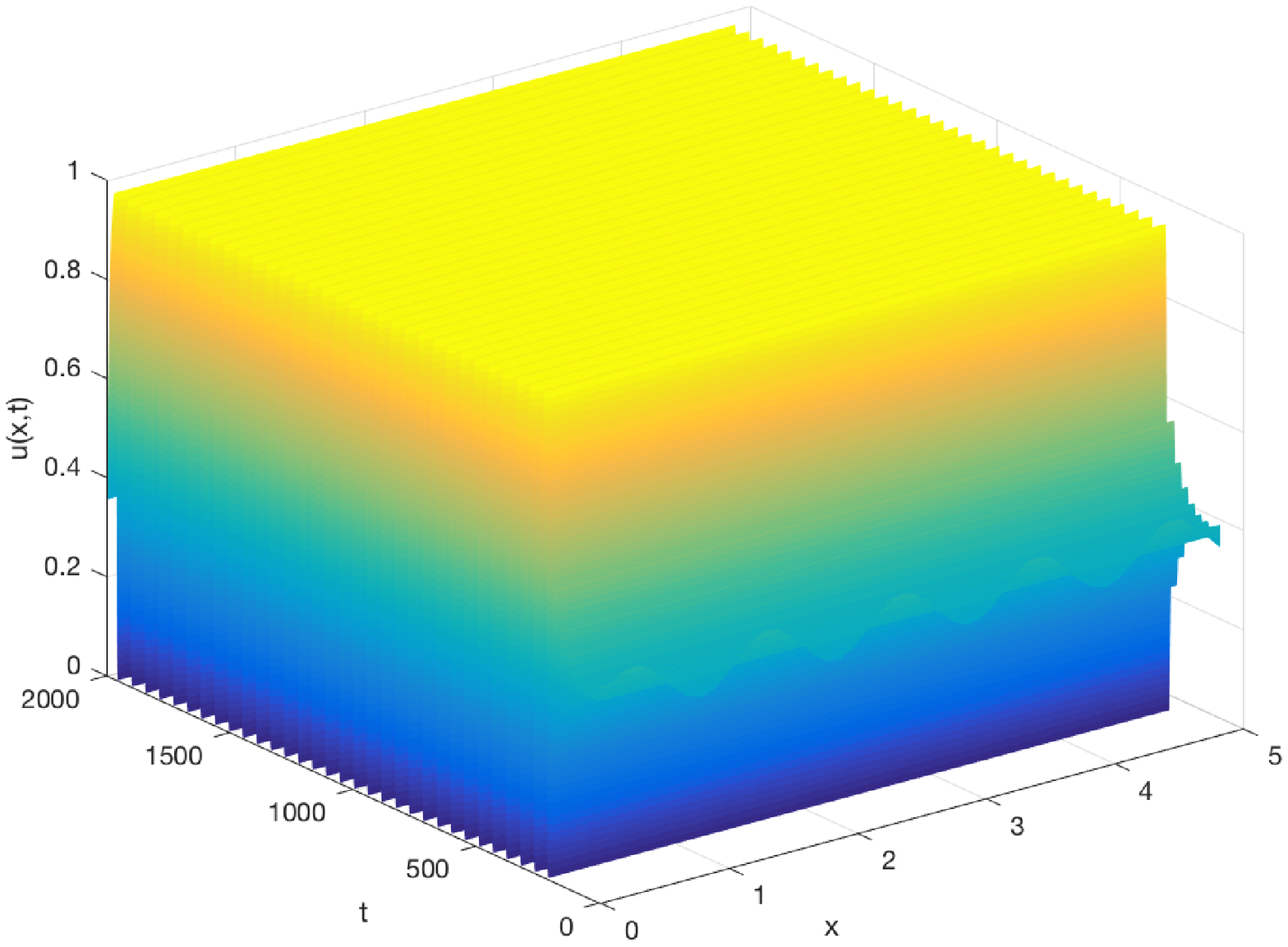}            
	\end{minipage}}
	\subfigure[$v(t,x)$]{                  
		\begin{minipage}{0.47\linewidth} 
			\centering                                     
			\includegraphics[scale=0.3]{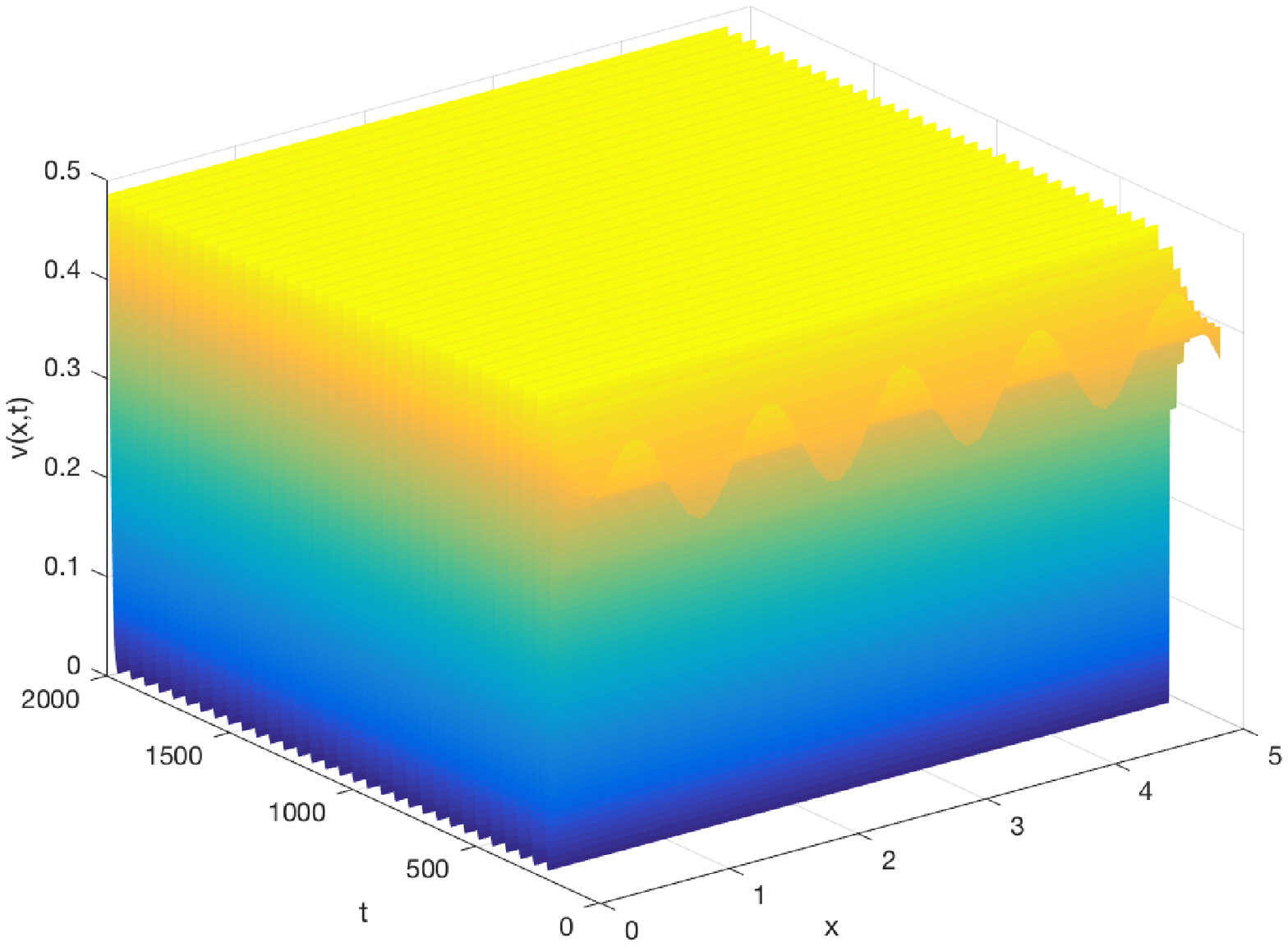}                
	\end{minipage}}
	\caption{A stable spatially homogeneous periodic solution in {$\mathbf D_2$},  with $(\alpha_1,\alpha_2)=(-0.0344,-0.0578)$ and  initial functions are $u_0(x)=v_0(x)=u_0+0.05\sin6x$.} 
	\label{fig2}                                                    
\end{figure}

 In $D_3$, the two unstable non-constant steady states  disappeared due to the existence of another pitchfork bifurcation curve $L_2$ of $(u_0,v_0)$. The  spatially homogeneous periodic solution is still a stable attractor of \eqref{eqA}. We simulate the dynamics with $(\alpha_1,\alpha_2)\in D_3$ in Figure \ref{fig3}.
 \begin{figure}[htbp]
 	\centering                           
 	\subfigure[$u(t,x)$]{       
 		\begin{minipage}{0.48\linewidth} 
 			\centering                                      
 			\includegraphics[scale=0.3]{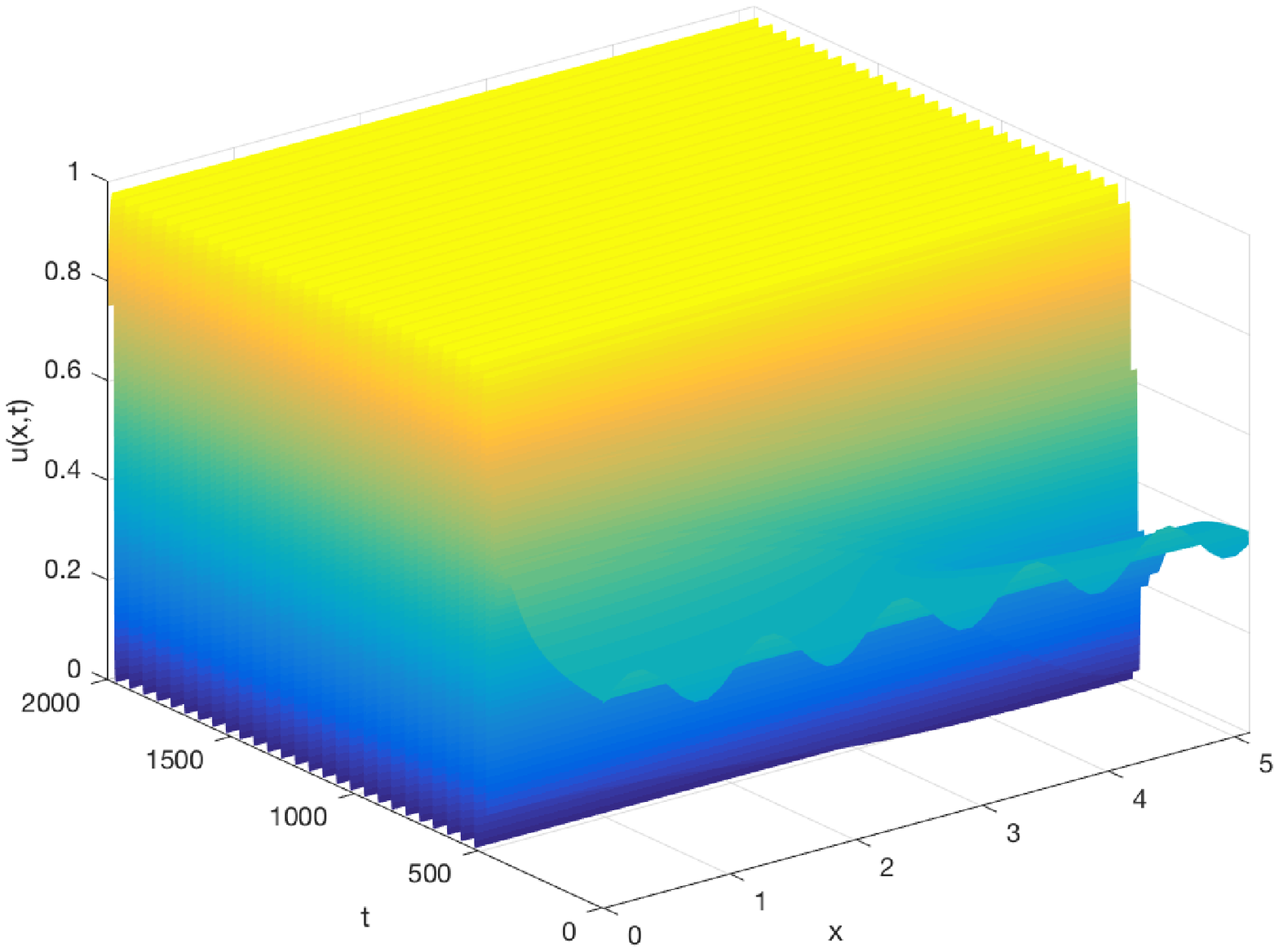}            
 	\end{minipage}}
 	\subfigure[$v(t,x)$]{                  
 		\begin{minipage}{0.48\linewidth} 
 			\centering                                     
 			\includegraphics[scale=0.3]{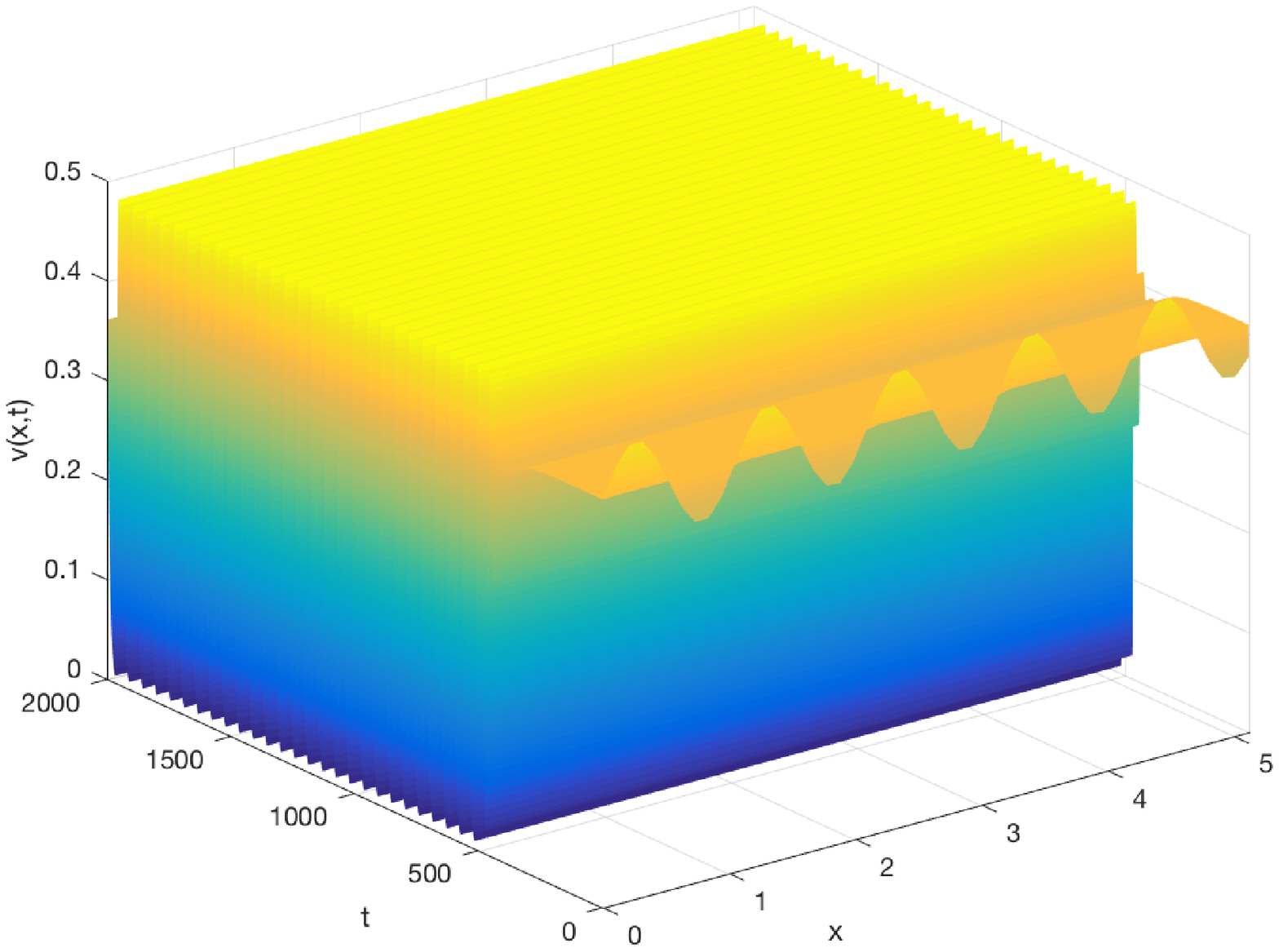}                
 		\end{minipage}
 	}\caption{A stable spatially homogeneous periodic solution in {$\mathbf D_3$}, with $(\alpha_1,\alpha_2)=(-0.0325, 0.0356)$ and the initial functions are $u_0(x)=v_0(x)=u_0+0.05\sin6x$.} 
 	\label{fig3}                                                    
 \end{figure}

  \begin{figure}[htbp]
  	\centering                           
  	\subfigure[$u(t,x)$]{       
  		\begin{minipage}{0.48\linewidth} 
  			\centering                                      
  			\includegraphics[scale=0.33]{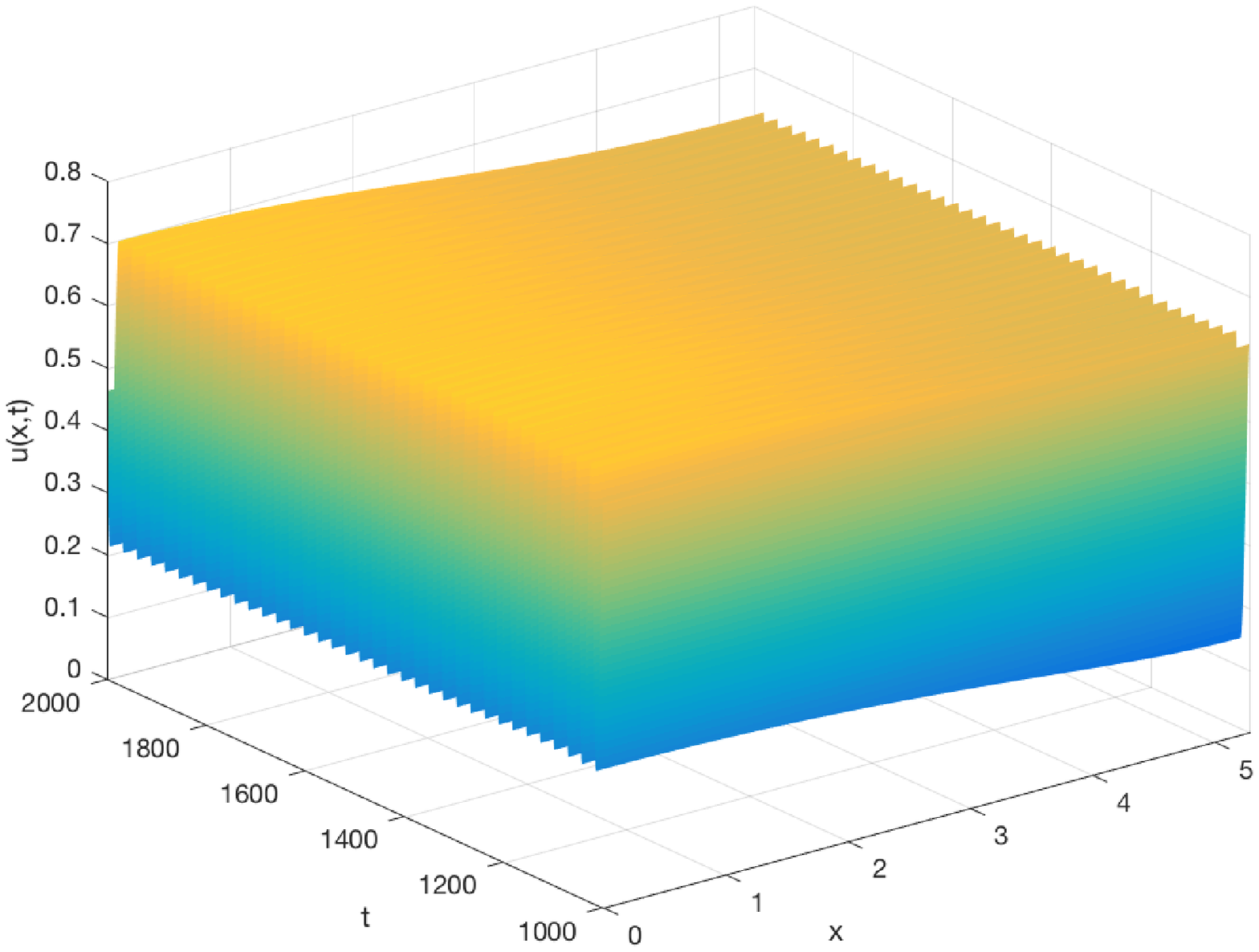}            
  	\end{minipage}}
  	\subfigure[$v(t,x)$]{                  
  		\begin{minipage}{0.48\linewidth} 
  			\centering                                     
  			\includegraphics[scale=0.33]{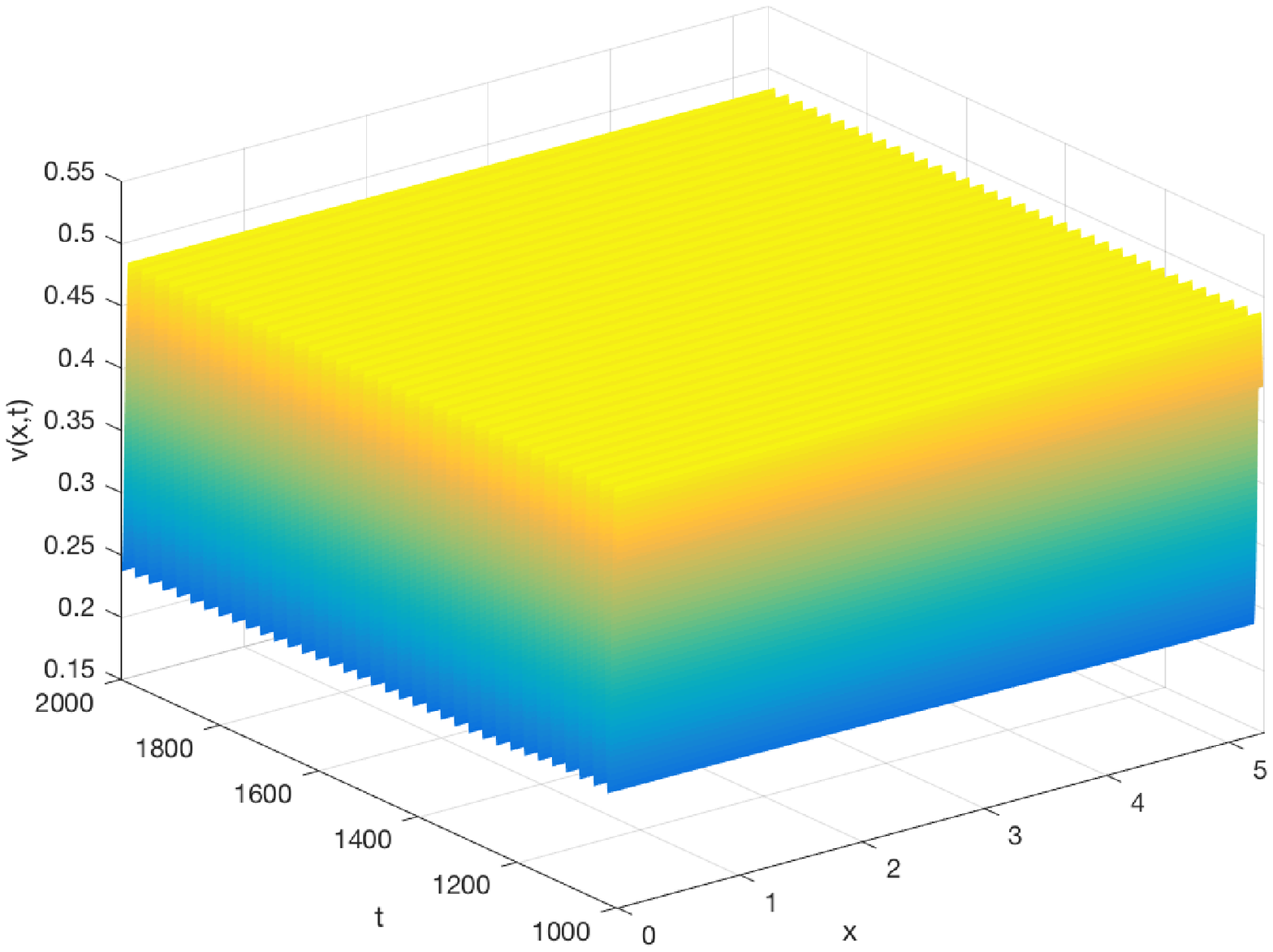}                
  	\end{minipage}}
\caption{A stable spatially non-homogeneous periodic solution in {$\mathbf D_4$}, with $(\alpha_1,\alpha_2)=(-0.0030, 0.0888)$ and the initial value functions are $u_0(x)=v_0(x)=u_0+0.05\sin2x$.} 
\label{fig4-1}                                                    
\end{figure}
 With the parameters $(\alpha_1,\alpha_2)$ move to $D_4$ and pass through the curve $L_3$, the system \eqref{eqA} undergoes a pitchfork bifurcation at the spatially homogeneous periodic solution. The directly result is two symmetric stable spatially non-homogeneous periodic solutions are emerged in $D_4$, while the spatially homogeneous periodic solution loses its stability. In Figure \ref{fig4-1}- Figure \ref{fig4-2}, parameters are chosen in $D_4$, and we find two spatially non-homogeneous periodic solutions coexist with the spatial amplitude is not very large.
\begin{figure}[htbp]
  	\subfigure[$u(t,x)$]{       
  		\begin{minipage}{0.48\linewidth} 
  			\centering                                      
  			\includegraphics[scale=0.33]{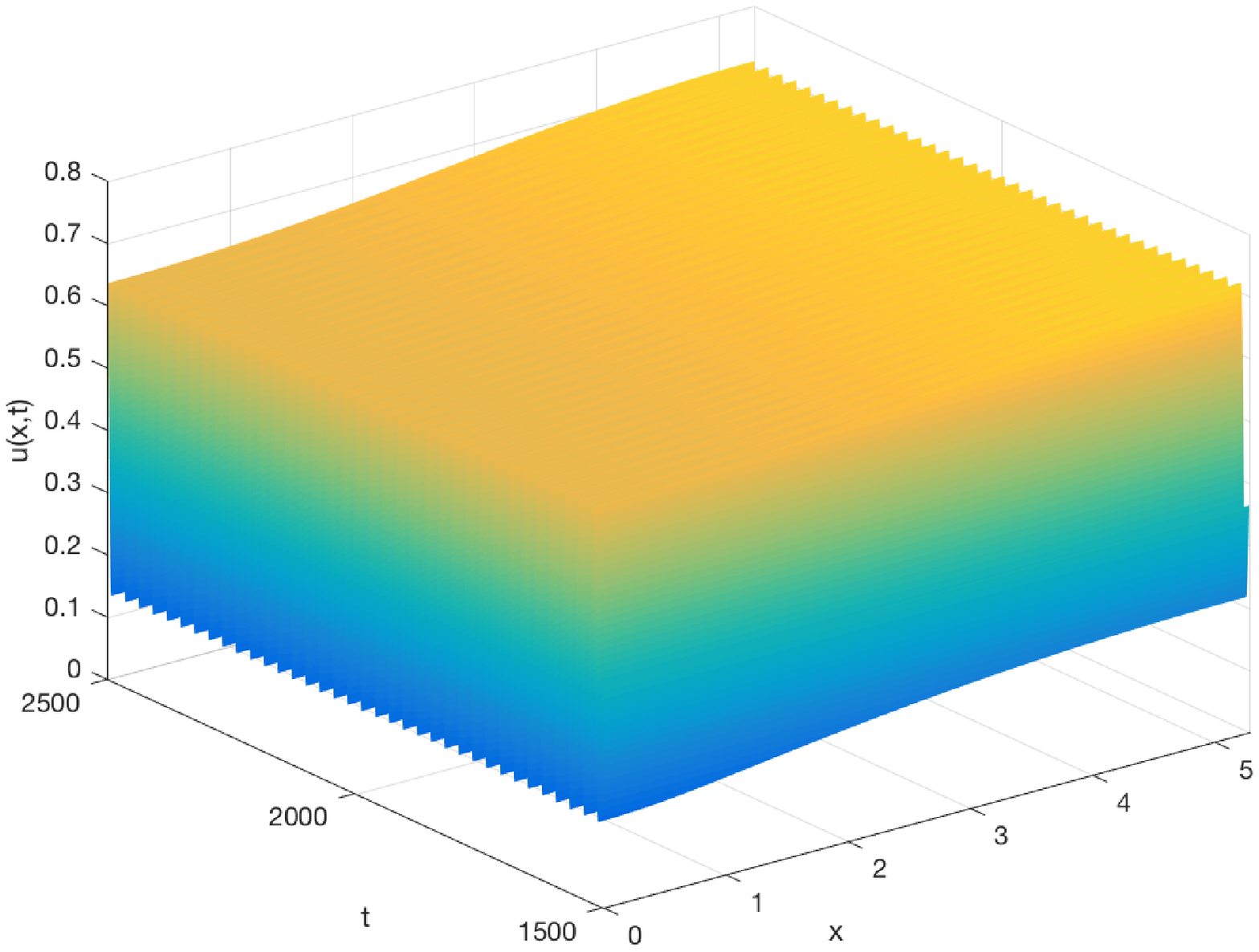}            
  	\end{minipage}}
  	\subfigure[$v(t,x)$]{                  
  		\begin{minipage}{0.48\linewidth} 
  			\centering                                     
  			\includegraphics[scale=0.33]{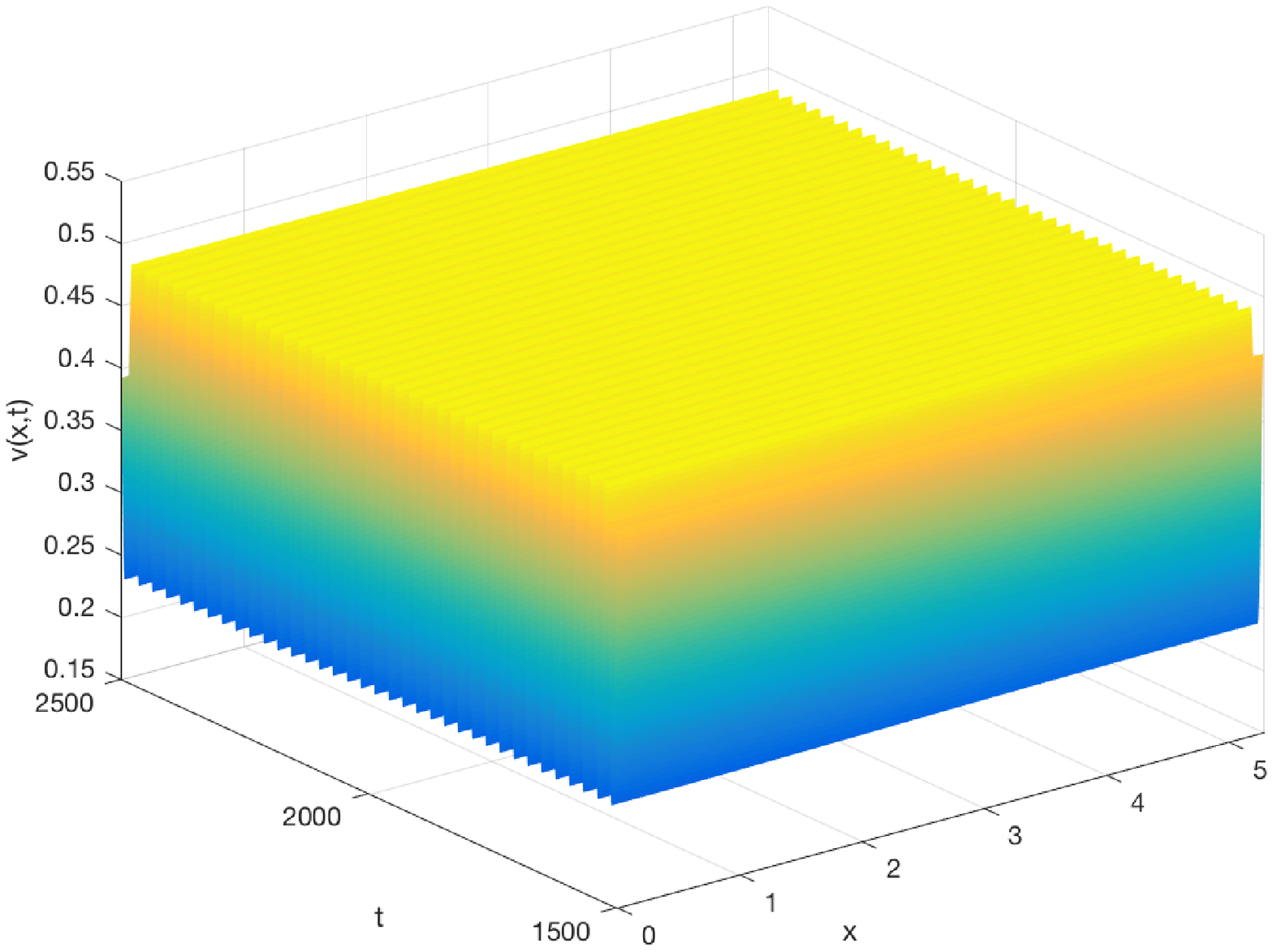}                
  		\end{minipage}
  	}
\caption{A stable spatially non-homogeneous periodic solution in {$\mathbf D_4$}, with $(\alpha_1,\alpha_2)=(-0.0030, 0.0888)$ and the initial value functions are $u_0(x)=v_0(x)=u_0-0.05\sin2x$.} 
  	\label{fig4-2}                                                    
  \end{figure}
  \begin{figure}[htbp]
	\centering                           
	\subfigure[$u(t,x)$]{       
		\begin{minipage}{0.48\linewidth} 
			\centering                                      
			\includegraphics[scale=0.33]{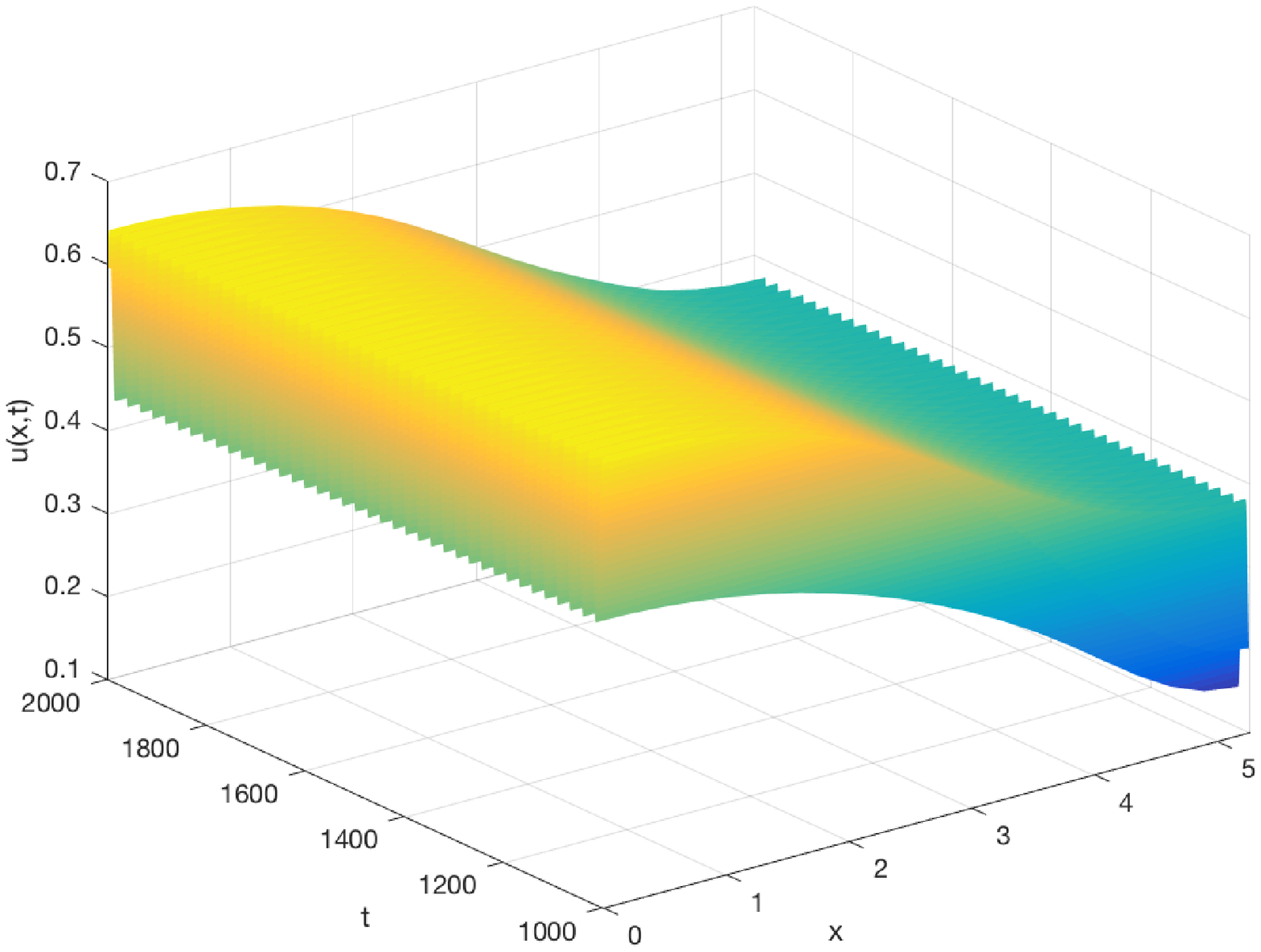}            
	\end{minipage}}
	\subfigure[$v(t,x)$]{                  
		\begin{minipage}{0.48\linewidth} 
			\centering                                     
			\includegraphics[scale=0.33]{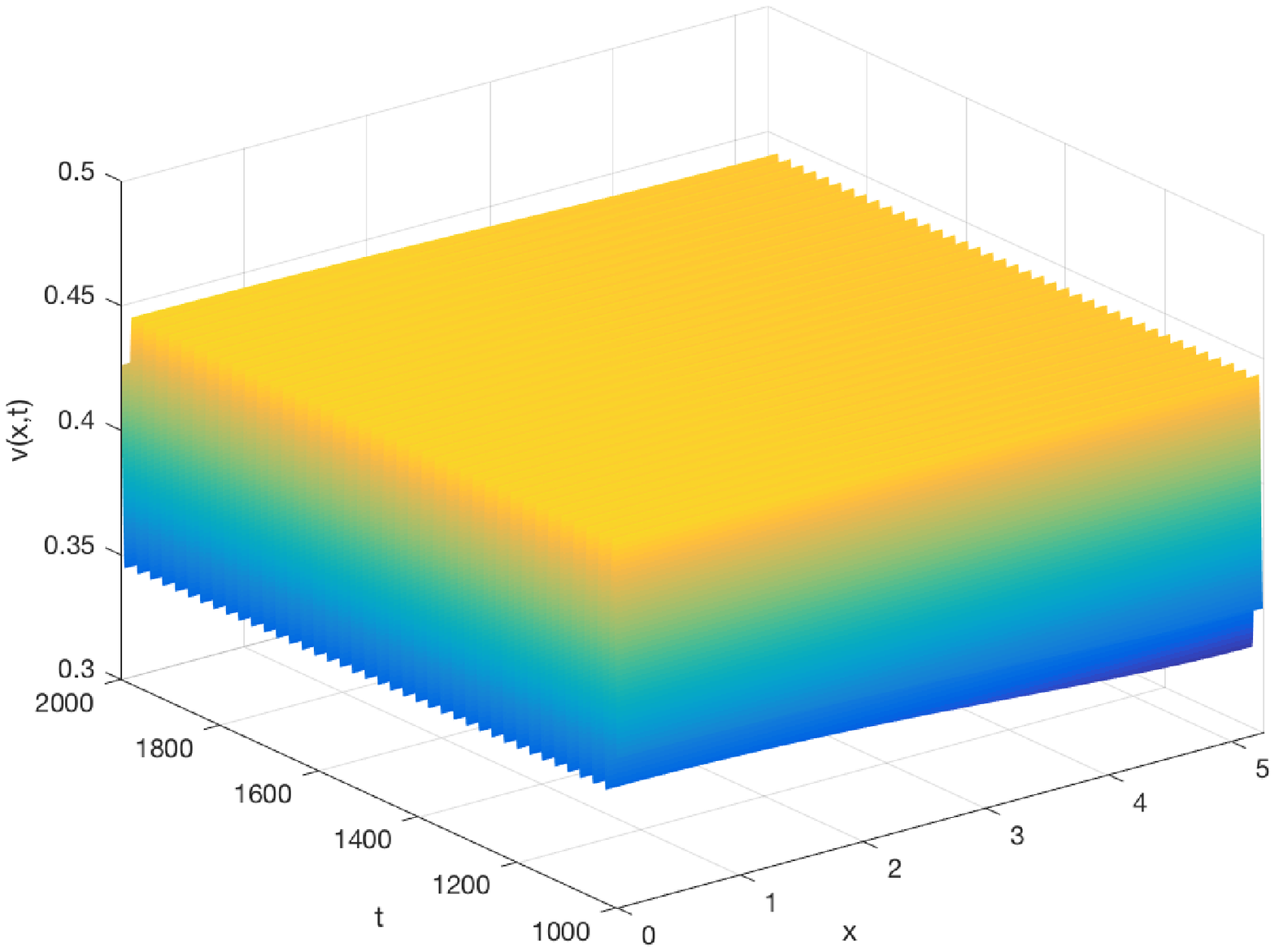}                
		\end{minipage}
	}
	\caption{A stable spatially non-homogeneous periodic solution in {$\mathbf D_5$}, with $(\alpha_1,\alpha_2)=(0.0352,0.0817)$ and the initial value functions are $u_0(x)=v_0(x)=u_0+0.05\sin2x$.}
	\label{fig5-1}
\end{figure} 

\begin{figure}[htbp]	
	\subfigure[$u(t,x)$]{       
		\begin{minipage}{0.48\linewidth} 
			\centering                                      
			\includegraphics[scale=0.33]{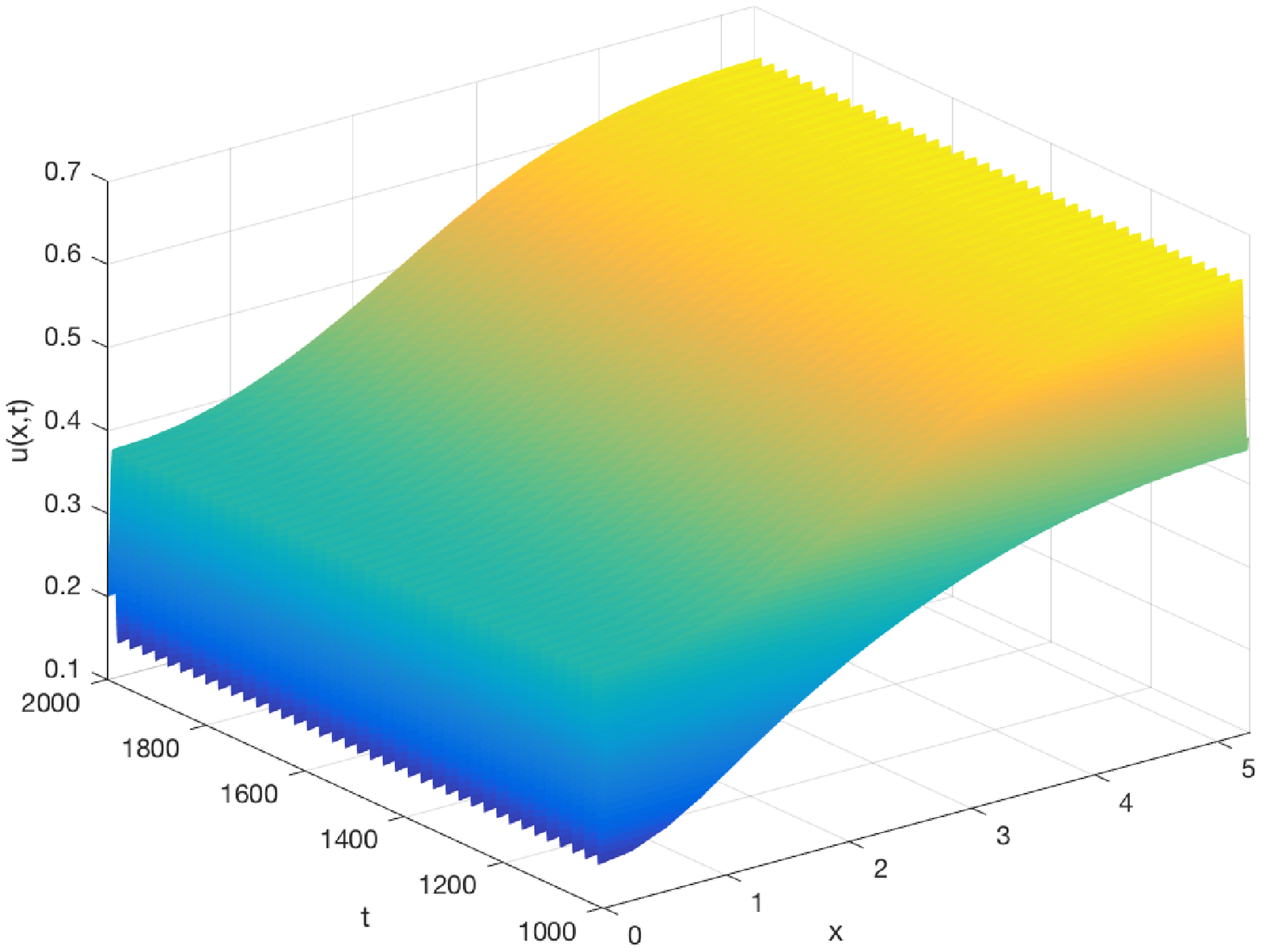}            
	\end{minipage}}
	\subfigure[$v(t,x)$]{                  
		\begin{minipage}{0.48\linewidth} 
			\centering                                     
			\includegraphics[scale=0.33]{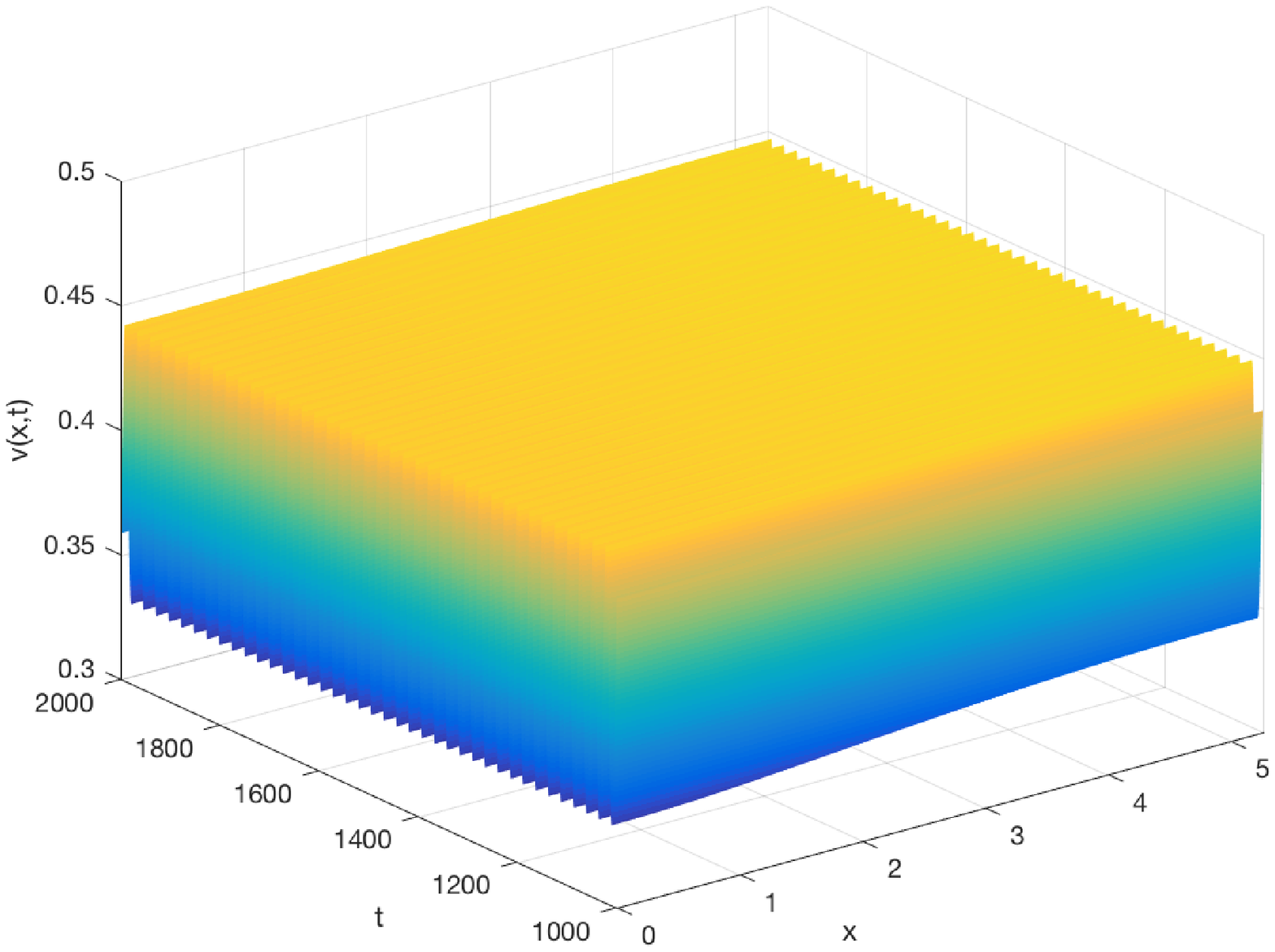}                
		\end{minipage}
	}\caption{A stable spatially non-homogeneous periodic solution in {$\mathbf D_5$}, with $(\alpha_1,\alpha_2)=(0.0352,0.0817)$ and the initial value functions are $u_0(x)=v_0(x)=u_0-0.05\sin2x$.}
	\label{fig5-2}
\end{figure} 
 In $D_5$, the unstable spatially homogeneous periodic solution disappeared, once again, because the existence of the pitchfork line $L_4$ of $(u_0,v_0)$. The  symmetric spatially non-homogeneous periodic solutions are still stable and we shown them in Figure \ref{fig5-1}- Figure \ref{fig5-2}. Compared with Figure \ref{fig4-1} Figure-\ref{fig4-2}, the spatial amplitude of the solutions is lager and the oscillation about time becomes smaller in  Figure \ref{fig5-1}- Figure\ref{fig5-2}.
 
  $L_5$ is a Hopf curve of the spatially non-homogeneous quasi-periodic solutions.  As a result,  two symmetric stable spatially non-homogeneous quasi-periodic solutions are bifurcated from the spatially non-homogeneous periodic solutions in $D_6$. Chosen $(\alpha_1,\alpha_2)\in D_6$, two spatially non-homogeneous quasi-periodic solutions are found in Figure  \ref{fig6-1}- Figure\ref{fig6-2}.
  \begin{figure}[htbp]
  	\centering                           
  	\subfigure[$u(t,x)$]{       
  		\begin{minipage}{0.48\linewidth} 
  			\centering                                      
  			\includegraphics[scale=0.33]{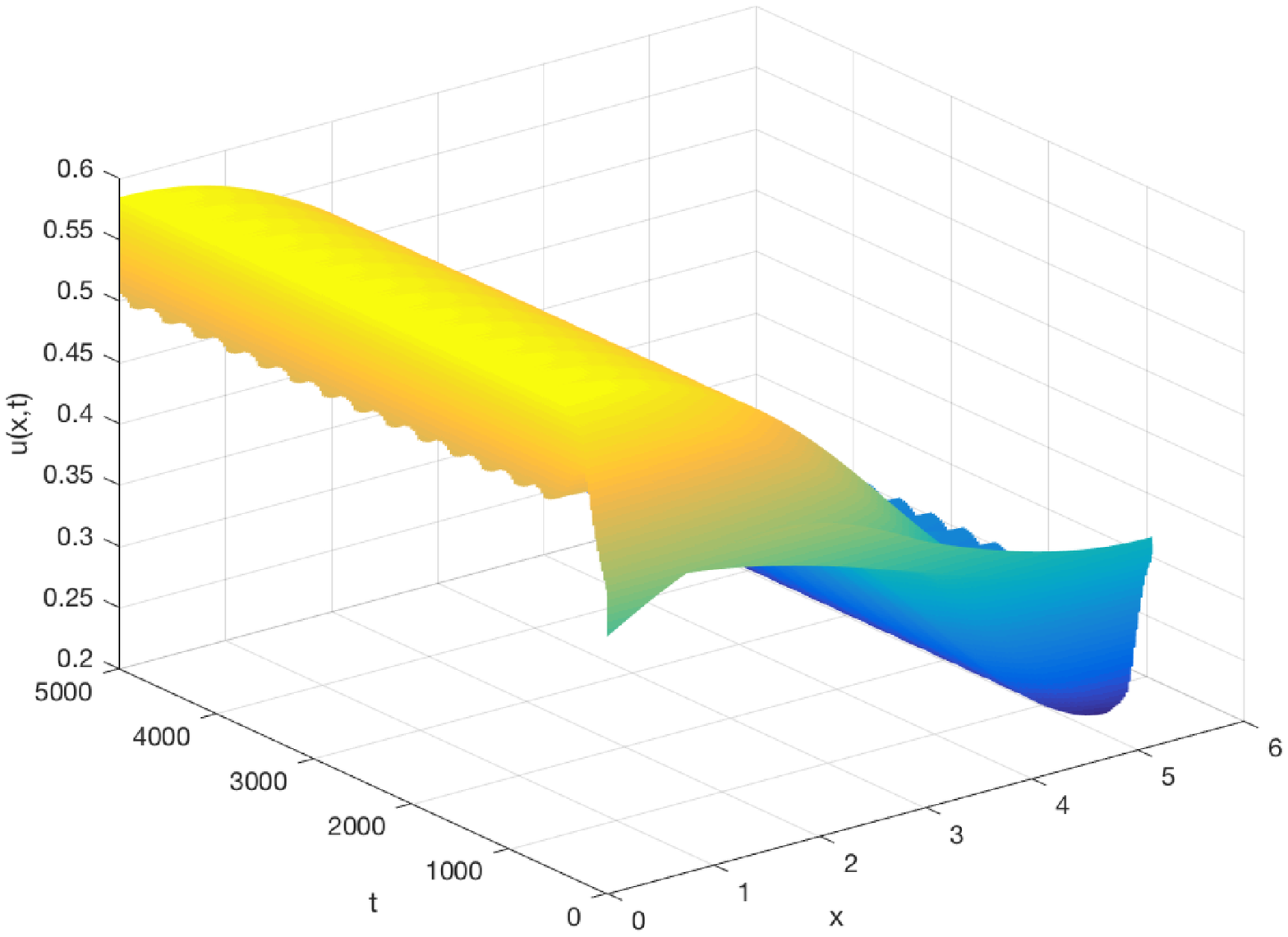}            
  	\end{minipage}}
  	\subfigure[$v(t,x)$]{       
  		\begin{minipage}{0.48\linewidth} 
  			\centering                                      
  			\includegraphics[scale=0.33]{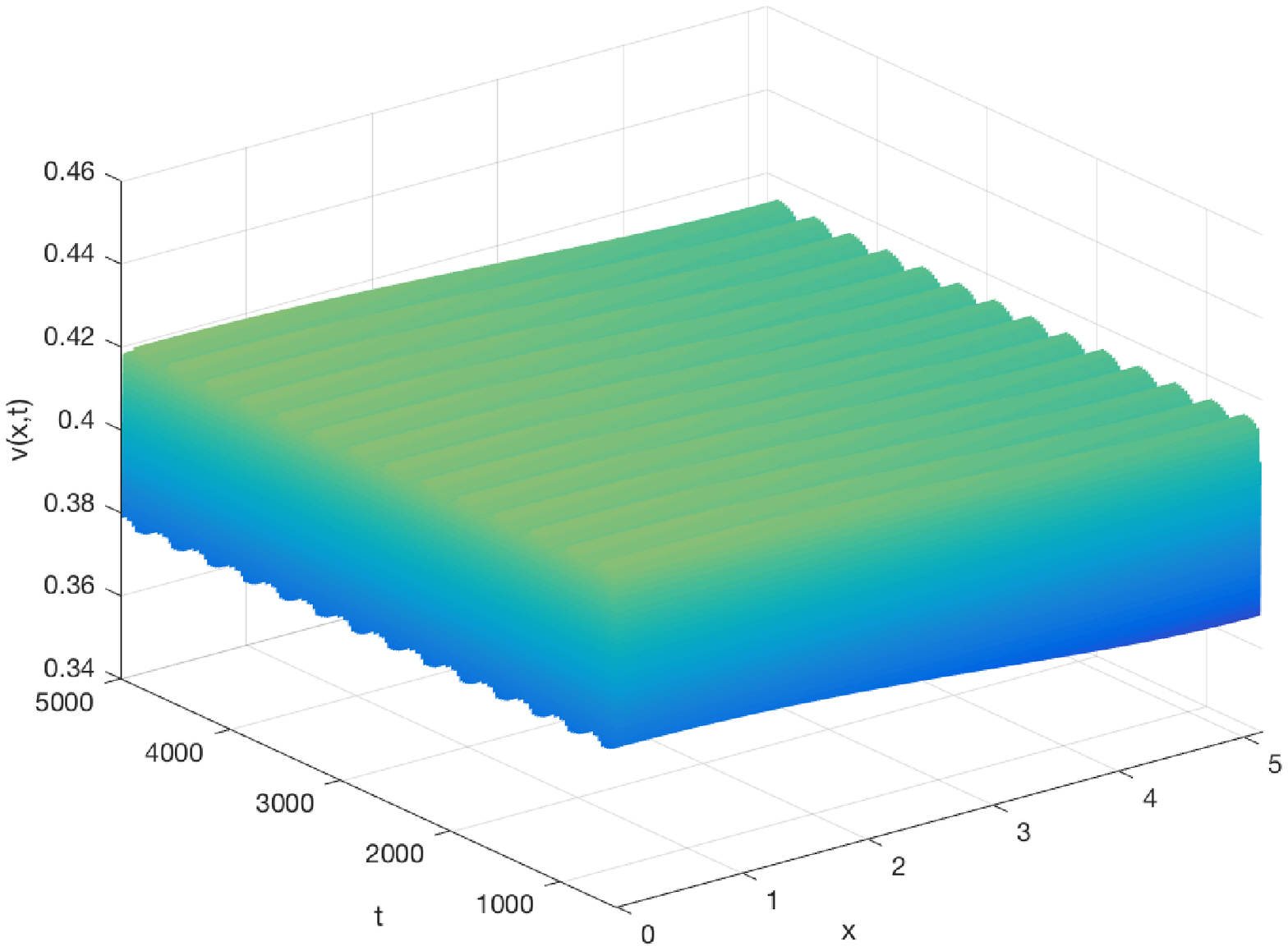}            
  	\end{minipage}}
  
  \subfigure[$u(t,0)$]{                  
  	\begin{minipage}{0.48\linewidth} 
  		\centering                                     
  		\includegraphics[width = 1.\textwidth,height = 0.4 \textwidth]{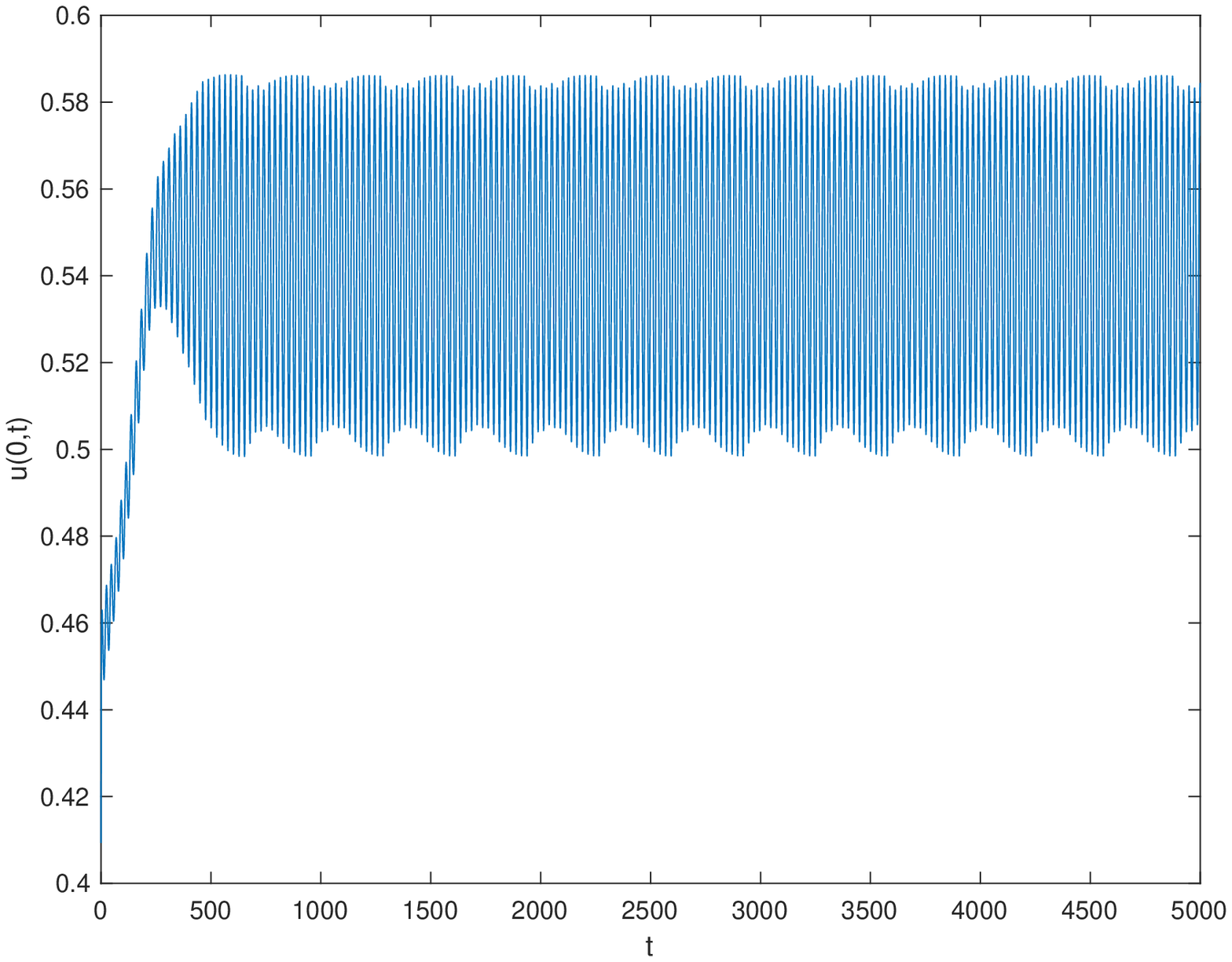}               
  \end{minipage}}
  	\subfigure[$v(t,0)$]{                  
  		\begin{minipage}{0.48\linewidth} 
  			\centering                                     
  			\includegraphics[width = 1.\textwidth,height = 0.4 \textwidth]{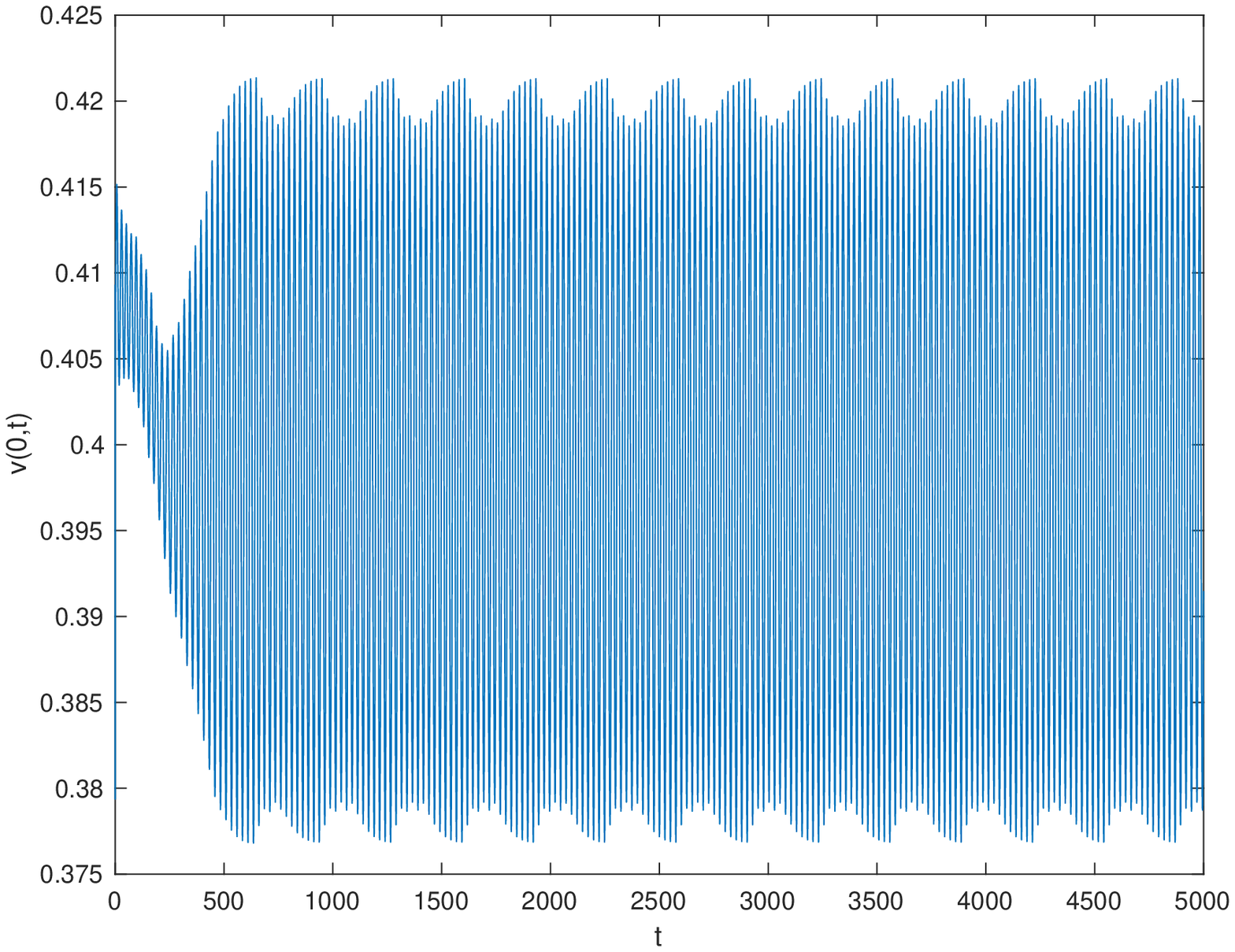}                
  	\end{minipage}}
  	\caption{A stable spatially non-homogeneous quasi-periodic solution in {$\mathbf D_6$}, with $(\alpha_1,\alpha_2)=(0.0405,0.0449)$ and the initial value functions are $u_0(x) = u_0+0.05\sin x,\; v_0(x)=u_0-0.05\sin x$.}
  	\label{fig6-1}
  \end{figure}
  \begin{figure}[htbp]
  	\centering                           
  	\subfigure[$u(t,x)$]{       
  		\begin{minipage}{0.48\linewidth} 
  			\centering                                      
  			\includegraphics[scale=0.33]{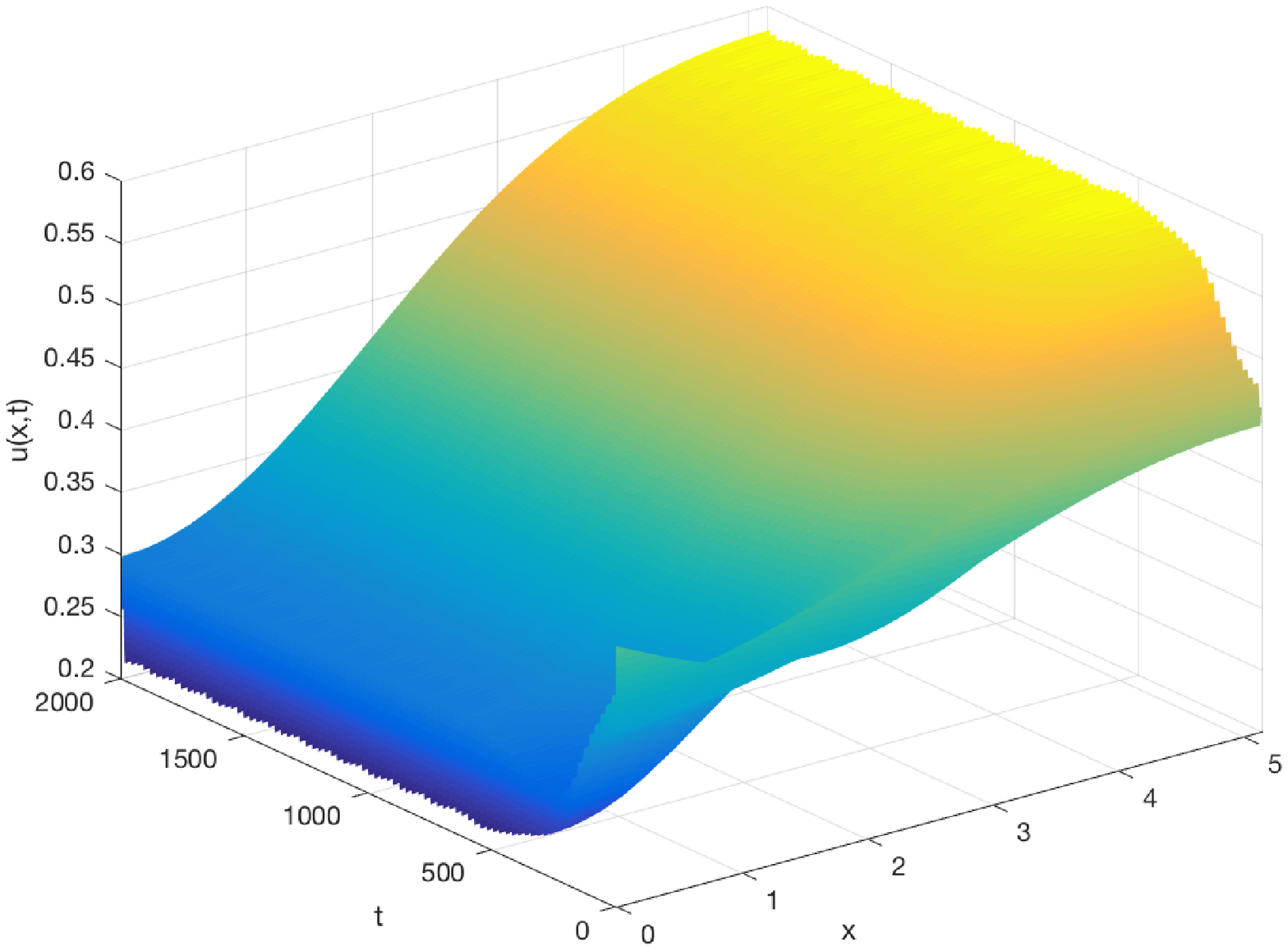}            
  	\end{minipage}}
  	\subfigure[$v(t,x)$]{       
  		\begin{minipage}{0.48\linewidth} 
  			\centering                                      
  			\includegraphics[scale=0.33]{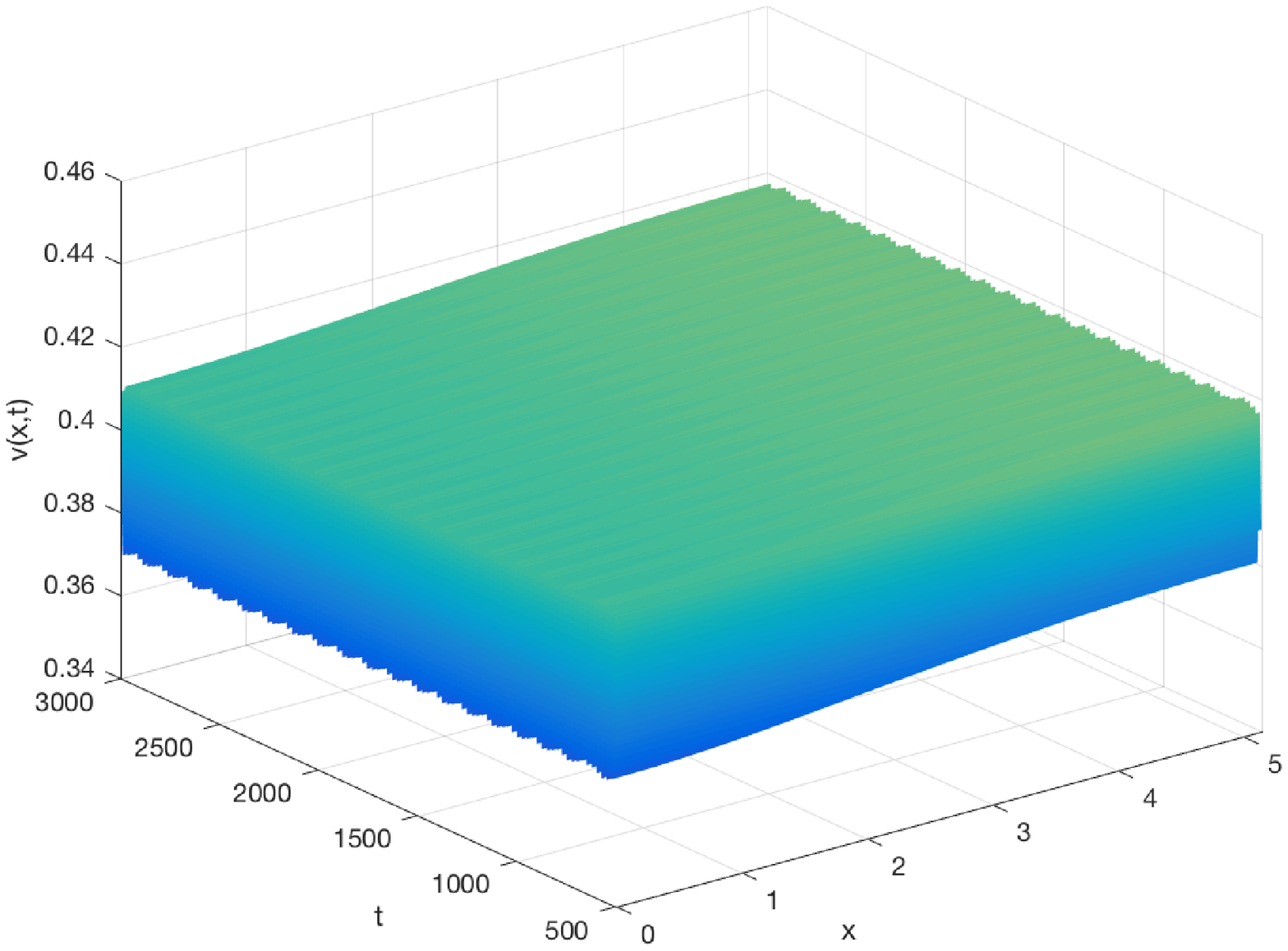}            
  	\end{minipage}}
  
  \subfigure[$u(t,0)$]{                  
  	\begin{minipage}{0.48\linewidth} 
  		\centering                                     
  		\includegraphics[width = 1.\textwidth,height = 0.4 \textwidth]{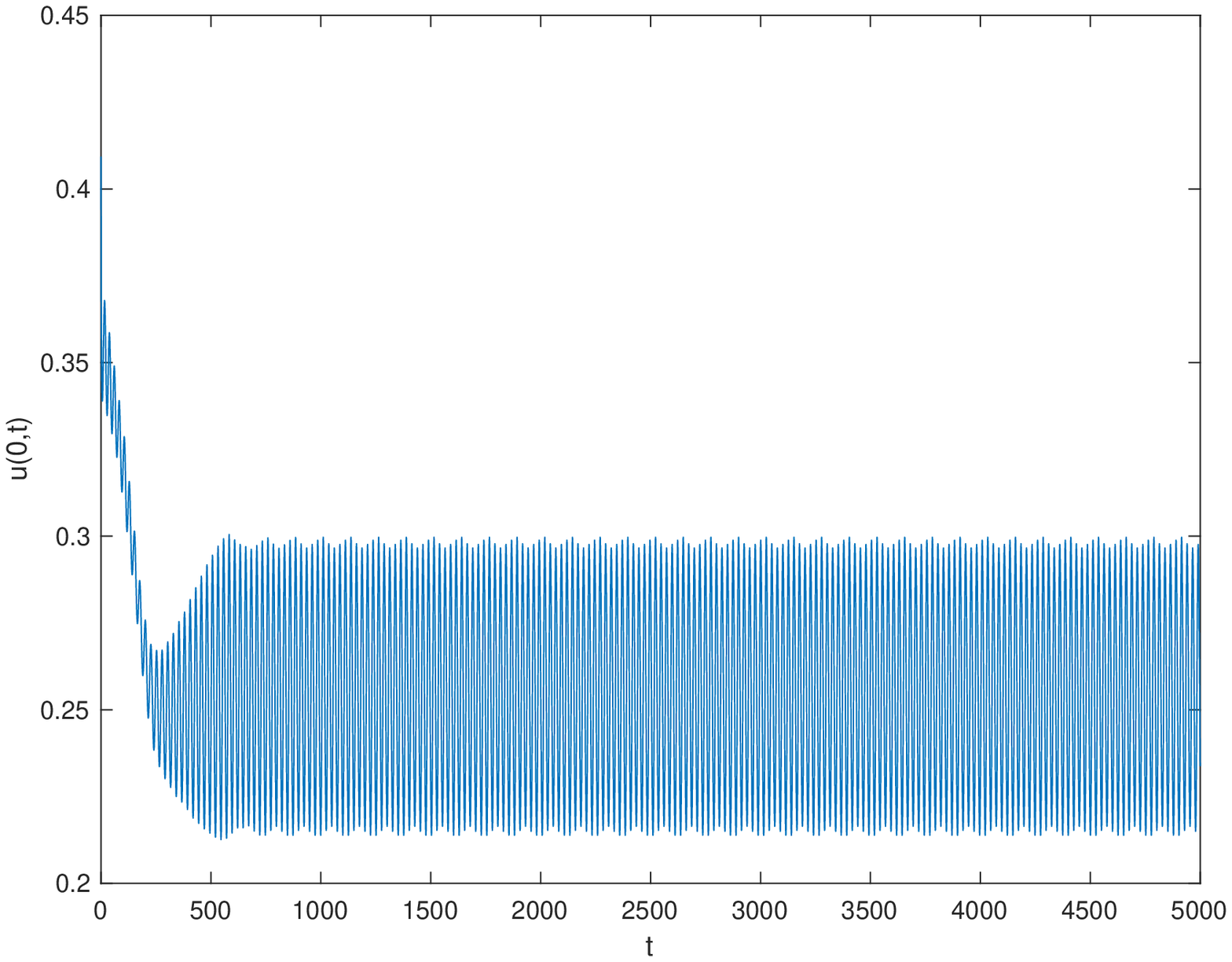}                
  	\end{minipage}}
  	\subfigure[$v(t,0)$]{                  
  		\begin{minipage}{0.48\linewidth} 
  			\centering                                     
  			\includegraphics[width = 1.\textwidth,height = 0.4 \textwidth]{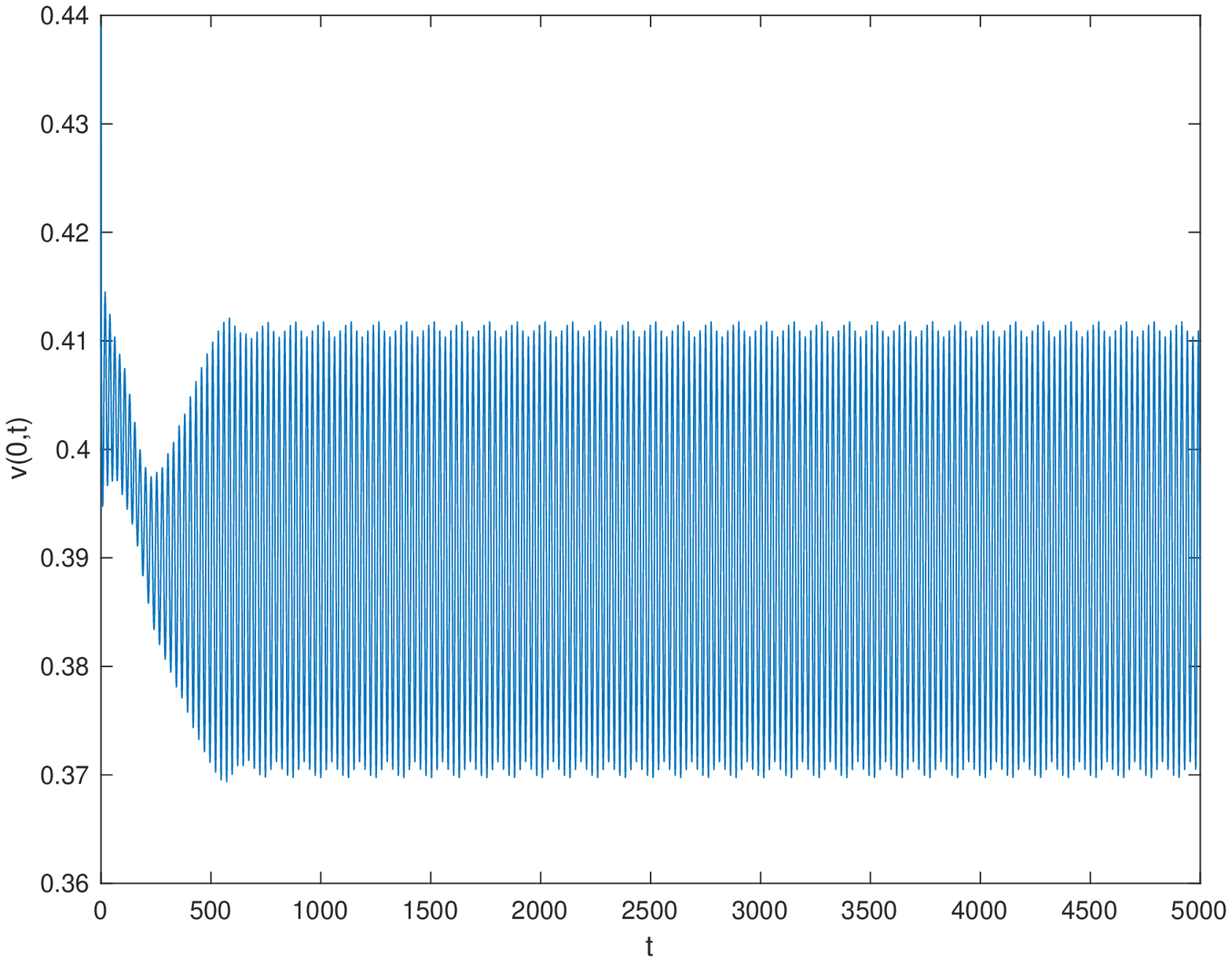}                
  	\end{minipage}}
  	\caption{A stable spatially non-homogeneous quasi-periodic solution in {$\mathbf D_6$}, with $(\alpha_1,\alpha_2)=(0.0405,0.0449)$ and the initial value functions are $u_0(x) = u_0-0.05\sin x,\; v_0(x)=u_0+0.05\sin x$.}
  	\label{fig6-2}
  \end{figure}

In $D_7$, the constant steady state $(u_0,v_0)$ becomes stable. Meanwhile two non-constant steady states appear and both of them are saddle points. The reason is also the pitchfork bifurcation of $(u_0,v_0)$ occurs at $L_6$. It is worth noting that, we also observed the existence of the spatially non-homogeneous quasi-periodic solutions in  the corresponding  numerical experiments. Which means there are three possible attractors coexist in the Holling-Tanner system \eqref{eqA} with parameters $(\alpha_1,\alpha_2)\in D_6$ and close to origin. We show these dynamical behaviors in Figure \ref{fig7-1}-Figure \ref{fig7-3}.
\begin{figure}[htbp]
	\centering                           
	\subfigure[$u(t,x)$]{       
		\begin{minipage}{0.48\linewidth} 
			\centering                                      
			\includegraphics[scale=0.33]{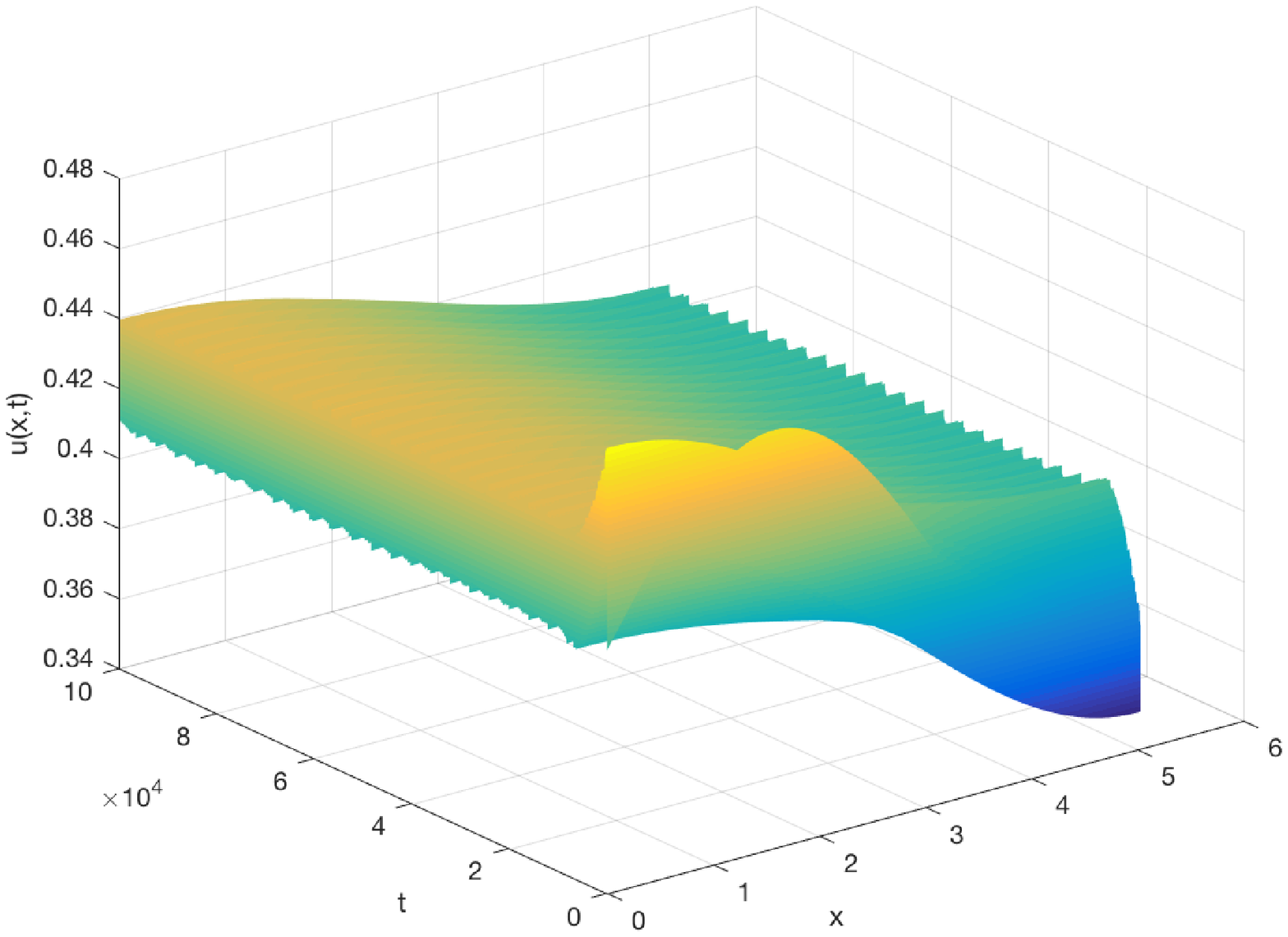}            
	\end{minipage}}
	\subfigure[$v(t,x)$]{       
		\begin{minipage}{0.48\linewidth} 
			\centering                                      
			\includegraphics[scale=0.33]{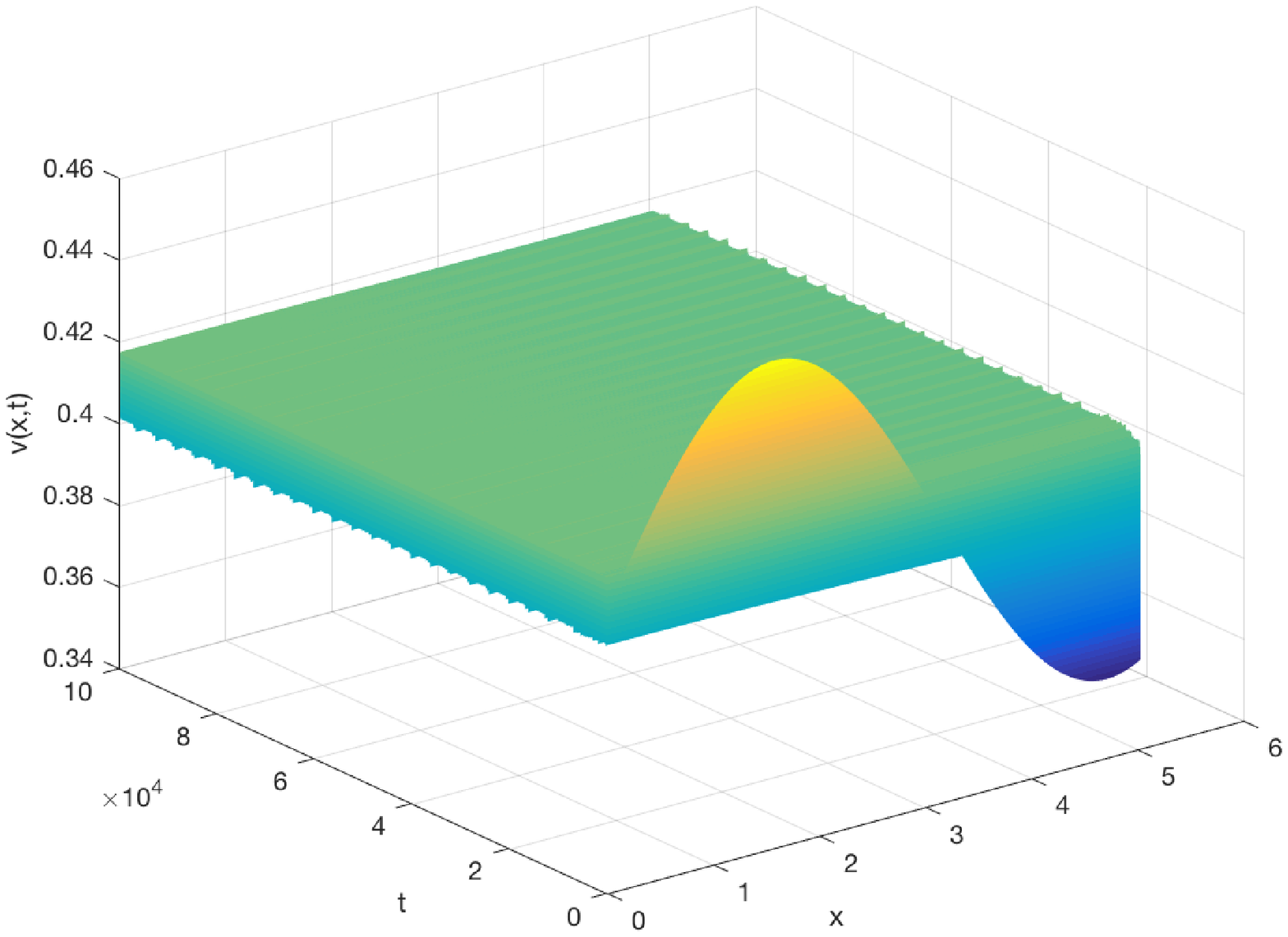}            
	\end{minipage}}
\vskip -0.3cm
	\subfigure[$u(t,0)$]{                  
	\begin{minipage}{0.48\linewidth} 
		\centering                                     
		\includegraphics[width = 1.\textwidth,height = 0.4 \textwidth]{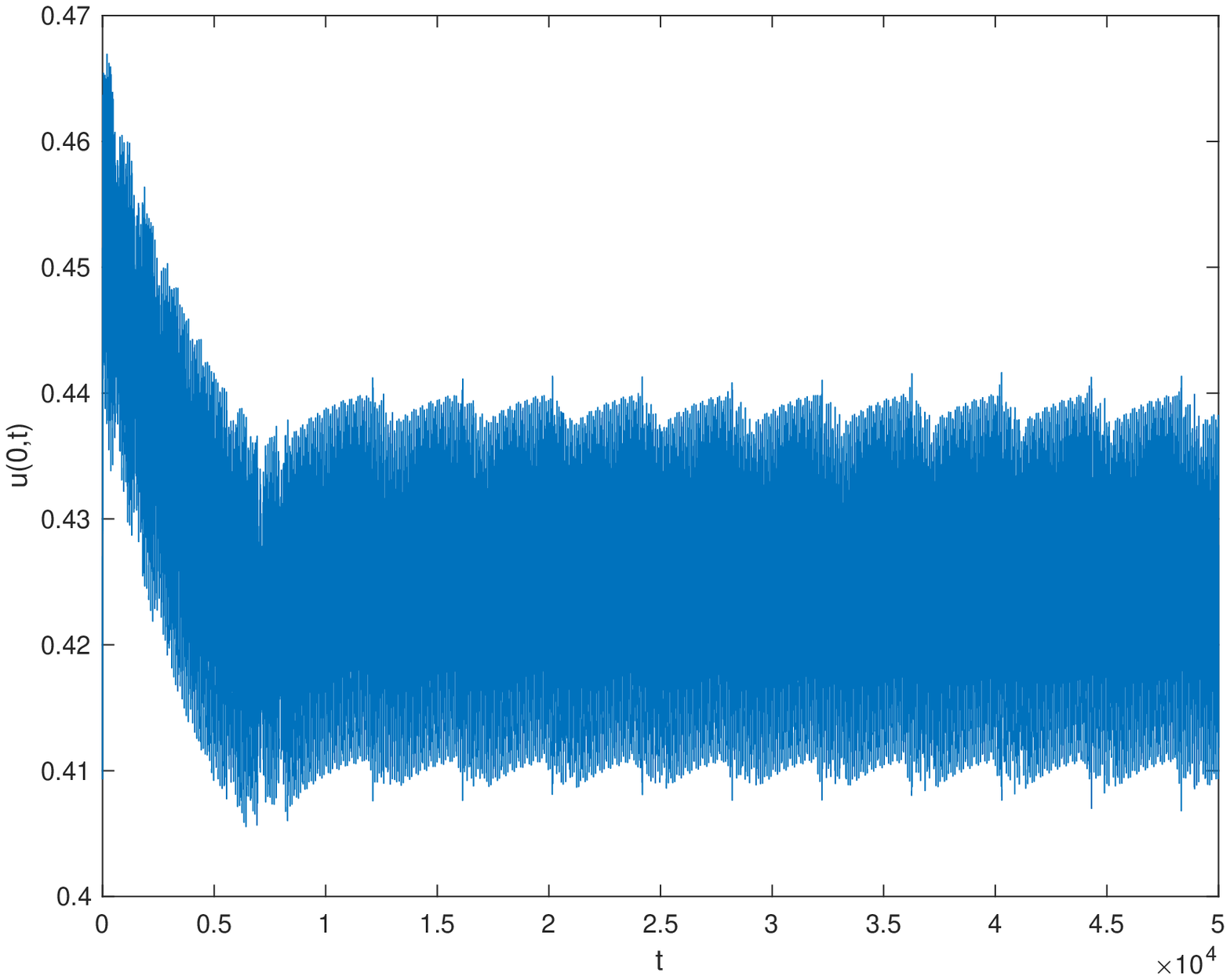}                
\end{minipage}}
	\subfigure[$v(t,0)$]{                  
		\begin{minipage}{0.48\linewidth} 
			\centering                                     
			\includegraphics[width = 1.\textwidth,height = 0.4 \textwidth]{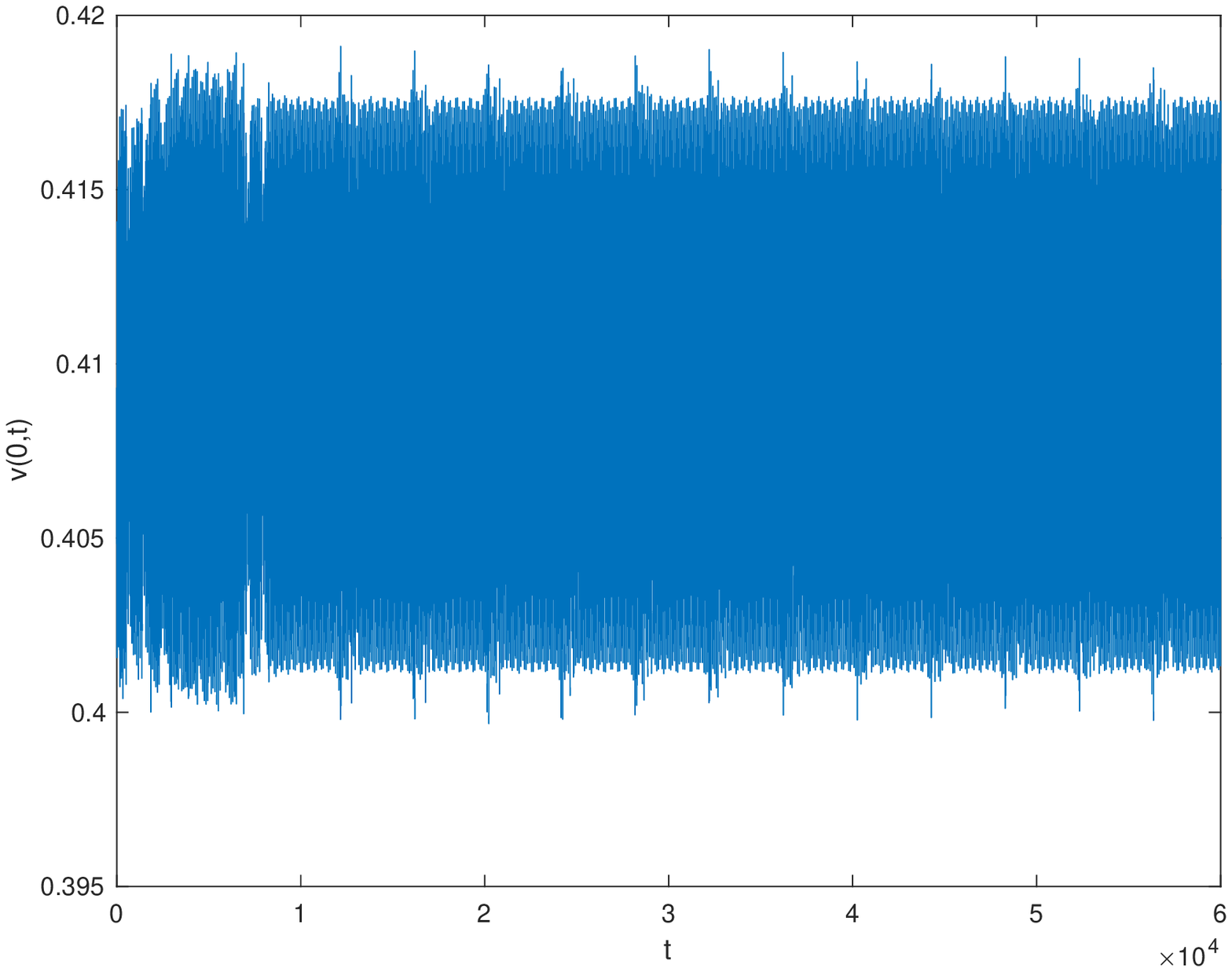}                
	\end{minipage}}
	\caption{A stable spatially non-homogeneous quasi-periodic solution in {$\mathbf D_7$}, with $(\alpha_1,\alpha_2)=(0.0220,0.0082)$ and the initial value functions are $u_0(x)=v_0(x) = u_0+0.05\sin x$.}
	\label{fig7-1}
\end{figure}

\begin{figure}[htbp]
	\centering                           
	\subfigure[$u(t,x)$]{       
		\begin{minipage}{0.48\linewidth} 
			\centering                                      
			\includegraphics[scale=0.33]{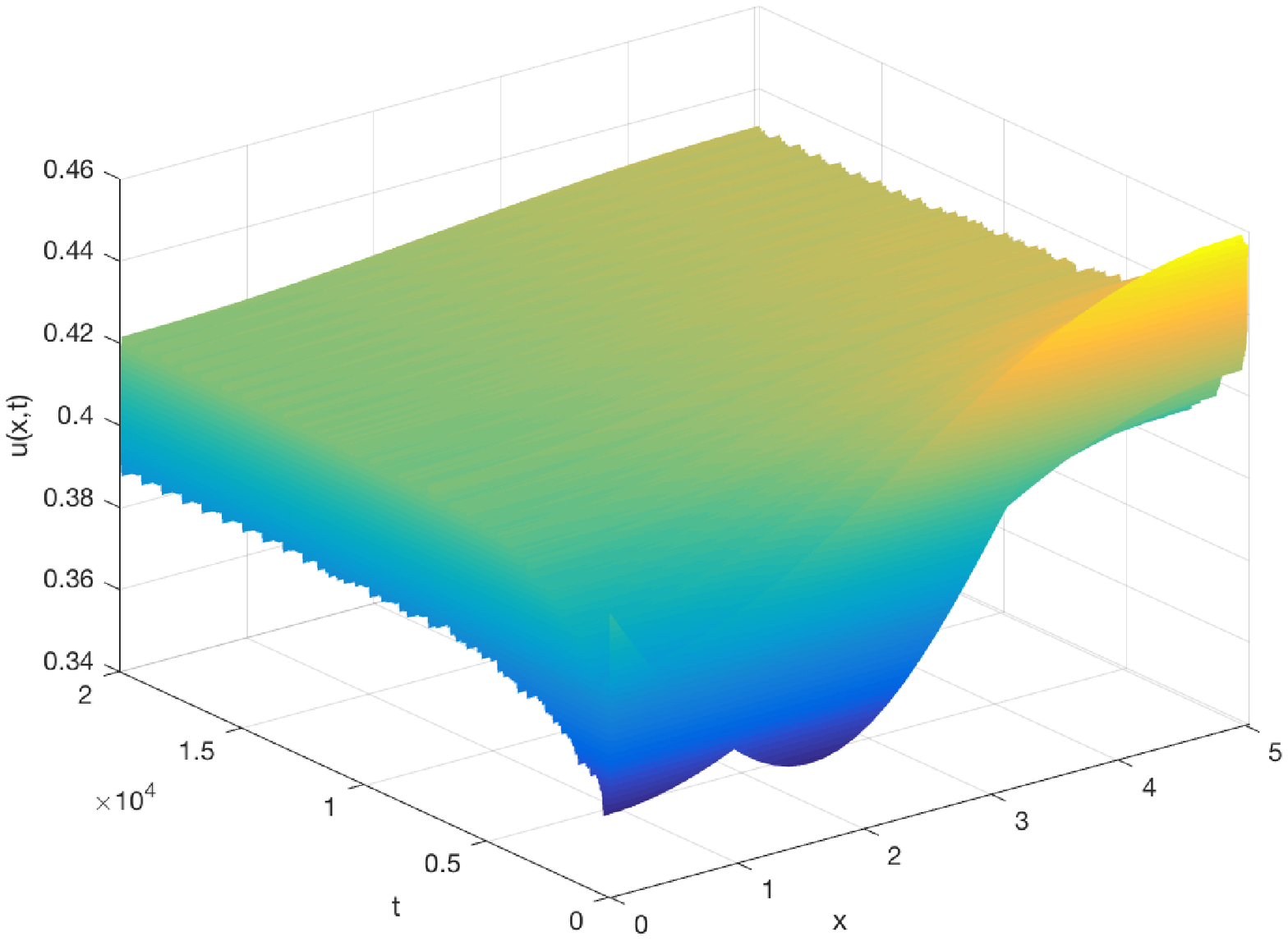}            
	\end{minipage}}
	\subfigure[$v(t,x)$]{       
		\begin{minipage}{0.48\linewidth} 
			\centering                                      
			\includegraphics[scale=0.33]{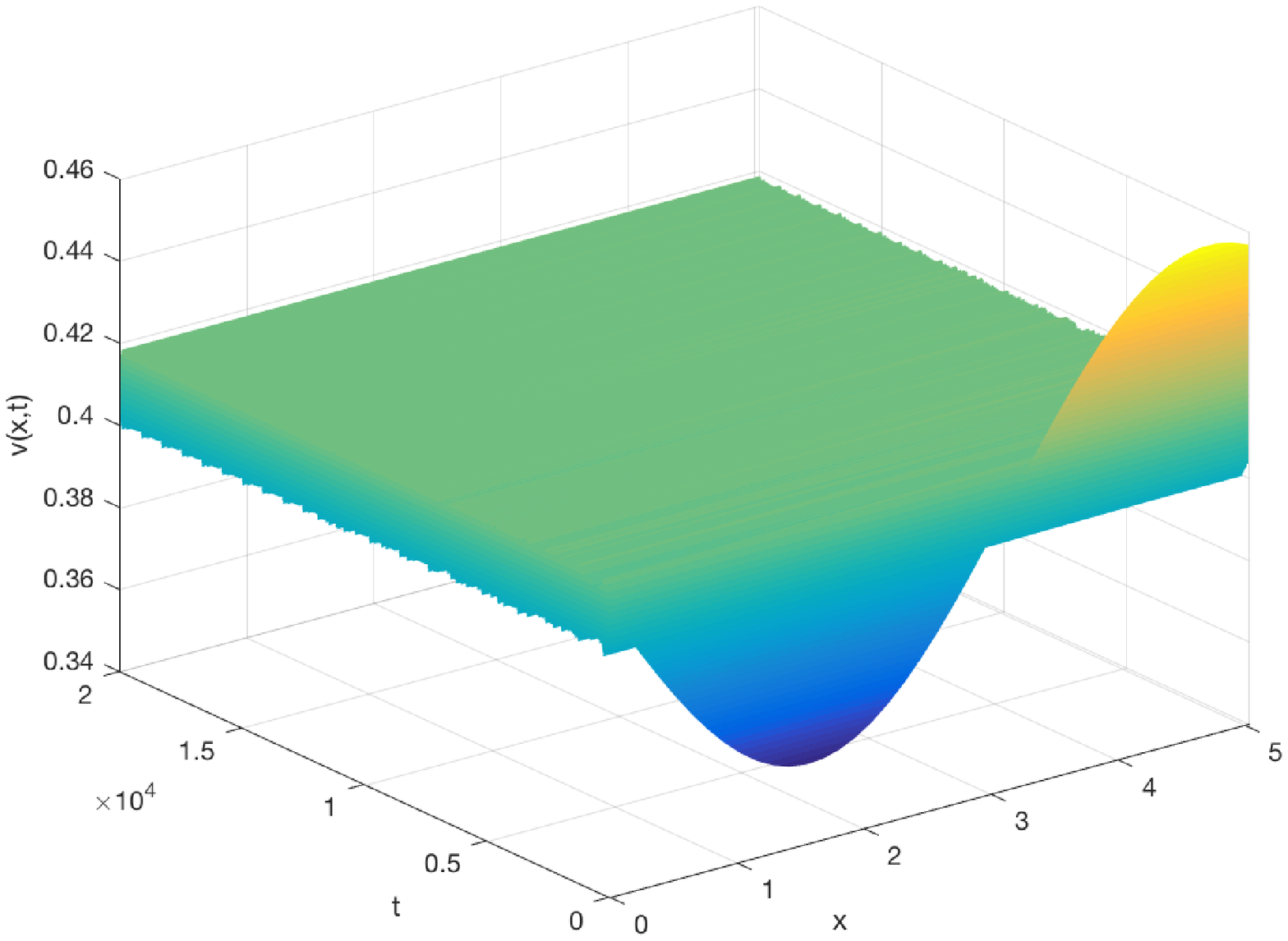}            
	\end{minipage}}
\vskip -0.3cm
	\subfigure[$u(t,0)$]{                  
	\begin{minipage}{0.48\linewidth} 
		\centering                                     
		\includegraphics[width = 1.\textwidth,height = 0.4 \textwidth]{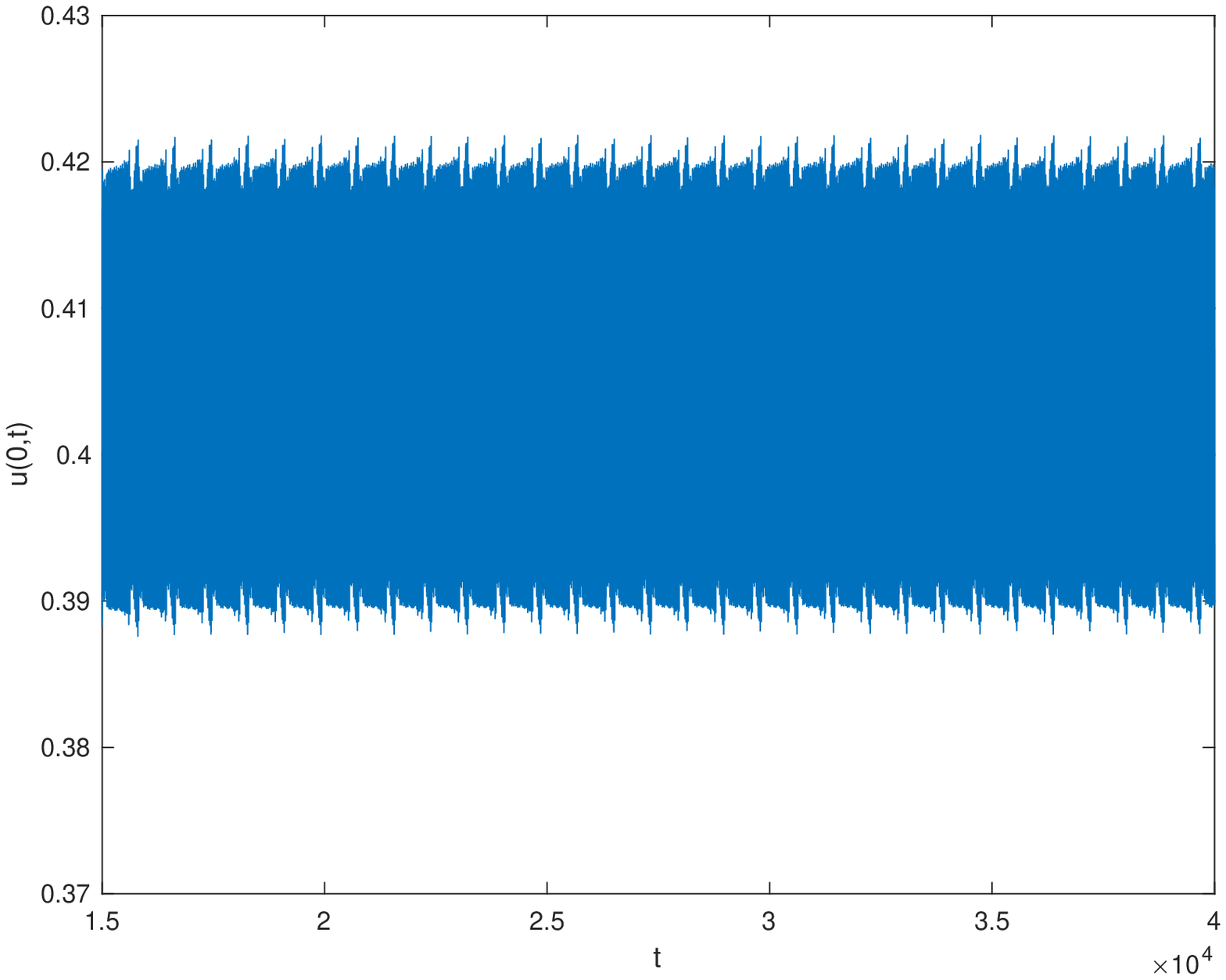}                
	\end{minipage}}
	\subfigure[$v(t,0)$]{                  
		\begin{minipage}{0.48\linewidth} 
			\centering                                     
			\includegraphics[width = 1.\textwidth,height = 0.4 \textwidth]{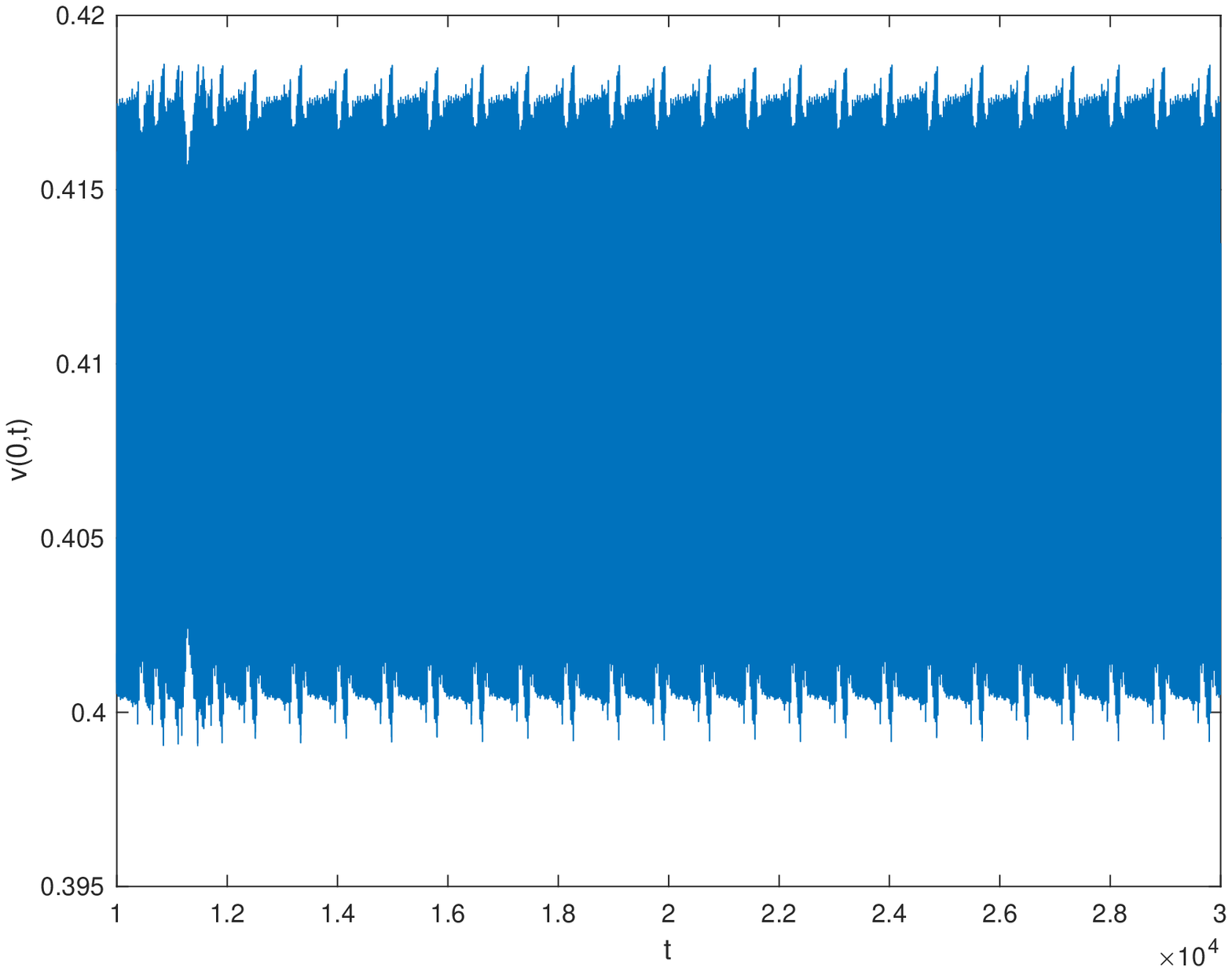}                
	\end{minipage}}
	\caption{A stable spatially non-homogeneous quasi-periodic solution in {$\mathbf D_7$}, with $(\alpha_1,\alpha_2)=(0.0220,0.0082)$ and the initial value functions are $u_0(x)=v_0(x) = u_0-0.05\sin x$..}
	\label{fig7-2}
\end{figure}

  \begin{figure}[htbp]
  	\centering                           
  	\subfigure[$u(t,x)$]{       
  		\begin{minipage}{0.48\linewidth} 
  			\centering                                      
  			\includegraphics[scale=0.33]{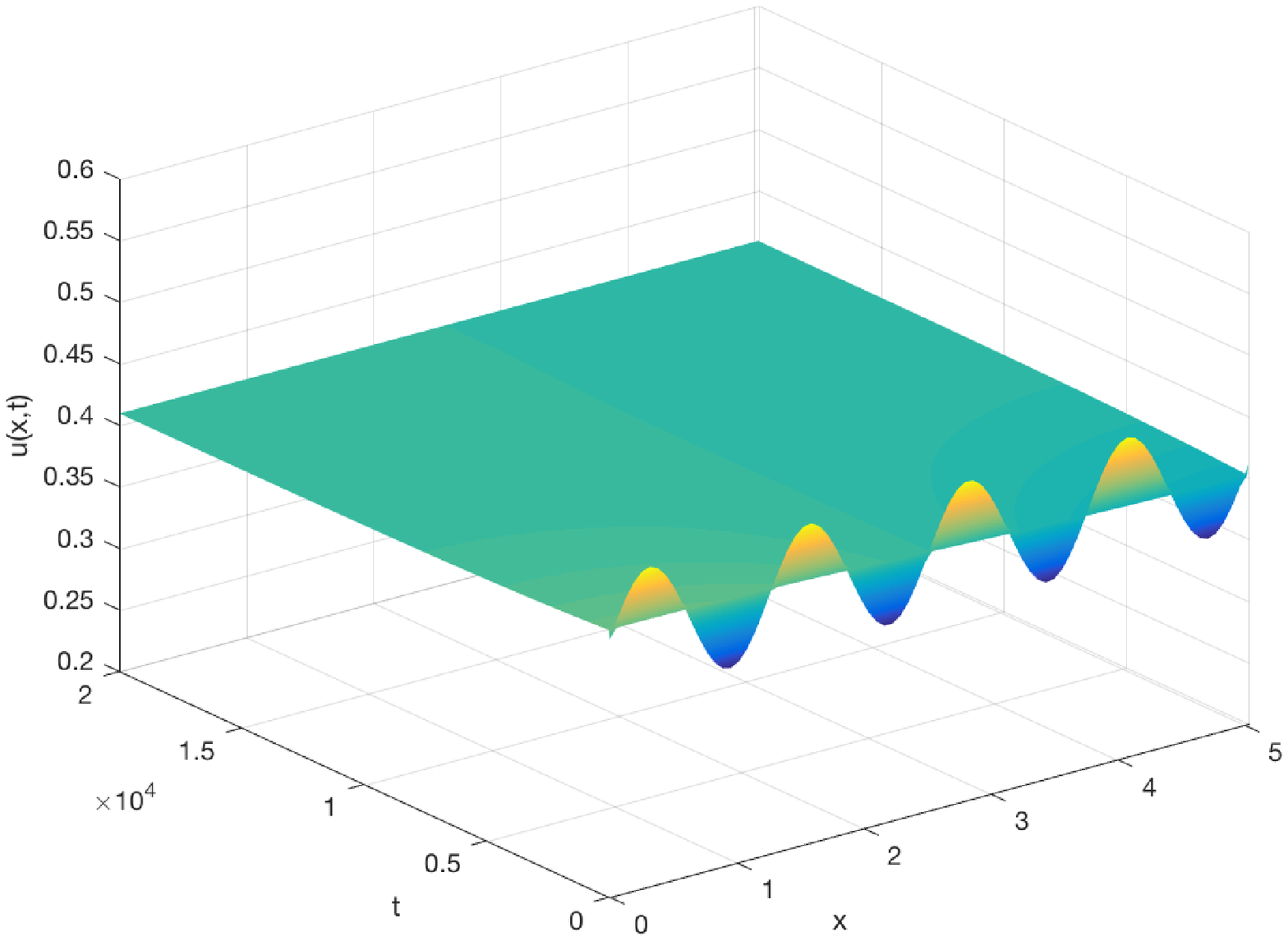}            
  	\end{minipage}}
  	\subfigure[$v(t,x)$]{                  
  		\begin{minipage}{0.48\linewidth} 
  			\centering                                     
  			\includegraphics[scale=0.33]{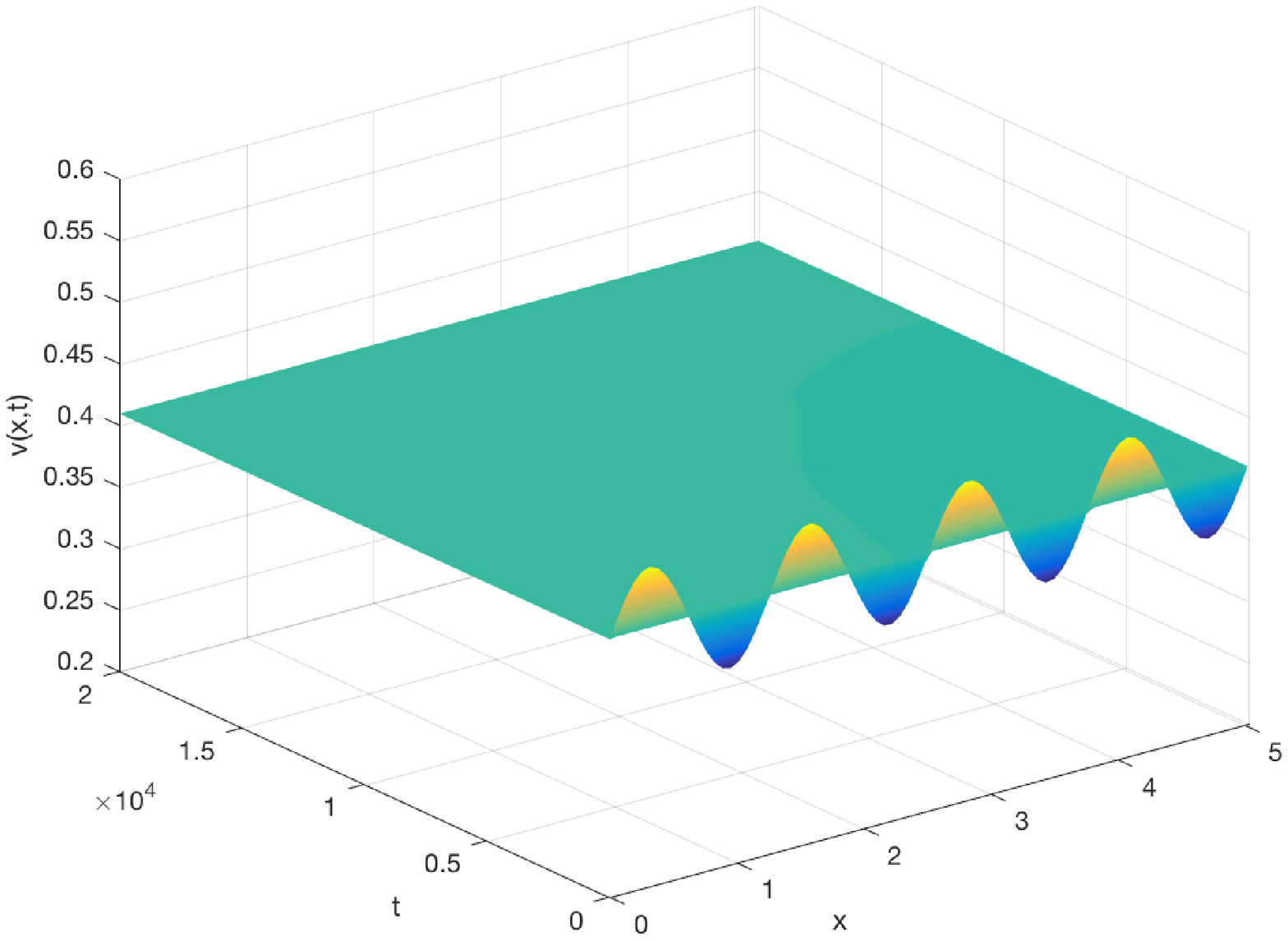}                
  		\end{minipage}}  	
  	\caption{A stable constant steady state $(u_0,v_0)$ in {$\mathbf D_7$},  with $(\alpha_1,\alpha_2)=(0.0220,0.0082)$ and the initial value functions are $u_0(x)=v_0(x) = u_0+0.05\sin 6x$.} 
  	\label{fig7-3}
  \end{figure}

\section{Conclusion}
A comprehensive investigation of the bifurcations of the modified Holling-Tanner systems at the positive equilibrium $(u_0,v_0)$ is given, and the spatio-temporal patterns induced by Turing-Hopf bifurcation are identified. The parameter ranges of the existence of multiple bifurcations are demonstrated.

All the parameters in \eqref{eqA} can be divided into three parts: the diffusion coefficients $(d_1,d_2)$, the auxiliary parameter $(a,b)$ and the main parameters $(r,l)$. When $a\leq \frac{(b+1)^2}{2(1-b)}$, the predator-prey system will eventually tend to balance in both time and space. When $a> \frac{(b+1)^2}{2(1-b)}$, the diffusion coefficients $d_2>d_1$ is a necessary condition for the system to form the spatial inhomogeneous patterns. That means, if the predator moves faster than prey, then the non-uniformly distribution of these two species  in space  are more likely to occur. 
Moreover, the large birth ratio $r=r_2/r_1$ of predator to prey is beneficial to the stability of the Holling-Tanner system, and the small spatial domains $l$ is not possible for the formation of the spatial patterns has been shown in our results.

The study of the synergistic effects of the two parameters $(r,l)$ on the Holling-Tanner system indicated that, the large space regions provide the possibility for the existence of more kinds of bifurcations and various spatio-temporal patterns. Among these possible bifurcations types, Turing-Hopf bifurcation is be mainly studied in this work and a wealth of self-organized spatio-temporal patterns generated of the Holling-Tanner system. 
It is worth mentioning that, the stable spatially non-homogeneous periodic or quasi-periodic solution can not be bifurcated by a simple Hopf bifurcation or steady state bifurcation in the reaction diffusion system subject to homogeneous Neumann boundary condition.

Compare the illustrations in Figure \ref{fig6-1} and Figure \ref{fig7-1}, we observed that the time-period of the spatially non-homogeneous quasi-periodic solution becomes large when the parameters is far away from $L_5$. 
%Further analysis the eventually state of such solutions with the parameters continue to move away $L_5$ and close to $L_7$ is remains a challenging problem. 
%It can be expected to bifurcate or break up. 
But the eventually state of such solutions with the parameters continue to move away $L_5$ and close to $L_7$ is almost nothing to know. We conjecture that these solutions will break up due to the occurrence of some bifurcation. In order to verify it, higher order normal form and some global analysis methods are required.

%\newpage
\bibliographystyle{plain}
\bibliography{mybibfile}
\end{document}